\documentclass[12pt,reqno]{amsart}
\usepackage{amsfonts,amsrefs,latexsym,amsmath, amssymb, mathrsfs, verbatim}
\usepackage{url,color}
\usepackage{pifont}
\usepackage{upgreek}
\usepackage{fancyhdr}
\usepackage{hyperref}
\usepackage{calligra}
\usepackage{marvosym}
\usepackage[percent]{overpic}
\usepackage{pict2e}
\usepackage[hypcap=false]{caption}
\usepackage{xcolor}
\usepackage{wasysym}
\usepackage{accents}

\topmargin - .23 in

\newtheorem{theorem}{Theorem}[section]
\newtheorem{proposition}{Proposition}[section]
\newtheorem{lemma}[proposition]{Lemma}

\theoremstyle{definition}
\newtheorem{definition}{Definition}[section]
\newtheorem{remark}{Remark}[section]
\newtheorem{notation}{Notation}[section]

\theoremstyle{plain}

\DeclareMathAlphabet{\mathcalligra}{T1}{calligra}{m}{n}
\DeclareFontShape{T1}{calligra}{m}{n}{<->s*[2.2]callig15}{}

\newcommand{\Trandatasize}[1]{\mathring{\updelta}}

\newcommand{\TrandatasizeWithFactor}{\mathring{\updelta}_*}

\newcommand{\gsphere}{g \mkern-8.5mu / }

\newcommand{\mytr}{{\mbox{\upshape{tr}}_{\mkern-2mu \gsphere}}}

\newcommand{\Flatdiv}{\mbox{\upshape div}\mkern 1mu}
\newcommand{\Flatcurl}{\mbox{\upshape curl}\mkern 1mu}

\newcommand{\D}{\mathscr{D}}

\newcommand{\angLap}{ {\Delta \mkern-12mu / \, } }

\newcommand{\Speed}{c}

\newcommand{\Transport}{B}

\newcommand{\Ent}{s}
\newcommand{\GradEnt}{S}

\newcommand{\Densrenormalized}{\uprho}

\newcommand{\Vortrenormalized}{\varOmega}
\newcommand{\CurlofVortrenormalized}{\mathcal{C}}

\newcommand{\DivofEntrenormalized}{\mathcal{D}}

\newcommand{\Fullset}{\mathscr{Z}}
\newcommand{\Tanset}{\mathscr{P}}

\newcommand{\Singletan}{P}

\newcommand{\Lunit}{L}
\newcommand{\uLunit}{\underline{L}}

\newcommand{\Rad}{\breve{X}}

\newcommand{\CoordAng}{\Theta}

\newcommand{\Mult}{T}
\newcommand{\GeoAng}{Y}

\newcommand{\smoothfunction}{\mathrm{f}}

\numberwithin{equation}{subsection}

\begin{document}
\title{A New Formulation of the $3D$ Compressible Euler Equations With Dynamic Entropy: Remarkable Null Structures and Regularity Properties}
\author{Jared Speck$^{* \dagger}$}

\thanks{$^{\dagger}$JS gratefully acknowledges support from NSF grant \# DMS-1162211,
from NSF CAREER grant \# DMS-1454419,
from a Sloan Research Fellowship provided by the Alfred P. Sloan foundation,
and from a Solomon Buchsbaum grant administered by the Massachusetts Institute of Technology.
}

\thanks{$^{*}$Massachusetts Institute of Technology, Cambridge, MA, USA.
\texttt{jspeck@math.mit.edu}}

\begin{abstract}
We derive a new formulation of the $3D$ compressible Euler equations with dynamic entropy
exhibiting remarkable null structures and regularity properties.
Our results hold for an arbitrary equation of state 
(which yields the pressure in terms of the density and the entropy)
in non-vacuum regions where the speed of sound is positive.
Our work is an extension of our prior joint work with J.~Luk, in which we derived a similar new
formulation in the special case of a barotropic fluid, 
that is, when the equation of state depends only on the density. 
The new formulation comprises covariant wave equations for 
the Cartesian components of the velocity and the logarithmic density
coupled to a transport equation for the specific vorticity
(defined to be vorticity divided by density),
a transport equation for the entropy,
and some additional transport-divergence-curl-type equations involving
special combinations of the derivatives of the solution variables.
The good geometric structures in the equations
allow one to use the full power of the vectorfield method
in treating the ``wave part'' of the system.
In a forthcoming application, we will use the new formulation 
to give a sharp, constructive proof of finite-time
shock formation, tied to the intersection of acoustic ``wave''
characteristics, for solutions with nontrivial vorticity and entropy
at the singularity. In the present article, we derive the new formulation
and overview the central role that it plays in the proof of shock formation.

Although the equations are significantly more complicated than they are
in the barotropic case, they enjoy many of same remarkable features including
\textbf{i)} all derivative-quadratic inhomogeneous terms are null
forms relative to the acoustical metric, which is the Lorentzian metric
driving the propagation of sound waves and
\textbf{ii)} the transport-divergence-curl-type equations allow one
to show that the entropy is one degree more differentiable than the velocity
and that the vorticity is exactly as differentiable as the velocity,
which represents a gain of one derivative compared to standard estimates.
This gain of a derivative, which seems to be new
for the entropy, is essential for closing the energy estimates
in our forthcoming proof of shock formation
since the second derivatives of the entropy 
and the first derivatives of the vorticity
appear as inhomogeneous terms in the 
wave equations.



\bigskip

\bigskip

\end{abstract}

\maketitle

\centerline{\today}

\tableofcontents
\setcounter{tocdepth}{2}

\section{Introduction and summary of main results}
\label{S:INTRO}
Our main result in this article is Theorem~\ref{T:GEOMETRICWAVETRANSPORTSYSTEM}, 
in which we provide a new formulation of the compressible Euler equations with dynamic entropy
that exhibits astoundingly good null structures and regularity properties.
We consider only the physically relevant case of three spatial dimensions, 
though similar results hold in any number of spatial dimensions.
Our results hold for an arbitrary equation of state
in non-vacuum regions where the speed of sound is positive.
By equation of state, we mean the function yielding
the pressure in terms of the density and the entropy.
Our results are an extension of our previous joint work 
with J.~Luk \cite{jLjS2016a}, 
in which we derived a similar new formulation of the
equations in the special case of a barotropic fluid, 
that is, when the equation of state depends only on the density.
Our work \cite{jLjS2016a} was in turn inspired by Christodoulou's remarkable
proofs \cites{dC2007,dCsM2014} of shock formation for small-data 
solutions to the compressible Euler equations
in irrotational (that is, vorticity-free) regions
as well as our prior work \cite{jS2016b} on shock formation for general classes of wave equations; 
we describe these works in more detail below.

A principal application of the new formulation 
is that it serves as the starting
point for our forthcoming work, in which we plan to give a sharp proof of finite-time
shock formation for an open set of initial conditions
without making any symmetry assumptions, 
irrotationality assumption,
or barotropic equation of state assumption.
The forthcoming work will be an extension of our recent work
with J.~Luk \cite{jLjS2016b}, in which we proved a similar shock formation result 
for barotropic fluids in the case of two spatial dimensions.

Our new formulation of the compressible Euler equations
comprises covariant wave equations, transport equations,
and transport-divergence-curl-type equations involving \emph{special combinations} 
of solution variables (see Def.~\ref{D:RENORMALIZEDCURLOFSPECIFICVORTICITY}). 
As we mentioned earlier, the inhomogeneous terms exhibit
good null structures, which we characterize in our second main result,
Theorem~\ref{T:STRONGNULL}. Its proof is quite simple given Theorem~\ref{T:GEOMETRICWAVETRANSPORTSYSTEM}.
As we mentioned above, in \cite{jLjS2016a}, 
we derived a similar new formulation of the equations
under the assumption that the fluid is barotropic.
The barotropic assumption, though often made in 
astrophysics, cosmology, and meteorology, is generally unjustified 
because it entails neglecting thermal dynamics and their effect on the fluid.
Compressible fluid models that are more physically
realistic feature equations of state that depend on the density 
\emph{and} a second thermodynamic state-space variable, such as the temperature, 
which satisfies an evolution equation that is coupled to the other fluid equations.
In the present article, 
we allow for an arbitrary physical equation of state in which,
for mathematical convenience, we have chosen\footnote{For sufficiently regular solutions, 
there are many equivalently formulations of the compressible Euler equations, 
depending on the state-space variables that one chooses as unknowns in the system.} 
the second thermodynamic variable to be the entropy per unit mass 
(which we refer to as simply the ``entropy'' from now on).

\subsection{Paper outline} 
\label{SS:OUTLINE}
In the remainder of Sect.~\ref{S:INTRO},
we provide some standard background material
on the compressible Euler equations,
define the solution variables 
that we use
in formulating our main results,
roughly summarize our main results,
and provide some preliminary context.
In Sect.~\ref{S:LORENTZIANGEOMETRYBASIC},
we define some geometric objects that we use in formulating
our main results and
provide some basic background on Lorentzian geometry
and null forms.
In Sect.~\ref{S:MAINRESULTS},
we give precise statements of our main results, 
namely Theorems~\ref{T:GEOMETRICWAVETRANSPORTSYSTEM} and \ref{T:STRONGNULL},
and give the simple proof of the latter.
In Sect.~\ref{S:OVERVIEWOFROLESOFTHEOREMS},
we overview our forthcoming proof of shock formation,
highlighting the roles that
Theorems~\ref{T:GEOMETRICWAVETRANSPORTSYSTEM} and \ref{T:STRONGNULL}
will play. In Sect.~\ref{S:PROOFOFMAINTHEOREM}, 
we prove Theorem~\ref{T:GEOMETRICWAVETRANSPORTSYSTEM}
via a series of calculations in which
we observe many important cancellations.

\subsection{Notation}
\label{SS:NOTATION}
Throughout $\lbrace x^{\alpha} \rbrace_{\alpha=0,1,2,3}$ denotes
the usual Cartesian coordinate system
on\footnote{In our forthcoming proof of shock formation, we will, for convenience,
consider spacetimes with topology $\mathbb{R} \times \Sigma$, 
where $\Sigma := \mathbb{R} \times \mathbb{T}^2$ is the space manifold; 
see Sect.~\ref{S:OVERVIEWOFROLESOFTHEOREMS}
for an overview. In that context, 
$\lbrace x^{\alpha} \rbrace_{\alpha=0,1,2,3}$ denotes
the usual Cartesian coordinate system on $\mathbb{R} \times \Sigma$,
where $x^0 \in \mathbb{R}$ is the time coordinate,
$x^1$ is a standard spatial coordinate on $\mathbb{R}$,
and $x^2$ and $x^3$ are standard (locally defined) coordinates
on $\mathbb{T}^2$. 
Note that the vectorfields
$
\displaystyle
\partial_{\alpha} 
:=
\frac{\partial}{\partial x^{\alpha}}
$
on $\mathbb{T}^2$
can be extended so as to be globally defined and smooth. \label{FN:CARTESIANCOORDNATESINT2CASE}}
$\mathbb{R}^{1+3} \simeq \mathbb{R} \times \mathbb{R}^3$.
More precisely, $x^0 \in \mathbb{R}$ is the time coordinate and $(x^1,x^2,x^3) \in \mathbb{R}^3$
are spatial coordinates.
$
\displaystyle
\partial_{\alpha} 
:=
\frac{\partial}{\partial x^{\alpha}}
$
denotes the corresponding Cartesian coordinate partial derivative vectorfields.
We often use the alternate notation $x^0 = t$ and $\partial_0 = \partial_t$.
Greek ``spacetime'' indices such as $\alpha$ vary over $0,1,2,3$ and
Latin ``spatial'' indices such as $a$ vary over $1,2,3$.
We use Einstein's summation convention in that repeated indices are summed
over their respective ranges.
$\Sigma_t$ denotes the usual flat hypersurface of constant Cartesian time $t$.

\subsection{Basic background on the compressible Euler equations}
\label{SS:BACKGROUND}
In this subsection, we provide some basic background on the compressible Euler equations.

\subsubsection{Equations of state}
\label{SSS:EOS}
We study the compressible Euler equations
for a perfect fluid in three spatial dimensions
under any equation of state with positive sound speed (see definition \eqref{E:SOUNDSPEED}). 
The equation of state is 
the function (which we assume to be given) 
that determines the pressure $p$ in terms 
of the density $\varrho \geq 0$ and the entropy $\Ent \in \mathbb{R}$:
\begin{align} \label{E:EOS}
	p = p(\varrho,\Ent).
\end{align}
Given the equation of state,
the compressible Euler equations can be formulated as 
evolution equations for the velocity 
$v:\mathbb{R}^{1+3} \rightarrow \mathbb{R}^3$,
the density $\varrho:\mathbb{R}^{1+3} \rightarrow [0,\infty)$,
and the entropy
$\Ent:\mathbb{R}^{1+3} \rightarrow (-\infty,\infty)$.

\subsubsection{Some definitions}
\label{SSS:SOMEDEFS}
We use the following notation\footnote{See Subsect.~\ref{SS:NOTATION}
regarding our conventions for indices and implied summation.} 
for the Euclidean 
divergence and curl of a $\Sigma_t-$tangent vectorfield $V$:
\begin{align} \label{E:FLATDIVANDCURL}
		\Flatdiv V
		& := \partial_a V^a,
			\qquad
		(\Flatcurl V)^i
		:= \epsilon_{iab} \partial_a V^b.
	\end{align}
	In \eqref{E:FLATDIVANDCURL}, $\epsilon_{ijk}$ is the fully antisymmetric
	symbol normalized by
	\begin{align} \label{E:EPSILONNORMALIZATION}
	\epsilon_{123} = 1.
\end{align}
The vorticity $\omega: \mathbb{R}^{1+3} \rightarrow \mathbb{R}^3$ 
is the vectorfield
\begin{align} \label{E:VORTICITYDEFINITION}
	\omega^i 
	& := (\Flatcurl v)^i.
\end{align}

Rather than formulating the equations in terms of the
density and the vorticity,
we find it convenient to use the
\emph{logarithmic density} $\Densrenormalized$
and the \emph{specific vorticity} $\Vortrenormalized$;
some of the equations that we study take a simpler
form when expressed in terms of these variables.

To define these quantities, we first fix a constant background density $\bar{\varrho}$ such that
\begin{align}  \label{E:BACKGROUNDDENSITY}
\bar{\varrho} > 0.
\end{align}
In applications, one may choose any
convenient value\footnote{For example, when studying solutions
that are perturbations of non-vacuum constant states,
one may choose $\bar{\varrho}$ so that 
in terms of the variable $\Densrenormalized$ from \eqref{E:RESCALEDVARIABLES},
the constant state corresponds to $\Densrenormalized \equiv 0$.} 
of $\bar{\varrho}$.

\begin{definition}[\textbf{Logarithmic density and specific vorticity}]
\label{D:RESCALEDVARIABLES}
\begin{align} \label{E:RESCALEDVARIABLES}
	\Densrenormalized
	& := \ln \left(\frac{\varrho}{\bar{\varrho}} \right),
		\qquad
	\Vortrenormalized
	:= \frac{\omega}{(\varrho/\bar{\varrho})}
	= \frac{\omega}{\exp \Densrenormalized}.
\end{align}
\end{definition}
We assume throughout that\footnote{We avoid discussing fluid dynamics 
in regions with vanishing density. The reason is that
the compressible Euler equations become degenerate along fluid-vacuum boundaries 
and not much is known about compressible fluid flow in this context; see, for example \cite{dCsS2012}
for more information.}
\begin{align} \label{E:DENSITYPOSITIVE}
	\varrho > 0.
\end{align}
In particular, the variable $\Densrenormalized$
is finite assuming \eqref{E:DENSITYPOSITIVE}.

In the study of shock formation,
to obtain sufficient top-order regularity for the
entropy, it is important to work with 
the $\Sigma_t$-tangent vectorfield
$\GradEnt$ provided by the next definition;
see Remark~\ref{R:NEEDFORGRADENTANDDIVCURL} for further discussion.

\begin{definition}[\textbf{Entropy gradient vectorfield}]
	\label{D:ENTROPYGRADIENT}
	We define the Cartesian components of the $\Sigma_t$-tangent
	\emph{entropy gradient vectorfield} $\GradEnt$
	as follows, $(i=1,2,3)$:
	\begin{align} \label{E:ENTROPYGRADIENT}
		\GradEnt^i
		& := \delta^{ia} \partial_a \Ent
			= \partial_i \Ent.
	\end{align}
\end{definition}

\begin{remark}[\textbf{The need for $\GradEnt$ and transport-$\Flatdiv$-$\Flatcurl$ estimates in controlling} $\Ent$]
	\label{R:NEEDFORGRADENTANDDIVCURL}
	In our forthcoming proof of shock formation,
	we will control the top-order derivatives of 
	$\Ent$ by combining estimates for transport equations
	with $\Flatdiv$-$\Flatcurl$-type elliptic estimates for $\GradEnt$
	and its higher derivatives. 
	At first glance, it might seem like the $\Flatdiv$-$\Flatcurl$ elliptic estimates
	could be replaced with simpler elliptic estimates for $\Delta \Ent$, in view of
	the simple identity $\Delta \Ent = \Flatdiv \GradEnt$.
	Although this is true for $\Delta \Ent$ itself,
	in our proof of shock formation, 
	the Euclidean Laplacian $\Delta$ is not compatible with the differential operators
	that we must use to commute the equations
	when obtaining estimates for the solution's higher derivatives.
	Specifically, like all prior works on shock formation in
	more than one spatial dimension, 
	our forthcoming proof is based on commuting 
	the equations with geometric vectorfields
	(see Subsect.~\ref{SS:SUMMARYOFPROOF} for an overview) 
	that are adapted to the acoustic 
	wave characteristics of the compressible Euler equations,\footnote{We define these wave characteristics,
	denoted by $\mathcal{P}_u$, in Subsect.~\ref{SS:SHOCKFORMATIONPROOFSUMMARY}.}
	which have essentially nothing to with the operator $\Delta$.
	Therefore, the geometric vectorfields exhibit very poor commutation properties with $\Delta$
	and in fact would generate uncontrollable error terms if commuted with it.
	In contrast, in carrying out our transport-divergence-curl-type estimates, 
	\emph{we only have to commute the geometric vectorfields through first-order operators}
	including the transport operator, $\Flatdiv$, and $\Flatcurl$;
	it turns out that commuting the geometric vectorfields 
	through first-order operators, as long as they are weighted with an appropriate geometric weight,\footnote{Specifically,
	the weight is the inverse foliation density $\upmu$ of the acoustic characteristics; see Def.~\ref{D:UPMUDEF}.} 
	leads to controllable error terms, compatible with following the solution all the way to the singularity.
	We explain this issue in more detail in Steps (1) and (2) of Subsect.~\ref{SS:SUMMARYOFPROOF}.
\end{remark}

\begin{notation}[\textbf{State-space variable differentiation via semicolons}]
	\label{N:STATESPACEDIFFERENTIATION}
	If $f = (\Densrenormalized,\Ent)$ is a scalar function, then
	we use the following notation to denote partial differentiation with respect to
	$\Densrenormalized$ and $\Ent$:
	$
	\displaystyle
	f_{;\Densrenormalized} 
	:= \frac{\partial f}{\partial \Densrenormalized}
	$
	and
	$
	\displaystyle
	f_{;\Ent} 
	:= \frac{\partial f}{\partial \Ent}
	$.
	Moreover, 
	$
	\displaystyle
	f_{;\Densrenormalized;\Ent} 
	:= \frac{\partial^2 f}{\partial \Ent \partial \Densrenormalized}
	$,
	and we use similar notation for other higher partial derivatives of $f$
	with respect to $\Densrenormalized$ and $\Ent$.
\end{notation}

\subsubsection{Speed of sound}
\label{SSS:SPEEDOFSOUND}
The scalar function $\Speed \geq 0$ defined by\footnote{On RHS \eqref{E:SOUNDSPEED},
$
\displaystyle
\frac{\partial p}{\partial \varrho}\left|\right._{\Ent}
$
denotes the derivative of the equation of state with respect to the (non-logarithmic) density $\varrho$
at fixed entropy.}
\begin{align} \label{E:SOUNDSPEED}
	\Speed 
	& := 
	\sqrt{\frac{\partial p}{\partial \varrho}\left|\right._{\Ent}}
		= \sqrt{\frac{1}{\bar{\varrho}}\exp(-\Densrenormalized) p_{;\Densrenormalized}}
\end{align}
is a fundamental quantity known as the \emph{speed of sound}.
To obtain the last equality in \eqref{E:SOUNDSPEED},
we used the chain rule identity
$
\displaystyle
\frac{\partial}{\partial \varrho} \left|\right._{\Ent}
=
\frac{1}{\bar{\varrho}}
\exp(-\Densrenormalized)
\frac{\partial}{\partial \Densrenormalized} \left|\right._{\Ent}
$.
From now on, we view
\begin{align} \label{E:FUNCNSOFDENSANDENT}
	\Speed 
	& = \Speed(\Densrenormalized,\Ent).
\end{align}

\begin{center}
\underline{\textbf{\large Assumptions on the equation of state}}
\end{center}
We make the following physical assumptions,
which ensure the hyperbolicity of the system when
$\varrho > 0$:
\begin{itemize}
	\item $\Speed \geq 0$.
	\item $\Speed > 0$ when $\varrho > 0$. Equivalently, $\Speed > 0$ whenever $\Densrenormalized \in (-\infty,\infty)$.
\end{itemize}

\subsubsection{A standard first-order formulation of the compressible Euler equations}
\label{SSS:FIRSTORDERFORMULATIONOFEULER}
We now state a standard first-order formulation of the compressible Euler equations;
these equations are the starting point of our new formulation.
Specifically, relative to Cartesian coordinates, the compressible Euler equations
can be expressed\footnote{Throughout, if $V$ is a vectorfield
and $f$ is a function, then $Vf := V^{\alpha} \partial_{\alpha} f$
denotes the derivative of $f$ in the direction $V$.
\label{F:VECTORFIELDSACTONFUNCTIONS}} 
as follows:
\begin{subequations}
\begin{align} \label{E:TRANSPORTDENSRENORMALIZEDRELATIVETORECTANGULAR}
	\Transport \Densrenormalized
	& = - \Flatdiv v,
		\\
	\Transport v^i 
	& = 
	- \Speed^2 \delta^{ia} \partial_a \Densrenormalized
	- \exp(-\Densrenormalized) \frac{p_{;\Ent}}{\bar{\varrho}} \delta^{ia} \partial_a \Ent,
	\label{E:TRANSPORTVELOCITYRELATIVETORECTANGULAR}
		\\
	\Transport \Ent
	& = 0.
	\label{E:ENTROPYTRANSPORT}
\end{align}
\end{subequations}
Above and throughout, $\delta^{ab}$ is the standard Kronecker delta
and
\begin{align} \label{E:MATERIALVECTORVIELDRELATIVETORECTANGULAR}
	\Transport 
	& := 
	\partial_t 
	+ 
	v^a \partial_a
\end{align}
is the material derivative vectorfield.
We note that $\Transport$ plays a critical role in the ensuing discussion.
Readers may consult, for example, \cite{dCsM2014} for discussion behind the physics of the equations
and for a first-order formulation of them in terms of $\varrho$, $\lbrace v^i \rbrace_{i=1,2,3}$, and $\Ent$,
which can easily seen to be equivalent to 
\eqref{E:TRANSPORTDENSRENORMALIZEDRELATIVETORECTANGULAR}-\eqref{E:ENTROPYTRANSPORT}.

\subsubsection{Modified fluid variables}
\label{SSS:MODIFIEDFLUIDVARIABLES}
Although it is not obvious, the quantities provided in the following definition
satisfy transport equations with a good structure;
see \eqref{E:EVOLUTIONEQUATIONFLATCURLRENORMALIZEDVORTICITY}
and \eqref{E:TRANSPORTFLATDIVGRADENT}.
When combined with elliptic estimates, 
the transport equations allow one to prove that 
the specific vorticity and entropy are one degree more
differentiable than naive estimates would yield.
This gain of regularity is essential 
in our forthcoming proof of shock formation since it is needed to control 
some of the source terms in the wave equations for the velocity and density,
specifically, the first products on RHSs \eqref{E:VELOCITYWAVEEQUATION}-\eqref{E:RENORMALIZEDDENSITYWAVEEQUATION}.
In addition, the source terms in the transport equations 
have a good null structure, which is also 
essential in the study of shock formation.
We discuss these issues in more detail in Sect.~\ref{S:OVERVIEWOFROLESOFTHEOREMS}.

\begin{definition}[\textbf{Modified fluid variables}]
	\label{D:RENORMALIZEDCURLOFSPECIFICVORTICITY}
	We define the Cartesian components of the $\Sigma_t$-tangent vectorfield $\CurlofVortrenormalized$ 
	and the scalar function $\DivofEntrenormalized$ as follows,
	($i=1,2,3$):
	\begin{subequations}
	\begin{align} \label{E:RENORMALIZEDCURLOFSPECIFICVORTICITY}
		\CurlofVortrenormalized^i
		& :=
			\exp(-\Densrenormalized) (\Flatcurl \Vortrenormalized)^i
			+
			\exp(-3\Densrenormalized) \Speed^{-2} \frac{p_{;\Ent}}{\bar{\varrho}} \GradEnt^a \partial_a v^i
			-
			\exp(-3\Densrenormalized) \Speed^{-2} \frac{p_{;\Ent}}{\bar{\varrho}} (\partial_a v^a) \GradEnt^i,
				\\
		\DivofEntrenormalized
		& := 
			\exp(-2 \Densrenormalized) \Flatdiv \GradEnt 
			-
			\exp(-2 \Densrenormalized) \GradEnt^a \partial_a \Densrenormalized.
			\label{E:RENORMALIZEDDIVOFENTROPY}
	\end{align}
	\end{subequations}
\end{definition}

\subsection{A brief summary of our main results}
\label{SS:BRIEFSUMMARYOFMAINRESULTS}

\begin{changemargin}{.25in}{.25in} 
\textbf{Summary of the main results.}
	The compressible Euler equations
	can be reformulated as a system of covariant wave equations for 
	the Cartesian components $\lbrace v^i \rbrace_{i=1,2,3}$ of the velocity 
	and the logarithmic density $\Densrenormalized$
	coupled to a transport equation for the entropy $\Ent$,
	transport equations for the Cartesian components 
	$\lbrace \GradEnt^i \rbrace_{i=1,2,3}$ of the entropy gradient,
	transport equations for the Cartesian components $\lbrace \Vortrenormalized^i \rbrace_{i=1,2,3}$ 
	of the specific vorticity,
	transport equations for the modified fluid variables of Def.~\ref{D:RENORMALIZEDCURLOFSPECIFICVORTICITY},
	and identities for $\Flatdiv \Vortrenormalized$ and $(\Flatcurl \GradEnt)^i$;
	see Theorem~\ref{T:GEOMETRICWAVETRANSPORTSYSTEM} on pg.~\pageref{T:GEOMETRICWAVETRANSPORTSYSTEM}
	for the equations.
	Moreover, the inhomogeneous terms exhibit remarkable structures, 
	\emph{including good null form structures tied to the acoustical metric $g$}
	(which is the Lorentzian metric corresponding to the propagation of sound waves, see Def.~\ref{D:ACOUSTICALMETRIC}); 
	see Theorem~\ref{T:STRONGNULL} on pg.~\pageref{T:STRONGNULL}
	for the precise statement.
\end{changemargin}

\subsection{Some preliminary context for the main results}
\label{SS:CONTEXT}
In this subsection, we provide some preliminary context for our main results,
with a focus on the special null structures
exhibited by the inhomogeneous terms in our new formulation of the compressible Euler equations
and their relevance for our forthcoming proof of shock formation.
The presence of special null structures in the equations
might seem surprising since they are often
associated with equations that admit global solutions. However, 
as we explain below, 
the good null structures are in fact key to proving that the shock forms.
Several works have contributed to our understanding of the important role
that the null structures play in the proof of shock formation,
including \cites{dC2007,gHsKjSwW2016,jS2016b,jLjS2016a,jLjS2016b}.
Here we review these works and some related ones
and, for the results in more than one spatial dimension, 
we highlight the role that the presence of good geo-analytic 
and null structures played in the proofs.


The famous work of Riemann \cite{bR1860}, in which he invented the Riemann invariants,
yielded the first general proof of shock formation for solutions to 
the compressible Euler equations in one
spatial dimension. More precisely, for such solutions,
the velocity and density remain bounded even though their first-order
Cartesian coordinate partial derivatives blow up in finite time.
The most standard proof of this phenomenon
is elementary and is essentially based on identifying a
Riccati-type blowup mechanism for the solution's first derivatives;
see Subsect.~\ref{SS:SIMPLEPLANEWAVESBLOWUP} for a review of these
ideas in the context of simple plane wave solutions.

In all prior proofs of shock formation in more than one spatial dimension,
there also was a Riccati-type mechanism that drives the blowup.
However, in the analysis, the authors encountered many new kinds of error terms that
are much more complicated than the ones encountered by Riemann.
A key aspect of the proofs was showing that the additional error terms
do not interfere with the Riccati-type blowup mechanism.
This is where the special null structure mentioned above enters into play: 
terms that enjoy the special null structure are weak compared to the Riccati-type terms that drive the singularity,
at least near the shock. In order to explain this in more detail,
we now review some prior works on shock formation in more than one spatial dimension.

Alinhac was the first \cites{sA1999a,sA1999b,sA2001b,sA2002}
to prove shock formation results for quasilinear hyperbolic PDEs
in more than one spatial dimension. Specifically, in two and three spatial dimensions,
he proved shock formation results for scalar
quasilinear wave equations of the form\footnote{Alinhac's equations were perturbations
of the linear wave equation in the sense that $g_{\alpha \beta}(\partial \Phi = 0) = m_{\alpha \beta}$, 
where $m$ is the Minkowski metric, e.g.
$m_{\alpha \beta} = \mbox{\upshape diag}(-1,1,1,1)$ three spatial dimensions.}
\begin{align} \label{E:IRROTATIONALFLUIDEQUATIONFIRSTFORM}
(g^{-1})^{\alpha \beta}(\partial \Phi) \partial_{\alpha} \partial_{\beta} \Phi = 0
\end{align}
whenever the nonlinearities fail to satisfy the null condition\footnote{Klainerman formulated the null condition
in three spatial dimensions \cite{sK1984}, while Alinhac formulated it in two spatial dimensions \cite{sA2001b}. 
For equations of type
\eqref{E:IRROTATIONALFLUIDEQUATIONFIRSTFORM}, the difference is that 
in three spatial dimensions,
the definition of the null condition involves only the structure of the 
quadratic part $\partial \Phi \cdot \partial^2 \Phi$ of the nonlinearities
(obtained by Taylor expansion)
while in two spatial dimensions,
it also involves the cubic part
$\partial \Phi \cdot \partial \Phi \cdot \partial^2 \Phi$.}
and the data are small, smooth, compactly supported, 
and verify a non-degeneracy condition.
Although Alinhac's work significantly advanced our understanding
of singularity formation in solutions to quasilinear wave equations,
the most robust and precise framework for proving shock formation
in solutions to quasilinear wave equations 
was developed by Christodoulou in his groundbreaking work \cite{dC2007}.
More precisely, in \cite{dC2007}, Christodoulou proved a small-data shock formation
result for irrotational solutions to the equations of compressible relativistic
fluid mechanics. In the irrotational case, the equations are equivalent to an
Euler-Lagrange equation for a potential function $\Phi$,
which can be expressed in the form \eqref{E:IRROTATIONALFLUIDEQUATIONFIRSTFORM}.
Christodoulou's sharp geometric framework relied on a reformulation of 
the wave equation \eqref{E:IRROTATIONALFLUIDEQUATIONFIRSTFORM}
that exhibits good geo-analytic structures (see equation \eqref{E:CHRISTODOULOUSHOMOGENEOUSCOVARIANTWAVEEQUATION}),
and his approach yielded information that is not accessible via
Alinhac's approach. In particular, Christodoulou's framework
is able to reveal information about the structure of the maximal
classical development\footnote{Roughly, the maximal classical development 
		is the largest possible classical solution that is uniquely determined by the data; see, for example, 
		\cites{jSb2016,wW2013} for further discussion.}  
of the initial data, all the way up to the boundary,
information that is essential for properly setting up the \emph{shock development problem} in compressible fluid mechanics. 
Roughly, the shock development problem is the problem of weakly continuing the solution past the singularity
under suitable jump conditions. We note that even if the data are irrotational,
vorticity can be generated to the future of the first singularity.
Thus, in the study of the shock development problem, one must consider the full compressible Euler equations.
The shock development problem remains open in full generality and is expected
to be very difficult. However, Christodoulou--Lisibach recently made important progress \cite{dCaL2016}: 
they solved it in spherical symmetry in the relativistic case.

Christodoulou's shock formation results 
for the irrotational relativistic compressible Euler equations
were extended to the non-relativistic irrotational compressible Euler equations by Christodoulou--Miao in \cite{dCsM2014},
to general classes of wave equations 
\cite{jS2016b} by the author, 
and to other solution regimes in
\cites{jSgHjLwW2016,sMpY2017,sM2016}.
Readers may consult the survey article \cite{gHsKjSwW2016} for an extended overview of 
some of these works.
Of the above works, the ones \cites{dC2007,dCsM2014} are most relevant for the present article.
In those works, the authors proved 
\emph{small-data} shock formation results for the 
compressible Euler equations in irrotational regions
by studying the wave equation for the potential function $\Phi$.
The wave equation can be written in the (non-Euler-Lagrange) form \eqref{E:IRROTATIONALFLUIDEQUATIONFIRSTFORM}
relative to Cartesian coordinates,\footnote{In discussing \cite{dC2007},
it would be better for us to call them ``rectangular coordinates'' 
since the equations there are introduced in the context of special relativity, and
the Minkowski metric takes the ``rectangular'' form $\mbox{diag}(-1,1,1,1)$ relative to these coordinates.}
where the Cartesian components
$g_{\alpha \beta} = g_{\alpha \beta}(\partial \Phi)$ are determined by the fluid equation of state.
In the context of fluid mechanics, the Lorentzian metric $g$ in \eqref{E:IRROTATIONALFLUIDEQUATIONFIRSTFORM}
is known as the acoustical metric (since it drives the propagation of sound waves).
We note that the acoustical metric also plays a fundamental role in the main results of this article
(see Def.~\ref{D:ACOUSTICALMETRIC}), even when the vorticity is non-zero.

A simple but essential step in Christodoulou's proof of shock formation
was to differentiate the wave equation \eqref{E:IRROTATIONALFLUIDEQUATIONFIRSTFORM}
with the Cartesian coordinate partial derivative vectorfields $\partial_{\nu}$,
which led to the following system of \emph{covariant} wave equations:
\begin{align} \label{E:CHRISTODOULOUSHOMOGENEOUSCOVARIANTWAVEEQUATION}
	\square_{\widetilde{g}(\vec{\Psi})} \Psi_{\nu}
	& = 0.
\end{align}
In \eqref{E:CHRISTODOULOUSHOMOGENEOUSCOVARIANTWAVEEQUATION},
$\vec{\Psi} := (\Psi_0,\Psi_1,\Psi_2,\Psi_3)$ is an array of scalar functions, 
$\Psi_{\nu} := \partial_{\nu} \Phi$ (with $\partial_{\nu}$ denoting a Cartesian coordinate partial derivative), 
$\widetilde{g}$ is a Lorentzian metric conformal to $g$,
$\square_{\widetilde{g}(\vec{\Psi})}$ is the covariant wave operator of $\widetilde{g}$
(see Def.~\ref{D:COVWAVEOP}), and $\Psi_{\nu}$ is treated 
as a scalar function under covariant differentiation in \eqref{E:CHRISTODOULOUSHOMOGENEOUSCOVARIANTWAVEEQUATION}.
A key feature of the system \eqref{E:CHRISTODOULOUSHOMOGENEOUSCOVARIANTWAVEEQUATION}
is that all of the terms that drive the shock formation are on the left-hand side, 
hidden in the lower-order terms generated by the operator $\square_{\widetilde{g}(\vec{\Psi})}$.
That is, if one expands $\square_{\widetilde{g}(\vec{\Psi})} \Psi_{\nu}$ relative to the 
standard Cartesian coordinates, one encounters Riccati-type terms\footnote{In reality, what blows up
is a specific tensorial component of $\partial \vec{\Psi}$; the tensorial structure in the problem
is rather intricate.}
of the schematic form
$\partial \vec{\Psi} \cdot \partial \vec{\Psi}$ that drive the blowup of a certain tensorial component of
$\partial \vec{\Psi}$, while $\vec{\Psi}$ itself remains uniformly bounded up to the singularity;
roughly, this is what it means for solutions to \eqref{E:CHRISTODOULOUSHOMOGENEOUSCOVARIANTWAVEEQUATION} 
to form a shock. Readers may consult Subsect.~\ref{SS:SIMPLEPLANEWAVESBLOWUP}
for a more detailed description of how the Riccati-type terms lead to blowup for simple isentropic
plane wave solutions to the compressible Euler equations.

The presence of a covariant wave operator on LHS~\eqref{E:CHRISTODOULOUSHOMOGENEOUSCOVARIANTWAVEEQUATION}
was crucial for Christodoulou's analysis. The reason is that he was able to construct,
with the help of an eikonal function (see Subsect.~\ref{SS:SHOCKFORMATIONPROOFSUMMARY}), 
a collection of geometric, \emph{solution-dependent} 
vectorfields that enjoy good commutation properties with $\square_{\widetilde{g}(\vec{\Psi})}$.
He then used the vectorfields to differentiate the wave equations and to obtain estimates for the
solution's higher derivatives, much like in his celebrated proof \cite{dCsK1993}, joint with Klainerman,
of the dynamic stability of Minkowski spacetime as a solution to the Einstein-vacuum equations.
Indeed, in more than one spatial dimension, 
the main technical challenge in the proof of shock formation 
is to derive sufficient energy estimates for the geometric vectorfield derivatives of the solution
that hold all the way up to the singularity.
In the context of shock formation, this step is exceptionally technical 
and we discuss it in more detail in Sect.~\ref{S:OVERVIEWOFROLESOFTHEOREMS}.
It is important to note that the standard Cartesian coordinate partial derivatives
$\partial_{\nu}$
generate uncontrollable error terms when commuted through $\square_{\widetilde{g}(\vec{\Psi})}$
and thus the geometric vectorfields and the operator $\square_{\widetilde{g}(\vec{\Psi})}$
are essential ingredients in the proof.

In \cite{jS2016b}, we showed that if one considers a general wave equation of type
\eqref{E:IRROTATIONALFLUIDEQUATIONFIRSTFORM},
not necessarily of the Euler-Lagrange type considered by Christodoulou \cite{dC2007} and Christodoulou--Miao \cite{dCsM2014},
then upon differentiating it with $\partial_{\nu}$,
one does not generate a system of type \eqref{E:CHRISTODOULOUSHOMOGENEOUSCOVARIANTWAVEEQUATION},
but rather an inhomogeneous system of the form
\begin{align} \label{E:JAREDSINHOMOGENEOUSCOVARIANTWAVEEQUATION}
	\square_{g(\vec{\Psi})} \Psi_{\nu}
	& = \smoothfunction(\vec{\Psi})\mathfrak{Q}(\partial \vec{\Psi},\partial \Psi_{\nu}),
\end{align}
where 
$\smoothfunction$ is smooth and
$\mathfrak{Q}$ is a \emph{standard null form relative to the acoustical metric $g$}; see Def.~\ref{D:NULLFORMS}.
We then showed that the null forms relative to $g$ have precisely the right structure such that
they do not interfere with or prevent the shock formation processes, at least for suitable data.
The $\mathfrak{Q}$ are canonical examples of terms that enjoy the good null structure 
that we mentioned at the beginning of this subsection.
More generally, we refer to the good null structure as the \emph{strong null condition}; see Def.~\ref{D:STRONGNULLCONDITION} and
Prop.~\ref{P:STANDARDNULLFORMSSATISFYSTRONGNULL}.
We stress that the full nonlinear structure of the null forms $\mathfrak{Q}$ is critically important.
This is quite different from the famous null condition identified by Klainerman \cite{sK1984}
in his study of wave equations in three spatial dimensions
that enjoy small-data global existence; 
in Klainerman's formulation of the null condition,
the structure of cubic and higher order terms is not even taken into consideration 
since, in the small-data regime that he studied,
wave dispersion causes the cubic terms to decay fast enough 
that their precise structure is typically not important.
The reason that the full nonlinear structure of the null forms $\mathfrak{Q}$
is of critical importance in the study of shock formation is that
they are adapted to the acoustical metric $g$ and enjoy the following
key property: each $\mathfrak{Q}$ is \emph{linear} in the tensorial component of $\partial \vec{\Psi}$
that blows up. Therefore, near the singularity, $\mathfrak{Q}$ 
is small relative to the quadratic terms $\partial \vec{\Psi} \cdot \partial \vec{\Psi}$
that drive the singularity formation (which we again stress are hidden in the definition of 
$\square_{g(\vec{\Psi})} \Psi_{\nu}$).
Roughly, this linear dependence is the crux of the strong null condition.
In contrast, a typical quadratic inhomogeneous term $\partial \vec{\Psi} \cdot \partial \vec{\Psi}$,
if present on RHS~\eqref{E:JAREDSINHOMOGENEOUSCOVARIANTWAVEEQUATION}, would distort the dynamics
near the singularity and could in principle prevent it from forming or change its nature.
Moreover, in the context of shock formation,
cubic or higher-order terms such as 
$\partial \vec{\Psi} \cdot \partial \vec{\Psi} \cdot \partial \vec{\Psi}$ 
are expected to become dominant in regions where $\partial \vec{\Psi}$ is large
and it is therefore critically important that there are no such terms on RHS~\eqref{E:JAREDSINHOMOGENEOUSCOVARIANTWAVEEQUATION}.
These observations suggest that proofs of shock formation are less stable under perturbations
of the equations compared to more familiar perturbative proofs of global existence.

The equations in our new formulation of the compressible Euler equations (see Theorem~\ref{T:GEOMETRICWAVETRANSPORTSYSTEM})
are drastically more complicated than the homogeneous wave equations \eqref{E:CHRISTODOULOUSHOMOGENEOUSCOVARIANTWAVEEQUATION}
that Christodoulou encountered in his study of irrotational compressible fluid mechanics
and the inhomogeneous equations \eqref{E:JAREDSINHOMOGENEOUSCOVARIANTWAVEEQUATION} that we encountered in
\cite{jS2016b}. The equations of Theorem~\ref{T:GEOMETRICWAVETRANSPORTSYSTEM} 
are even considerably more complicated than the equations we derived in
\cite{jLjS2016a} in our study of the barotropic fluids with vorticity.
However, they exhibit many of the same good structures as the equations of \cite{jLjS2016a}
as well as some remarkable new ones. Specifically, in the present article,
we derive geometric equations whose inhomogeneous terms are either null forms
similar to the ones on RHS~\eqref{E:JAREDSINHOMOGENEOUSCOVARIANTWAVEEQUATION}
or less dangerous terms that are at most linear in the solution's derivatives.
We find this presence of this null structure to be somewhat miraculous in view of the 
sensitivity of proofs of shock formation under perturbations of the equations
that we described in the previous paragraph.
Moreover, in Theorem~\ref{T:GEOMETRICWAVETRANSPORTSYSTEM},
we also exhibit special combinations of the solution variables 
that solve equations with good source terms, 
which, when combined with elliptic estimates,
can be used show that the vorticity is
\emph{one degree more differentiable than one might expect};\footnote{To show the gain in regularity, 
one must use a combination of energy estimates and elliptic estimates 
along hypersurfaces of constant time.
In the present article, we do not actually derive 
energy estimates and elliptic estimates, 
but rather only equations that one can use to derive them.} see Def.~\ref{D:RENORMALIZEDCURLOFSPECIFICVORTICITY} for the special combinations, which we refer to as ``modified fluid variables.''
The gain in differentiability for the vorticity has long been known relative 
to Lagrangian coordinates, in particular because 
it has played an important role in proofs of local well-posedness 
\cites{dChLsS2010,dCsS2011,dCsS2012,jJnM2011,jJnM2009}
for the compressible Euler equations for data featuring a physical vacuum-fluid boundary.
However, the gain in differentiability for the vorticity
with respect to arbitrary vectorfield differential operators
(with coefficients of sufficient regularity relative to the solution)
seems to originate in \cite{jLjS2016a}.
The freedom to gain the derivative relative to general vectorfield differential operators
is important because Lagrangian coordinates are
not adapted to the wave characteristics, whose intersection 
corresponds to the formation of a shock.
Therefore, Lagrangian coordinates are not suitable
for following the solution all the way to the shock; instead, 
as we describe in Subsects.~\ref{SS:SHOCKFORMATIONPROOFSUMMARY} and \ref{SS:SUMMARYOFPROOF},
one needs a system of geometric coordinates constructed 
with the help of an eikonal function as well the aforementioned geometric vectorfields,
which are closely related to the geometric coordinates.
We remark that in the barotropic case \cite{jLjS2016a},
the ``special combinations'' of solution variables were simpler than they are
in the present article.
Specifically, in the barotropic case, the specific vorticity
and its curl satisfied good
transport equations; compare with \eqref{E:RENORMALIZEDCURLOFSPECIFICVORTICITY}. 
Similarly, we can prove that the entropy is one degree more differentiable 
than one might expect by studying a rescaled version of its Laplacian;\footnote{Actually, with
our future study of shock formation in mind, we formulate a div-curl-transport equation system for the
gradient of the entropy; see equations \eqref{E:TRANSPORTFLATDIVGRADENT}-\eqref{E:DIVGRADENT} and Remark~\ref{R:NEEDFORGRADENTANDDIVCURL}.}
see \eqref{E:RENORMALIZEDDIVOFENTROPY}. To the best of our knowledge, the
gain in regularity for the entropy is a new observation.

As we mentioned above, we exhibit the special null structure of the inhomogeneous terms
in Theorem~\ref{T:STRONGNULL}. Given Theorem~\ref{T:GEOMETRICWAVETRANSPORTSYSTEM},
the proof of Theorem~\ref{T:STRONGNULL} is simple and is essentially by observation.
However, it is difficult to overstate its profound importance in the study of shock formation
since, as we described above, the good null structures are essential for showing that the inhomogeneous terms are
not strong enough to interfere with the shock formation processes (at least for suitable data).
The gain of differentiability mentioned in the previous paragraph is also essential for our
forthcoming work on shock formation since we need it in order to control
some of the source terms in the wave equations.

\section{Geometric background and the strong null condition}
\label{S:LORENTZIANGEOMETRYBASIC}
In this section, we define some geometric objects and concepts that we need in order
to state our main results.

\subsection{Geometric tensorfields associated to the flow}
\label{SS:GEOMETRICTENSORFIELDS}
Roughly, there are two kinds of motion associated 
to compressible Euler flow: the transporting of vorticity and the
propagation of sound waves. We now discuss the tensorfields associated 
to these phenomena.

The material derivative vectorfield $\Transport$, defined in \eqref{E:MATERIALVECTORVIELDRELATIVETORECTANGULAR}, 
is associated to the transporting of vorticity. We now define the Lorentzian metric $g$ corresponding to the propagation of sound waves.

\begin{definition}[\textbf{The acoustical metric\footnote{Other authors have defined the acoustical metric to be $\Speed^2 g$.
	We prefer our definition because it implies that $(g^{-1})^{00} = - 1$,
	which simplifies the presentation of many formulas.} 
and its inverse}] 
\label{D:ACOUSTICALMETRIC}
We define the \emph{acoustical metric} $g$ and the 
\emph{inverse acoustical metric} $g^{-1}$ relative
to the Cartesian coordinates as follows:
\begin{subequations}
	\begin{align}
		g 
		& := 
		-  dt \otimes dt
			+ 
			\Speed^{-2} \sum_{a=1}^3(dx^a - v^a dt) \otimes (dx^a - v^a dt),
				\label{E:ACOUSTICALMETRIC} \\
		g^{-1} 
		& := 
			- \Transport \otimes \Transport
			+ \Speed^2 \sum_{a=1}^3 \partial_a \otimes \partial_a.
			\label{E:INVERSEACOUSTICALMETRIC}
	\end{align}
\end{subequations}
\end{definition}

\begin{remark}
	\label{R:GINVERSEISTHEINVERSE}
	One can easily check that
	$g^{-1}$ is the matrix inverse of $g$, that
	is, we have $(g^{-1})^{\mu \alpha} g_{\alpha \nu} = \delta_{\nu}^{\mu}$,
	where $\delta_{\nu}^{\mu}$ is the standard Kronecker delta.
\end{remark}

The vectorfield $\Transport$ enjoys some simple but important
geometric properties, which we provide in the next lemma.
We repeat the simple proof from \cite{jLjS2016a} for convenience.
\begin{lemma}[\textbf{Basic geometric properties of} $\Transport$]
	\label{L:BASICPROPERTIESOFTRANSPORT}
	$\Transport$ is $g$-timelike,\footnote{$g$-timelike vectorfields $V$ are such that $g(V,V) < 0$.}
	future-directed,\footnote{A vectorfield $V$ is future directed if 
	$V^0 > 0$, where $V^0$ is the $0$ Cartesian component.
	\label{FN:FUTUREDIRECTED}} 
	$g-$orthogonal to $\Sigma_t$,
	and unit-length:\footnote{Throughout we use the notation $g(V,W) := g_{\alpha \beta}V^{\alpha} W^{\beta}$.}
	\begin{align} \label{E:TRANSPORTISUNITLENGTH}
		g(\Transport,\Transport)
		& = - 1.
	\end{align}
\end{lemma}

\begin{proof}
	Clearly $\Transport$ is future-directed.
	The identity \eqref{E:TRANSPORTISUNITLENGTH}
	(which also implies that $\Transport$ is $g$-timelike)
	follows from a simple calculation
	based on \eqref{E:MATERIALVECTORVIELDRELATIVETORECTANGULAR} and \eqref{E:ACOUSTICALMETRIC}.
	Similarly, we compute that
	$
	\displaystyle
	g(\Transport,\partial_i)
	:= g_{\alpha i}\Transport^{\alpha}
	= 0
	$
	for $i=1,2,3$, from which it follows that
	$\Transport$
	is $g-$orthogonal to $\Sigma_t$.
\end{proof}

\subsection{Decompositions relative to null frames}
\label{SS:NULLDECOMPOSITIONS}
The special null structure of our new formulation of the compressible Euler equations, 
which we briefly described in Subsect.~\ref{SS:CONTEXT}, 
is intimately connected to the notion of a null frame.

\begin{definition}[\textbf{$g$-Null frame}]
	\label{D:NULLFRAME}
	Let $g$ be a Lorentzian metric on\footnote{The topology of the spacetime manifold is not relevant for our discussion here.} 
	$\mathbb{R}^{1+3}$.
	A $g$-\emph{null frame} (``null frame'' for short, when the metric is clear)
	at a point $q$ is a set of vectors 
	\begin{align} \label{E:NULLFRAME}
		\mathscr{N} := \lbrace \Lunit,\uLunit,e_1,e_2 \rbrace
	\end{align}
	belonging to the tangent space of $\mathbb{R}^{1+3}$ at $q$
	with
	\begin{subequations}
	\begin{align}
	g(\Lunit,\Lunit) 
	& = g(\uLunit,\uLunit) = 0,
		&& 
		\label{E:NULLFRAMEVECTORFIELDSLENGTHZERO} \\
	g(\Lunit,\uLunit) &= -2,
		&& 
		\label{E:NULLPAIRINNERPRODUCTMINUS2} \\
	g(\Lunit,e_A) &= g(\uLunit,e_A) = 0, 
	&& (A=1,2),
		\label{E:NULLVECTORFIELDSORTHOGONALTOSPACELIKEONES} \\
	g(e_A,e_B) &= \delta_{AB},
	&& (A,B=1,2),
	\label{E:SPACELIKEVECTORFIELDSORTHONORMAL}
	\end{align}
	\end{subequations}
	where $\delta_{AB}$ is the standard 
	Kronecker delta.
\end{definition}

The following lemma is a simple consequence of Def.~\ref{D:NULLFRAME};
we omit the simple proof.

\begin{lemma}[\textbf{Decomposition of $g^{-1}$ relative to a null frame}]
	\label{L:DECOMPOFGINVERSERELATIVETONULLFRAME}
	Relative to an arbitrary $g-$null frame, we have
	\begin{align} \label{E:DECOMPOFGINVERSERELATIVETONULLFRAME}
		g^{-1} 
		& = - \frac{1}{2} \Lunit \otimes \uLunit
				- \frac{1}{2} \uLunit \otimes \Lunit
				+ \sum_{A=1}^2 e_A \otimes e_A.
	\end{align}
\end{lemma}

\begin{definition}[\textbf{Decomposition of a derivative-quadratic nonlinear term relative to a null frame}]
	\label{D:DECOMPOSINGNONLINEARTERMRELATIVETONULLFRAME}
	Let 
	\begin{align} \label{E:SOLUTIONARRAY}
	\vec{V} 
	& := 
		(\Densrenormalized,v^1,v^2,v^3,\Ent,\Vortrenormalized^1,\Vortrenormalized^2,\Vortrenormalized^3,\GradEnt^1,\GradEnt^2,\GradEnt^3)
	\end{align}
	be the array of unknowns in the system
	\eqref{E:VELOCITYWAVEEQUATION}-\eqref{E:DIVGRADENT}.
	We label the components of $\vec{V}$ as follows:
	\[
	V^0=\Densrenormalized, 
	V^i=v^i, 
	V^4 = \Ent,
	V^{i+4}=\Vortrenormalized^i, 
	\mbox{ and }
	V^{i+7}=\GradEnt^i
	\mbox{ for } 
	i=1,2,3.
	\]
	Let $\mathcal{N}(\vec{V},\partial \vec{V})$
	be a smooth nonlinear term
	that is quadratically nonlinear in $\partial \vec{V}$.
	That is, we assume that 
	$\mathcal{N}(\vec{V},\partial \vec{V})
	=\smoothfunction(\vec V)_{\Theta\Gamma}^{\alpha \beta}\partial_{\alpha} V^\Theta \partial_{\beta} V^\Gamma$, 
	where $\smoothfunction(\vec V)_{\Theta\Gamma}^{\alpha \beta}$ is symmetric in $\Theta$ and $\Gamma$ 
	and is a smooth function of $\vec V$ 	(\emph{not} necessarily vanishing at $0$) 
	for $\alpha,\beta=0,1,2,3$ and $\Theta,\Gamma=0,1,\cdots,10$.
	Given a null frame $\mathscr{N}$ as defined in Def.~\ref{D:NULLFRAME}, 
	we denote
	\[
\mathscr{N} := \lbrace e_1, e_2, e_3:= \uLunit, e_4 := \Lunit \rbrace.
\]
Moreover, we let $M_{\alpha}^A$ be the scalar functions
corresponding to expanding the Cartesian coordinate partial derivative vectorfield
$\partial_{\alpha}$ at $q$ relative to the null frame, that is,
\[
\partial_{\alpha} = \sum_{A=1}^4 M_{\alpha}^A e_A. 
\]
	Then\footnote{Here and below, we use Einstein's summation convention, where uppercase Latin indices 
	such as $A$ and $B$
	vary over $1,2,3,4$, 
	lowercase Latin ``spatial'' indices such as $a$ and $b$ vary over $1,2,3$,
	uppercase Greek indices such as $\Theta$ and $\Gamma$
	vary over $0,1,\dots,10$, and lowercase Greek ``spacetime'' such as $\alpha$ and $\beta$
	indices vary over $0,1,2,3$. \label{FN:INDEXCONVENTIONS}}
	\begin{align} \label{E:NONLINEARTERMDECOMPOSEDRELATIVETONULLFRAME}
		\mathcal{N}_{\mathscr{N}} :=\smoothfunction(\vec V)_{\Theta\Gamma}^{\alpha \beta} 
			M_{\alpha}^A M_{\beta}^B (e_A V^\Theta)(e_B V^\Gamma)
	\end{align}
	denotes the nonlinear term
	obtained by expressing $\mathcal{N}(\vec{V},\partial \vec{V})$ in terms of the derivatives
	of $\vec{V}$ with respect to the elements of $\mathscr{N}$, that is, by
	expanding $\partial \vec{V}$ as a linear combination of 
	the derivatives of $\vec{V}$
	with respect to the elements
	of $\mathscr{N}$ and substituting the expression for the factor $\partial \vec{V}$
	in $\mathcal{N}(\vec{V},\partial \vec{V})$.
\end{definition}

\subsection{Strong null condition and standard null forms}
\label{SS:STRONGNULLANDNULLFORMS}
In Subsect.~\ref{SS:CONTEXT}, we roughly described 
the special null structure enjoyed by the inhomogeneous terms
in our new formulation of the compressible Euler equations.
We precisely define the special null structure
in the next definition, which we recall from \cite{jLjS2016a}.

\begin{definition}[\textbf{Strong null condition}]
	\label{D:STRONGNULLCONDITION}
	Let $\mathcal{N}(\vec{V},\partial \vec{V})$ 
	be as in Def. \ref{D:DECOMPOSINGNONLINEARTERMRELATIVETONULLFRAME}. We say that $\mathcal{N}(\vec{V},\partial \vec{V})$ 
	verifies the \emph{strong null condition} relative to $g$ if
	the following condition holds: for \emph{every} $g$-null frame $\mathscr{N}$,
	$\mathcal{N}_{\mathscr{N}}$ can be expressed in a form that depends linearly 
	(or not at all)
	on $\Lunit \vec{V}$ and $\uLunit \vec{V}$. 
	That is, there exist scalars 
	$\overline{\smoothfunction}_{\Theta \Gamma}^{AB}(\vec V)$ 
	and 
	$\underline{\smoothfunction}_{\Theta\Gamma}^{AB}(\vec V)$ 
	such that
	\begin{align} \label{E:KEYCOMPONENTSVANISH}
	\overline{\smoothfunction}_{\Theta\Gamma}^{33}(\vec V)
	& = 
	\overline{\smoothfunction}_{\Theta\Gamma}^{44}(\vec V)=0,
	&&
	\underline{\smoothfunction}_{\Theta\Gamma}^{33}(\vec V)
	=
	\underline{\smoothfunction}_{\Theta\Gamma}^{44}(\vec V)
	 	= 0
	\end{align}
	and such that the following hold for $\Theta,\Gamma = 0,1,\cdots,10$:
	\begin{equation}\label{strong.null.def}
	\begin{split}
	\smoothfunction(\vec V)_{\Theta\Gamma}^{\alpha \beta} M_{\alpha}^3 M_{\beta}^3 
	(e_3 V^\Theta)(e_3 V^\Gamma)=&\overline{\smoothfunction}_{\Theta\Gamma}^{AB}(\vec V)(e_A V^\Theta)(e_B V^\Gamma),\\
	\smoothfunction(\vec V)_{\Theta\Gamma}^{\alpha \beta} M_{\alpha}^4 M_{\beta}^4 
	(e_4 V^\Theta)(e_4 V^\Gamma)
	= & \underline{\smoothfunction}_{\Theta\Gamma}^{AB}(\vec V)(e_A V^\Theta)(e_B V^\Gamma).  
	\end{split}
	\end{equation}
	\end{definition}
	
	\begin{remark}[\textbf{Some comments on the strong null condition}]
		\label{R:WHYCOMPLICATEDSTRONGNULL}
		Equation \eqref{strong.null.def} allows for the possibility that one uses 
		evolution equations to algebraically substitute for terms
		on LHS~\eqref{strong.null.def}, thereby generating
		the good terms on RHS~\eqref{strong.null.def}, which verify 
		the essential condition \eqref{E:KEYCOMPONENTSVANISH}. 
		 As our proof of Prop.~\ref{P:STANDARDNULLFORMSSATISFYSTRONGNULL}
		below shows, this kind of substitution is not needed for null form nonlinearities,
		which can directly be shown to exhibit the desired structure without the help of external evolution equations.
		That is, for null forms $\mathfrak{Q}$, one can directly show that 
		$
		\smoothfunction(\vec V)_{\Theta\Gamma}^{\alpha \beta} M_{\alpha}^3 M_{\beta}^3 
		= 
		\smoothfunction(\vec V)_{\Theta\Gamma}^{\alpha \beta} M_{\alpha}^4 M_{\beta}^4
		=
		0
		$.
		In the present article, the formulation of the equations that
		we provide (see Theorem~\ref{T:GEOMETRICWAVETRANSPORTSYSTEM}) is such that all derivative-quadratic
		terms are null forms. Readers might then wonder why our definition of the strong null condition
		allows for the more complicated scenario in which one uses external evolution equations for algebraic substitution
		to detect the good null structure.
		The reason is that in our work \cite{jLjS2016a} on the barotropic case, 
		we encountered the inhomogeneous terms
		$
		\displaystyle
		\epsilon_{iab} 
			 \left\lbrace
				(\partial_a \Vortrenormalized^d) \partial_d v^b
				- 
				(\partial_a v^d) \partial_d \Vortrenormalized^b
			 \right\rbrace
		$,
		which are not null forms. To show that this term had the desired null structure, 
		we used evolution equations for substitution and therefore relied on the full scope of Def.~\ref{D:STRONGNULLCONDITION}.
		In the present article, we encounter the same term,
		but we treat it in a different way 
		and show that in fact,
		$
		\displaystyle
		\epsilon_{iab} 
			 \left\lbrace
				(\partial_a \Vortrenormalized^d) \partial_d v^b
				- 
				(\partial_a v^d) \partial_d \Vortrenormalized^b
			 \right\rbrace
		$
		is equal to a null form plus other terms that are either harmless or 
		that can be incorporated
		into our definition of the modified fluid variables from Def.~\ref{D:RENORMALIZEDCURLOFSPECIFICVORTICITY};
		see the identity \eqref{E:REWRITINGOFCURLVORTICITYNULLFORM}
		and the calculations below it.
	\end{remark}
	
A key feature of our new formulation of the compressible
Euler equations is that all derivative-quadratic
inhomogeneous terms are linear combinations of
the standard null forms relative to the acoustical metric $g$,
which verify the strong null condition relative to $g$
(see Prop.~\ref{P:STANDARDNULLFORMSSATISFYSTRONGNULL}).
We now recall their standard definition.

\begin{definition}[\textbf{Standard null forms}]
	\label{D:NULLFORMS}
	The standard null forms $\mathfrak{Q}^g(\cdot,\cdot)$
	(relative to $g$)
	and
	$\mathfrak{Q}_{(\alpha \beta)}(\cdot,\cdot)$
	act on pairs $(\phi,\widetilde{\phi})$
	of scalar-valued functions as follows:
	\begin{subequations}
		\begin{align}
		\mathfrak{Q}^g(\partial \phi, \partial \widetilde{\phi})
		&:= (g^{-1})^{\alpha \beta} \partial_{\alpha} \phi \partial_{\beta} \widetilde{\phi},
			\label{E:Q0NULLFORM} \\
		\mathfrak{Q}_{(\alpha \beta)}(\partial \phi, \partial \widetilde{\phi})
		& = \partial_{\alpha} \phi \partial_{\beta} \widetilde{\phi}
			-
			 \partial_{\alpha} \widetilde{\phi} \partial_{\beta} \phi.
			\label{E:QALPHABETANULLFORM}
		\end{align}
	\end{subequations}
\end{definition}

\begin{proposition}[\textbf{The standard null forms satisfy the strong null condition}]
	\label{P:STANDARDNULLFORMSSATISFYSTRONGNULL}
	Let $\mathfrak{Q}$ be a standard null form relative to $g$ and let
	$\phi$ and $\widetilde{\phi}$ be any two entries of the array $\vec{V}$ 
	from Def.~\ref{D:DECOMPOSINGNONLINEARTERMRELATIVETONULLFRAME}.
	Let $\smoothfunction = \smoothfunction(\vec{V})$	
	be a smooth scalar-valued function of the entries of $\vec{V}$.
	Then $\smoothfunction(\vec{V}) \mathfrak{Q}(\partial \phi, \partial \widetilde{\phi})$
	verifies the strong null condition relative to $g$ from Def.~\ref{D:STRONGNULLCONDITION}.
\end{proposition}

\begin{proof}
	In the case of the null form $\mathfrak{Q}^g$,
	the proof is a direct consequence of 
	the identity \eqref{E:DECOMPOFGINVERSERELATIVETONULLFRAME}.
	
	In the case of the null form $\mathfrak{Q}_{(\alpha \beta)}$
	defined in \eqref{E:QALPHABETANULLFORM},
we consider any $g$-null frame \eqref{E:NULLFRAME}, and we label its elements as follows:
$\mathscr{N} := \lbrace e_1, e_2, e_3:= \uLunit, e_4 := \Lunit \rbrace$.
Since $\mathscr{N}$ spans the tangent space at each point where it is defined, 
there exist scalar functions
$M_{\alpha}^A$ such that the following identity holds 
for $\alpha = 0,1,2,3$:
\begin{equation}\label{M.def}
\partial_{\alpha} = \sum_1^4 M_{\alpha}^A e_A.
\end{equation}
From \eqref{E:QALPHABETANULLFORM}
and \eqref{M.def},
we deduce
\[
\mathfrak{Q}_{(\alpha \beta)}(\partial \phi,\partial \widetilde{\phi})
= \sum_{A,B=1}^4
	\left\lbrace 
		M_{\alpha}^A M_{\beta}^B 
		-
		M_{\alpha}^B M_{\beta}^A
	\right\rbrace
	(e_A \phi) 
	e_B \widetilde{\phi}.
\]
The key point is that the terms in braces are antisymmetric in $A$ and $B$.
It follows that the sum does not contain any diagonal terms, that is, terms
proportional to $(e_A \phi) e_A \widetilde{\phi}$.
In particular, terms proportional to $(\uLunit \phi) \uLunit \widetilde{\phi}$
and
$(\Lunit \phi) \Lunit \widetilde{\phi}$ are not present,
which is the desired result.
\end{proof}

\section{Precise statement of the main results}
\label{S:MAINRESULTS}
\setcounter{equation}{0}
In this section, we precisely state our two main theorems and give the simple proof of the second one.
We start recalling the standard definition of the covariant wave operator $\square_g$.

\begin{definition}[\textbf{Covariant wave operator}]
\label{D:COVWAVEOP}
Let $g$ be a Lorentzian metric.
Relative to arbitrary coordinates,
the covariant wave operator $\square_g$ 
acts on scalar-valued functions $\phi$ as follows:
\begin{align} \label{E:WAVEOPERATORARBITRARYCOORDINATES}
\square_g \phi
	= \frac{1}{\sqrt{|\mbox{\upshape det} g|}}
	\partial_{\alpha}
	\left\lbrace
			\sqrt{|\mbox{\upshape det} g|} (g^{-1})^{\alpha \beta}
			\partial_{\beta} \phi
	\right\rbrace.
\end{align}
\end{definition}

\subsection{The new formulation of the compressible Euler equations with entropy}
\label{SS:NEWFORMULATION}
Our first main result is Theorem~\ref{T:GEOMETRICWAVETRANSPORTSYSTEM}, which provides the new formulation
of the compressible Euler equations. 
We postpone its lengthy proof until Sect.~\ref{S:PROOFOFMAINTHEOREM}.

\begin{remark}[\textbf{Explanation of the different kinds of inhomogeneous terms}]
	In the equations of Theorem~\ref{T:GEOMETRICWAVETRANSPORTSYSTEM},
	there are many inhomogeneous terms that are denoted by decorated versions of
	$\mathfrak{Q}$. These terms are linear combinations of $g$-null forms that, 
	in our forthcoming proof of shock formation, 
	can be controlled in the energy estimates without elliptic estimates.
	Similarly, in the equations of Theorem~\ref{T:GEOMETRICWAVETRANSPORTSYSTEM},
	decorated versions of the symbol $\mathfrak{L}$ 
	denote terms that are at most linear in the derivatives of the solution
	and that can be controlled in the energy estimates without elliptic estimates. 
	In our forthcoming proof of proving shock formation,
	the $\mathfrak{Q}$'s and $\mathfrak{L}$'s will be simple error terms.
	The equations of Theorem~\ref{T:GEOMETRICWAVETRANSPORTSYSTEM}
	also feature additional null form inhomogeneous terms 
	depending on 
	$\partial \Vortrenormalized$
	and
	$\partial \GradEnt$,
	which we explicitly display because one needs elliptic estimates
	along $\Sigma_t$ to control them in the energy estimates. 
	For this reason, in the proof of shock formation,
	these terms are substantially more difficult to 
	bound compared to the $\mathfrak{Q}$'s and $\mathfrak{L}$'s.
	Similarly, terms that are linear in
	$\partial \Vortrenormalized$, 
	$\partial \GradEnt$,
	$\CurlofVortrenormalized$,
	or
	$\DivofEntrenormalized$
	can be controlled only with the help 
	of elliptic estimates along $\Sigma_t$.
\end{remark}

\begin{theorem}[\textbf{The geometric wave-transport-divergence-curl formulation of the compressible Euler equations}]
	\label{T:GEOMETRICWAVETRANSPORTSYSTEM}
	Let $\bar{\varrho} > 0$ be any constant background density (see \eqref{E:BACKGROUNDDENSITY}),
	let $\Transport$ be the material derivative vectorfield defined in \eqref{E:MATERIALVECTORVIELDRELATIVETORECTANGULAR},
	let $g$ be the acoustical metric from Def.~\ref{D:NULLFORMS},
	and let $\CurlofVortrenormalized$ and $\DivofEntrenormalized$
	be the modified fluid variables from Def.~\ref{D:RENORMALIZEDCURLOFSPECIFICVORTICITY}.
	In three spatial dimensions under an arbitrary equation of state \eqref{E:EOS} with positive sound speed 
	$\Speed$ (see \eqref{E:SOUNDSPEED}),
	the compressible Euler equations 
	\eqref{E:TRANSPORTDENSRENORMALIZEDRELATIVETORECTANGULAR}-\eqref{E:ENTROPYTRANSPORT}
	imply the following equations
	(see Footnote~\ref{F:VECTORFIELDSACTONFUNCTIONS})
	for the scalar-valued functions
	$\Densrenormalized$
	and
	$v^i$,
	$\Vortrenormalized^i$,
$\Ent$,
$\GradEnt^i$,
$\Flatdiv \Vortrenormalized$,
$\CurlofVortrenormalized^i$,
$\DivofEntrenormalized$,
and
$(\Flatcurl \GradEnt)^i$,
	$(i=1,2,3)$,
	where the Cartesian component functions $v^i$ are 
	\textbf{treated as scalar-valued functions
	under covariant differentiation} on LHS~\eqref{E:VELOCITYWAVEEQUATION}:

\medskip

\noindent \underline{\textbf{\upshape Covariant wave equations}}:
	\begin{subequations}
	\begin{align}
		\square_g v^i
		& = 
			- 
			\Speed^2 \exp(2 \Densrenormalized) \CurlofVortrenormalized^i
			+ 
			\mathfrak{Q}_{(v)}^i
			+ 
			\mathfrak{L}_{(v)}^i,
			\label{E:VELOCITYWAVEEQUATION}	\\
	\square_g \Densrenormalized
	& = 
		-
		\exp(\Densrenormalized) \frac{p_{;\Ent}}{\bar{\varrho}} \DivofEntrenormalized
		+
		\mathfrak{Q}_{(\Densrenormalized)}
		+
		\mathfrak{L}_{(\Densrenormalized)}.
			\label{E:RENORMALIZEDDENSITYWAVEEQUATION} 
\end{align}

\medskip

\noindent \underline{\textbf{\upshape Transport equations}}:
\begin{align}	\Transport \Vortrenormalized^i
	& = \mathfrak{L}_{(\Vortrenormalized)}^i,
		\label{E:RENORMALIZEDVORTICTITYTRANSPORTEQUATION}
		\\
	\Transport \Ent
	& = 0,	
		\label{E:ENTROPYTRANSPORTMAINSYSTEM}
			\\
	\Transport \GradEnt^i
	& = \mathfrak{L}_{(\GradEnt)}^i.
		\label{E:GRADENTROPYTRANSPORT}
	\end{align}
	\end{subequations}

\medskip	
	
\noindent \underline{\textbf{\upshape Transport-divergence-curl system for the specific vorticity}}
\begin{subequations}
\begin{align} \label{E:FLATDIVOFRENORMALIZEDVORTICITY}
	\Flatdiv \Vortrenormalized
	& = 
		\mathfrak{L}_{(\Flatdiv \Vortrenormalized)},
		\\
\Transport 
\CurlofVortrenormalized^i
& = 
		- 
		2 \delta_{jk} \epsilon_{iab} \exp(-\Densrenormalized) (\partial_a v^j) \partial_b \Vortrenormalized^k
		+
		\epsilon_{ajk}
		\exp(-\Densrenormalized)
		(\partial_a v^i) 
		\partial_j \Vortrenormalized^k
		\label{E:EVOLUTIONEQUATIONFLATCURLRENORMALIZEDVORTICITY} 
			\\
& \ \
		+ 
		\exp(-3 \Densrenormalized) \Speed^{-2} \frac{p_{;\Ent}}{\bar{\varrho}} 
		\left\lbrace
			(\Transport \GradEnt^a) \partial_a v^i
			-
			(\Transport v^i) \partial_a \GradEnt^a
		\right\rbrace
		\notag \\
	& \ \
		+
		\exp(-3 \Densrenormalized) \Speed^{-2} \frac{p_{;\Ent}}{\bar{\varrho}}  
		\left\lbrace
			(\Transport v^a) \partial_a \GradEnt^i
			- 
			(\partial_a v^a) \Transport \GradEnt^i
		\right\rbrace
		\notag \\
& \ \
	+
	\mathfrak{Q}_{(\CurlofVortrenormalized)}^i
	+	
	\mathfrak{L}_{(\CurlofVortrenormalized)}^i.
	\notag 
\end{align}	
\end{subequations}

\medskip

\noindent \underline{\textbf{\upshape Transport-divergence-curl system for the entropy gradient}}
\begin{subequations}
\begin{align} 	
\Transport \DivofEntrenormalized
	& =  
			2 \exp(-2 \Densrenormalized) 
			\left\lbrace
				(\partial_a v^a) \partial_b \GradEnt^b
				-
				(\partial_a v^b) \partial_b \GradEnt^a
			\right\rbrace
			+
			\exp(-\Densrenormalized) \delta_{ab} (\Flatcurl \Vortrenormalized)^a \GradEnt^b
			\label{E:TRANSPORTFLATDIVGRADENT}
				\\
	& \ \
			+
			\mathfrak{Q}_{(\DivofEntrenormalized)},
			\notag 
			\\
	(\Flatcurl \GradEnt)^i & = 0.
	\label{E:DIVGRADENT}
\end{align}
\end{subequations}

	Above, 
	$\mathfrak{Q}_{(v)}^i$,
	$\mathfrak{Q}_{(\Densrenormalized)}$, 
	$\mathfrak{Q}_{(\CurlofVortrenormalized)}^i$,
	and
	$\mathfrak{Q}_{(\DivofEntrenormalized)}$
	are the \textbf{null forms relative to} $g$ defined by
	\begin{subequations}
		\begin{align}
		\mathfrak{Q}_{(v)}^i	
		& := 	-
					\left\lbrace
						1
						+
						\Speed^{-1} \Speed_{;\Densrenormalized}
					\right\rbrace
					(g^{-1})^{\alpha \beta} \partial_{\alpha} \Densrenormalized \partial_{\beta} v^i,
			\label{E:VELOCITYNULLFORM} \\
		\mathfrak{Q}_{(\Densrenormalized)}
		& := 
		- 
		3 \Speed^{-1} \Speed_{;\Densrenormalized} 
		(g^{-1})^{\alpha \beta} \partial_{\alpha} \Densrenormalized \partial_{\beta} \Densrenormalized
		+ 
		\left\lbrace
			(\partial_a v^a) (\partial_b v^b)
			-
			(\partial_a v^b) \partial_b v^a
		\right\rbrace,
			\label{E:DENSITYNULLFORM}
				\\
	\mathfrak{Q}_{(\CurlofVortrenormalized)}^i
	& :=
		\exp(-3 \Densrenormalized) \Speed^{-2} \frac{p_{;\Ent}}{\bar{\varrho}}  \GradEnt^i
		\left\lbrace
			(\partial_a v^b) \partial_b v^a
			-
			(\partial_a v^a) \partial_b v^b
		\right\rbrace
			\label{E:RENORMALIZEDVORTICITYCURLNULLFORM} \\
& \ \
		+ 
		\exp(-3 \Densrenormalized) \Speed^{-2} \frac{p_{;\Ent}}{\bar{\varrho}}
		\left\lbrace
			(\partial_a v^a) \GradEnt^b \partial_b v^i 
			- 
			(\GradEnt^a \partial_a v^b) \partial_b v^i 
		\right\rbrace
		\notag \\
& \ \
		+ 
		2 \exp(-3 \Densrenormalized) \Speed^{-2} \frac{p_{;\Ent}}{\bar{\varrho}}
		\left\lbrace
			(\GradEnt^a \partial_a \Densrenormalized)  \Transport v^i
			 - 
		  (\Transport \Densrenormalized) \GradEnt^a \partial_a v^i
		\right\rbrace
		\notag \\
	&  \ \
			+ 
			2 \exp(-3 \Densrenormalized) \Speed^{-3} \Speed_{;\Densrenormalized} \frac{p_{;\Ent}}{\bar{\varrho}}
			\left\lbrace
				(\GradEnt^a \partial_a \Densrenormalized)  \Transport v^i
				- 
				(\Transport \Densrenormalized) \GradEnt^a \partial_a v^i
			\right\rbrace
		\notag \\
	& \ \
		+ 
		\exp(-3 \Densrenormalized) \Speed^{-2} \frac{p_{;\Ent;\Densrenormalized}}{\bar{\varrho}}
		\left\lbrace
			(\Transport \Densrenormalized) \GradEnt^a \partial_a v^i
			- 
			(\GradEnt^a \partial_a \Densrenormalized) \Transport v^i
		\right\rbrace
		\notag \\
	& \ \
		+
		\exp(-3 \Densrenormalized) \Speed^{-2} \frac{p_{;\Ent;\Densrenormalized}}{\bar{\varrho}} \GradEnt^i
		\left\lbrace
			(\Transport v^a) \partial_a \Densrenormalized
			-
			 (\partial_a v^a) \Transport \Densrenormalized
		\right\rbrace
		\notag \\
	& \ \
		+
		2 \exp(-3 \Densrenormalized) \Speed^{-2} \frac{p_{;\Ent}}{\bar{\varrho}} \GradEnt^i
		\left\lbrace
			(\partial_a v^a) \Transport \Densrenormalized
			- 
			(\Transport v^a) \partial_a \Densrenormalized
		\right\rbrace
		\notag \\
	& \ \
		 	+
			2 \exp(-3 \Densrenormalized) \Speed^{-3} \Speed_{;\Densrenormalized} \frac{p_{;\Ent}}{\bar{\varrho}} \GradEnt^i
			\left\lbrace
				(\partial_a v^a) \Transport \Densrenormalized
		 		- 
		 		(\Transport v^a) \partial_a \Densrenormalized
		 	\right\rbrace,
		 		\notag
		 		\\
\label{E:DIVENTROPYGRADIENTNULLFORM}
	\mathfrak{Q}_{(\DivofEntrenormalized)}
	& :=
		2 \exp(-2 \Densrenormalized) 
		\left\lbrace
			(\GradEnt^a \partial_a v^b) \partial_b \Densrenormalized
			-
			(\partial_a v^a) \GradEnt^b \partial_b \Densrenormalized
		\right\rbrace.
\end{align}
\end{subequations}
	In addition, the terms
	$\mathfrak{L}_{(v)}^i$,
	$\mathfrak{L}_{(\Densrenormalized)}$,
	$\mathfrak{L}_{(\Vortrenormalized)}^i$,
	$\mathfrak{L}_{(\GradEnt)}^i$,
		$\mathfrak{L}_{(\Flatdiv \Vortrenormalized)}$,
		and
	$\mathfrak{L}_{(\CurlofVortrenormalized)}^i$,
	which are at most linear in the derivatives of the unknowns, are defined as follows:
	\begin{subequations}
	\begin{align} \label{E:VELOCITYILINEARORBETTER} 
		\mathfrak{L}_{(v)}^i
		& := 
		2 \exp(\Densrenormalized) \epsilon_{iab} (\Transport v^a) \Vortrenormalized^b
		-
		\frac{p_{;\Ent}}{\bar{\varrho}} \epsilon_{iab} \Vortrenormalized^a \GradEnt^b
			\\
	& \ \
		- 
		\exp(-\Densrenormalized) \frac{p_{;\Ent}}{\bar{\varrho}} \GradEnt^a \partial_a v^i
		- 
		\frac{1}{2} \exp(-\Densrenormalized) \frac{p_{;\Densrenormalized;\Ent}}{\bar{\varrho}} \GradEnt^a \partial_a v^i
			\notag \\
		& \ \
		- 
		2 \exp(-\Densrenormalized) \Speed^{-1} \Speed_{;\Densrenormalized} \frac{p_{;\Ent}}{\bar{\varrho}} 
		(\Transport \Densrenormalized) \GradEnt^i
		+
		\exp(-\Densrenormalized) \frac{p_{;\Ent;\Densrenormalized}}{\bar{\varrho}} (\Transport \Densrenormalized) \GradEnt^i,
			\notag  \\
		\mathfrak{L}_{(\Densrenormalized)}
		& :=
		-
		2 \exp(-\Densrenormalized) \frac{p_{;\Ent;\Densrenormalized}}{\bar{\varrho}} \GradEnt^a \partial_a \Densrenormalized 
		-
		\exp(-\Densrenormalized) \frac{p_{;\Ent;\Ent}}{\bar{\varrho}} \delta_{ab} \GradEnt^a \GradEnt^b,
		 \label{E:DENSITYLINEARORBETTER}
		 	\\
		\mathfrak{L}_{(\Vortrenormalized)}^i
		& := 
		\Vortrenormalized^a \partial_a v^i
		-
		\exp(-2 \Densrenormalized) \Speed^{-2} \frac{p_{;\Ent}}{\bar{\varrho}} \epsilon_{iab} (\Transport v^a) \GradEnt^b,
		\label{E:SPECIFICVORTICITYLINEARORBETTER}
			\\
		\mathfrak{L}_{(\GradEnt)}^i
		& :=
			- \GradEnt^a \partial_a v^i
			+ 
			\epsilon_{iab} \exp(\Densrenormalized) \Vortrenormalized^a \GradEnt^b,
			\label{E:ENTROPYGRADIENTLINEARORBETTER}
				\\
	\mathfrak{L}_{(\Flatdiv \Vortrenormalized)}
		& := - \Vortrenormalized^a \partial_a \Densrenormalized,
		\label{E:RENORMALIZEDVORTICITYDIVLINEARORBETTER} \\
	\mathfrak{L}_{(\CurlofVortrenormalized)}^i	 		
	& :=
		 		2 \exp(-3 \Densrenormalized) \Speed^{-3} \Speed_{;\Ent} \frac{p_{;\Ent}}{\bar{\varrho}} 
				(\Transport v^i) \delta_{ab} \GradEnt^a \GradEnt^b
				-
				2 \exp(-3 \Densrenormalized) \Speed^{-3} \Speed_{;\Ent} \frac{p_{;\Ent}}{\bar{\varrho}} 
				\delta_{ab} \GradEnt^a (\Transport v^b) \GradEnt^i
			\label{E:RENORMALIZEDVORTICITYCURLLINEARORBETTER} \\
		& \ \
			+ 
			\exp(-3 \Densrenormalized) \Speed^{-2} \frac{p_{;\Ent;\Ent}}{\bar{\varrho}} (\Transport v^i) \delta_{ab} \GradEnt^a \GradEnt^b
			- 
			\exp(-3 \Densrenormalized) \Speed^{-2} \frac{p_{;\Ent;\Ent}}{\bar{\varrho}} \delta_{ab} (\Transport v^a) \GradEnt^b \GradEnt^i.
			\notag 
	\end{align}
	\end{subequations}
\end{theorem}

\begin{remark}[\textbf{Comparison to the results of} \cite{jLjS2016a}]
	For barotropic fluids, we have $p_{;\Ent} \equiv 0$,
	and the variables $\Ent$ and $\GradEnt^i$
	do not influence the dynamics of the remaining solution variables.
	For such fluids, one can check that
	equations \eqref{E:VELOCITYWAVEEQUATION}-\eqref{E:RENORMALIZEDDENSITYWAVEEQUATION},
	\eqref{E:RENORMALIZEDVORTICTITYTRANSPORTEQUATION},
	and
	\eqref{E:FLATDIVOFRENORMALIZEDVORTICITY}-\eqref{E:EVOLUTIONEQUATIONFLATCURLRENORMALIZEDVORTICITY}
	are equivalent to the equations that we derived in
	\cite{jLjS2016a}. However, one needs some observations described in Remark~\ref{R:WHYCOMPLICATEDSTRONGNULL}
	in order to see the equivalence.
\end{remark}

\begin{remark}[\textbf{The data for the system} \eqref{E:VELOCITYWAVEEQUATION}-\eqref{E:DIVGRADENT}]
	\label{R:REMARKSONDATA}
	The fundamental initial data for the compressible Euler equations 	
	\eqref{E:TRANSPORTDENSRENORMALIZEDRELATIVETORECTANGULAR}-\eqref{E:ENTROPYTRANSPORT}
	are 
	$\Densrenormalized|_{t=0}$,
	$\lbrace v^i|_{t=0} \rbrace_{i=1,2,3}$,
	and $\Ent|_{t=0}$. 
	On the other hand, to solve the Cauchy problem for the system
	\eqref{E:VELOCITYWAVEEQUATION}-\eqref{E:DIVGRADENT},
	one also needs the data
	$\partial_t \Densrenormalized|_{t=0}$,
	$\lbrace \partial_t v^i|_{t=0} \rbrace_{i=1,2,3}$
	$\lbrace \Vortrenormalized^i|_{t=0} \rbrace_{i=1,2,3}$,
	and
	$\lbrace \GradEnt^i|_{t=0} \rbrace|_{i=1,2,3}$.
	These data can be obtained by differentiating the fundamental data
	and using equations \eqref{E:TRANSPORTDENSRENORMALIZEDRELATIVETORECTANGULAR}-\eqref{E:ENTROPYTRANSPORT}
	for substitution.\footnote{One could also consider more general
	data in which 
	$\partial_t \Densrenormalized|_{t=0}$,
	$\lbrace \partial_t v^i|_{t=0} \rbrace_{i=1,2,3}$,
	$\lbrace \Vortrenormalized^i|_{t=0} \rbrace_{i=1,2,3}$,
	and
	$\lbrace \GradEnt^i|_{t=0} \rbrace|_{i=1,2,3}$
	are chosen independently of the fundamental data. 
	However, the corresponding solution
	to the system \eqref{E:VELOCITYWAVEEQUATION}-\eqref{E:DIVGRADENT}
	would not generally be a solution to the compressible Euler equations
	\eqref{E:TRANSPORTDENSRENORMALIZEDRELATIVETORECTANGULAR}-\eqref{E:ENTROPYTRANSPORT}.}	
\end{remark}

\subsection{The structure of the inhomogeneous terms}
\label{SS:STRUCTUREOFINHOMOGENEOUSTERMS}
The next theorem is our second main result. 
In the theorem, we characterize the structure of the 
inhomogeneous terms in the equations Theorem~\ref{T:GEOMETRICWAVETRANSPORTSYSTEM}.
The most important part of the theorem is the null structure of the
type $\textbf{(iii)}$ terms.

\begin{theorem}[\textbf{The structure of the inhomogeneous terms}]
	\label{T:STRONGNULL}
	Let
	\[\vec{V}
	:=
	(\Densrenormalized,v^1,v^2,v^3,\Ent,\Vortrenormalized^1,\Vortrenormalized^2,\Vortrenormalized^3,\GradEnt^1,\GradEnt^2,\GradEnt^3)
	\]
	denote the array of unknowns.
	The inhomogeneous terms on the right-hand sides of equations 
	\eqref{E:VELOCITYWAVEEQUATION}-\eqref{E:DIVGRADENT}
	consist of three types: 
	\renewcommand{\theenumi}{\textbf{\roman{enumi}}}
	\begin{enumerate}
	\item Terms of the form $\smoothfunction(\vec{V})$, where $\smoothfunction$ is smooth and vanishes when 
	$\GradEnt = \Vortrenormalized \equiv 0$
	\item Terms of the form $\smoothfunction(\vec{V}) \cdot \partial \vec{V}$ where $\smoothfunction$ is smooth,
	that is, terms that depend linearly on $\partial \vec{V}$
	\item Terms of the form $\smoothfunction(\vec{V}) \mathfrak{Q}(\partial \vec{V},\partial \vec{V})$,
	where $\smoothfunction$ is smooth and 
	$\mathfrak{Q}$ is a \textbf{standard null form relative to $g$} 
	from Def.~\ref{D:NULLFORMS}
	\end{enumerate}
\end{theorem}

\begin{proof}
	It is easy to see that 
	$\mathfrak{Q}_{(v)}^i$,
	$\mathfrak{Q}_{(\Densrenormalized)}$, 
	$\mathfrak{Q}_{(\CurlofVortrenormalized)}^i$,
	and
	$\mathfrak{Q}_{(\DivofEntrenormalized)}$
	are type $\textbf{(iii)}$ terms
	and that the same is true for the products on the
	first through third lines of RHS~\eqref{E:EVOLUTIONEQUATIONFLATCURLRENORMALIZEDVORTICITY} 
	and the terms in braces on the first line of RHS~\eqref{E:TRANSPORTFLATDIVGRADENT}.
	Similarly, it is easy to see that
	$\mathfrak{L}_{(v)}^i$,
	$\mathfrak{L}_{(\Densrenormalized)}$,
	$\mathfrak{L}_{(\Vortrenormalized)}^i$,
	$\mathfrak{L}_{(\GradEnt)}^i$,
		$\mathfrak{L}_{(\Flatdiv \Vortrenormalized)}$,
		and
	$\mathfrak{L}_{(\CurlofVortrenormalized)}^i$
	are sums of terms of type $\textbf{(i)}$ and $\textbf{(ii)}$,
	while the 
	first product on RHS~\eqref{E:VELOCITYWAVEEQUATION},
	the first product on RHS~\eqref{E:RENORMALIZEDDENSITYWAVEEQUATION},
	and the second product on RHS~\eqref{E:TRANSPORTFLATDIVGRADENT}
	are , in view of Def.~\ref{D:RENORMALIZEDCURLOFSPECIFICVORTICITY}, 
	type $\textbf{(ii)}$.
\end{proof}

\section{Overview of the roles of Theorems~\ref{T:GEOMETRICWAVETRANSPORTSYSTEM} and \ref{T:STRONGNULL} in proving shock formation}
\label{S:OVERVIEWOFROLESOFTHEOREMS}
As we mentioned in Sect.~\ref{S:INTRO}, 
in forthcoming work, 
we plan to use the results of
Theorems~\ref{T:GEOMETRICWAVETRANSPORTSYSTEM} and \ref{T:STRONGNULL}
as the starting point for a proof of finite-time shock formation 
for the compressible Euler equations. 
In this section, we overview the main ideas in the proof and highlight the role that
Theorems~\ref{T:GEOMETRICWAVETRANSPORTSYSTEM} and \ref{T:STRONGNULL} play.
We plan to study a convenient open set of initial conditions in three spatial dimensions
whose solutions typically have non-zero vorticity and non-constant entropy: perturbations (without symmetry conditions) 
of simple isentropic (that is, constant entropy\footnote{Note that 
the transport equation \eqref{E:ENTROPYTRANSPORT}
implies that the entropy is constant in the maximal classical development of the data if it is constant along $\Sigma_0$.}) 
plane waves.\footnote{These simple plane waves have vanishing vorticity
and constant entropy, through their perturbations generally do not.}
We note that in our joint work \cite{jSgHjLwW2016}
on scalar wave equations in two spatial dimensions,
we proved shock formation for solutions corresponding to 
a similar set of nearly plane symmetric initial data.
The advantage of studying perturbations of simple isentropic plane waves is that 
it allows us to focus our attention on the singularity formation
without having to confront additional evolutionary phenomena
that are often found in solutions to wave-like systems.
For example, nearly plane symmetric solutions
do not exhibit wave dispersion because their dynamics are dominated by 
$1D$-type wave behavior.\footnote{In one spatial dimension, 
wave equations are essentially transport equations and thus their solutions
do not experience dispersive decay.}
In particular, our forthcoming analysis will not feature time weights or radial weights.

\subsection{Blowup for simple isentropic plane waves}
\label{SS:SIMPLEPLANEWAVESBLOWUP}
Simple isentropic plane waves are a subclass of plane symmetric solutions. By plane symmetric solutions,
we mean solutions that depend only on $t$ and $x^1$ and such that $v^2 \equiv v^3 = 0$.
To further explain simple isentropic plane wave solutions, 
we start by defining the Riemann invariants:
\begin{align} \label{E:RIEMANNINVARIANTS}
	\mathcal{R}_{\pm}
	& := v^1 \pm F(\Densrenormalized).
\end{align}
The function $F$ in \eqref{E:EVOLUTIONRIEMANNINVARIANT}
solves the following initial value problem:
\begin{align} \label{E:DENSITYFUNCTIONUSEDFORRIEMANNINVARIANTS}
	\frac{d}{d\Densrenormalized} F(\Densrenormalized) 
	&= \Speed(\Densrenormalized),
	&&
	F(\Densrenormalized=0) = 0,
\end{align}
where we are omitting the dependence of $\Speed$ on $\Ent$ since $\Ent \equiv \mbox{\upshape constant}$
and $F(\Densrenormalized=0) = 0$ is a convenient normalization condition.
As is well-known, 
in one spatial dimension,
in terms of $\mathcal{R}_{\pm}$,
the compressible Euler equations
\eqref{E:TRANSPORTDENSRENORMALIZEDRELATIVETORECTANGULAR}-\eqref{E:ENTROPYTRANSPORT}
with constant entropy are equivalent to the system
\begin{align} \label{E:EVOLUTIONRIEMANNINVARIANT}
	\uLunit \mathcal{R}_-
	& = 0,
	&&
	\Lunit \mathcal{R}_+
	= 0,
\end{align}
where
\begin{align} 
		\Lunit 
		& := \partial_t + (v^1 + \Speed) \partial_1,
		&&
		\uLunit: = \partial_t + (v^1 - \Speed) \partial_1
		\label{E:CHARVECPLANESYMMETRY}
\end{align}
are null vectorfields relative to the acoustical metric of Def.~\ref{D:NULLFORMS}.
That is, one can easily check that $g(\Lunit, \Lunit) = g(\uLunit,\uLunit) = 0$.
The initial data are $\mathcal{R}_{\pm}|_{t=0}$ 
(and the initial constant value of the entropy, which we suppress for the rest of the discussion).
A simple isentropic plane wave is a solution such that one of the Riemann invariants, say
$\mathcal{R}_-$, completely vanishes. Note that by the first equation in
\eqref{E:EVOLUTIONRIEMANNINVARIANT}, the condition $\mathcal{R}_- = 0$ is propagated
by the flow of the equations if it is verified at time $0$. 

The simple isentropic plane wave solutions described in the previous paragraph typically form a shock in finite time
via the same mechanism that leads to singularity formation in solutions to Burgers' equation.
For illustration, we now quickly sketch the argument. We assume the simple isentropic plane wave condition 
$\mathcal{R}_- \equiv 0$,
which implies that the system \eqref{E:EVOLUTIONRIEMANNINVARIANT}
reduces to 
$
\lbrace 
	\partial_1 +  f(\mathcal{R}_+) \partial_1 
\rbrace 
\mathcal{R}_+ = 0
$,
where $f$ is a smooth function determined by $F$. It can be shown that
$f$ is \emph{not} a constant-valued function 
of $\mathcal{R}_+$
except in the case of
the equation of state of a Chaplygin gas, 
which is
$
\displaystyle
p = p(\varrho)=C_0-\frac{C_1}{\varrho}$, 
where $C_0\in \mathbb R$ and $C_1>0$.
We now take a $\partial_1$ derivative of the evolution equation for $\mathcal{R}_+$
to deduce the equation
$
\lbrace 
	\partial_1 +  f(\mathcal{R}_+) \partial_1 
\rbrace  
\partial_1 \mathcal{R}_+  
= - f'(\mathcal{R}_+) (\partial_1 \mathcal{R}_+)^2
$.
Since $\mathcal{R}_+$ 
is constant along the integral curves of $\partial_1 +  f(\mathcal{R}_+) \partial_1$
(which are also known as characteristics in the present context), the above equation may be viewed as
a Riccati-type ODE for $\partial_1 \mathcal{R}_+$ along the characteristics
of the form
\[
\frac{d}{dt}
\partial_1 \mathcal{R}_+
=
k (\partial_1 \mathcal{R}_+)^2,
\]
where the constant $k$ is equal to
$- f'(\mathcal{R}_+)$ evaluated at the point on the $x^1$-axis
from which the characteristic emanates.
Thus, for initial data such that $\partial_1 \mathcal{R}_+$ and $k$ have the same (non-zero) sign
at some point along the $x^1$ axis,
the solution $\partial_1 \mathcal{R}_+$
will blow up in finite time along the corresponding characteristic,
even though $\mathcal{R}_+$ remains bounded;
this is essentially the crudest picture of the formation of a shock singularity.
Note that there is no blowup in the case of the Chaplygin gas since
$f' \equiv 0$ in that case; 
see Footnote~\ref{FN:NOSHOCKSCHAPLYGIN} for related remarks.

\subsection{Fundamental ingredients in the proof of shock formation in more than one spatial dimension}
\label{SS:SHOCKFORMATIONPROOFSUMMARY}
We can view the simple isentropic plane waves described in Subsect.~\ref{SS:SIMPLEPLANEWAVESBLOWUP}
as solutions in three spatial dimension that have symmetry. 
In our forthcoming work on shock formation in three spatial dimensions, 
we will study perturbations (without symmetry assumptions) of simple isentropic plane waves
and show that the shock formation illustrated in Subsect.~\ref{SS:SIMPLEPLANEWAVESBLOWUP} 
is stable. For technical convenience, instead of considering data on $\mathbb{R}^3$,
we will consider initial data on the spatial manifold 
\[
\Sigma := \mathbb{R} \times \mathbb{T}^2,
\]
where the factor of $\mathbb{T}^2$ corresponds 
to perturbations away from plane symmetry.
This allows us to circumvent some technical difficulties,
such as the fact that non-trivial plane wave solutions have infinite energy
when viewed as solutions in three spatial dimensions.

Although the method of Riemann invariants allows for an easy proof of shock formation
for simple isentropic plane waves, the method is not available in more than one spatial dimension. 
Another key feature of the study of shock formation in more than one spatial dimension
is that all known proofs rely on sharp estimates that provide much more information than 
the proof of blowup for simple plane waves from Subsect.~\ref{SS:SIMPLEPLANEWAVESBLOWUP}.
Therefore, in our forthcoming proof of shock formation for perturbations of simple isentropic plane waves, 
we will use the geometric formulation of the equations provided by Theorem~\ref{T:GEOMETRICWAVETRANSPORTSYSTEM}.
We will show that these equations have the right structure such that they can be incorporated
into an extended version of the paradigm for proving shock formation initiated by 
Alinhac \cites{sA1999a,sA1999b,sA2001b,sA2002}
and significantly advanced by Christodoulou \cite{dC2007}.
The most fundamental ingredient in the approaches of Alinhac and Christodoulou
is a system of \emph{geometric coordinates} 
\begin{align} \label{E:GEOMETRICCOORDINATEs}
	(t,u,\vartheta^1,\vartheta^2)
\end{align}
that are dynamically adapted to the solution. 
We denote the corresponding partial derivative vectorfields as follows:
\begin{align} \label{E:GEOMETRICCOORDINATEPARTIALDERIVATIVEVECTORFIELDS}
\left\lbrace
	\frac{\partial}{\partial t},
	\frac{\partial}{\partial u},
	\frac{\partial}{\partial \vartheta^1},
	\frac{\partial}{\partial \vartheta^2}
\right\rbrace.
\end{align}
Here, $t$ is the standard Cartesian time function
and $u$ is an eikonal function adapted to the acoustical metric. That is, $u$ solves
the following hyperbolic PDE, known as the eikonal equation:
\begin{subequations}
\begin{align} \label{E:FIRSTEIK}
	(g^{-1})^{\alpha \beta} \partial_{\alpha} u \partial_{\beta} u 
	& = 0,
	& \partial_t u & > 0,
		\\
	u|_{t=0} = 1 - x^1.
	\label{E:EIKIC}
\end{align}
\end{subequations} 
Above and throughout the rest of the article, $g$ is the acoustical metric from Def.~\ref{D:ACOUSTICALMETRIC}.
We construct the geometric torus coordinates $\vartheta^A$ 
by solving the transport equations
\begin{subequations}
\begin{align} \label{E:TORUSCOORDINATETRANSPORT}
	(g^{-1})^{\alpha \beta} \partial_{\alpha} u \partial_{\beta} \vartheta^A
	& = 0,
	\\
	\vartheta^1|_{t=0} 
	&= x^2,
	\qquad
	\vartheta^2|_{t=0}
	= x^3,
	\label{E:TORUSCOORDINATEIC}
\end{align}
\end{subequations}
where $x^2$ and $x^3$ are standard (locally defined) Cartesian coordinates on $\mathbb{T}^2$;
see Footnote~\ref{FN:CARTESIANCOORDNATESINT2CASE} regarding the Cartesian coordinates in the present context.
For various reasons, 
when differentiating the equations to obtain estimates for the solution's derivatives, 
one needs to use geometric vectorfields, described below,
rather than the partial derivative vectorfields in \eqref{E:GEOMETRICCOORDINATEPARTIALDERIVATIVEVECTORFIELDS}.
For this reason, the coordinates $(\vartheta^1,\vartheta^2)$ play only a minor role
in the analysis and we will downplay them for most of the remaining discussion.

Note that the Cartesian components $g_{\alpha \beta}$ depend on the fluid variables $\Densrenormalized$ and $v^i$
(see \eqref{E:ACOUSTICALMETRIC}) and therefore the eikonal equation \eqref{E:FIRSTEIK} is coupled to the Euler equations.
The initial conditions \eqref{E:EIKIC} are adapted to the approximate plane symmetry of the solutions under study.
The level sets of $u$ are known as the ``characteristics'' or the
``acoustic characteristics,'' and we denote them by $\mathcal{P}_u$. 
The $\mathcal{P}_u$ are null hypersurfaces relative to the 
acoustical metric $g$. As we further explain below, the intersection of the level sets of the function $u$
(viewed as an $\mathbb{R}$-valued function of the Cartesian coordinates)
corresponds to the formation of a shock singularity and the blowup of the first Cartesian
coordinate partial derivatives of the density and velocity.
$u$ is a sharp coordinate that can be used to reveal special structures in the equations
and to construct geometric objects adapted to the characteristics.
The price that one pays for the precision
is that the top-order regularity theory for $u$ is very complicated
and tensorial in nature. As we later explain, the regularity theory is especially difficult 
near the shock and leads to degenerate high-order energy estimates. 
The first use of an eikonal function in proving a global result for a nonlinear hyperbolic system
occurred in the celebrated proof \cite{dCsK1993} of the stability of the Minkowski
spacetime as a solution to the Einstein-vacuum equations.\footnote{Roughly, \cite{dCsK1993} is a small-data
global existence result for the Einstein-vacuum equations.} Eikonal functions have also played a central role
in proofs of low-regularity well-posedness for quasilinear hyperbolic equations,
most notably the recent Klainerman--Rodnianski--Szeftel proof of the bounded $L^2$
curvature conjecture \cite{sKiRjS2015}.

The paradigm for proving shock formation originating in the works \cites{sA1999a,sA1999b,sA2001b,sA2002,dC2007}
can be summarized as follows:
\begin{quote}
To the extent possible, prove ``long-time-existence-type'' estimates for the solution relative to the geometric
coordinates and then recover the formation of the shock singularity as a degeneration between the geometric 
coordinates and the Cartesian ones. In particular, prove that the solution remains many times differentiable
relative to the geometric coordinates, even though the first Cartesian coordinate partial derivatives of 
the density and velocity blow up.
\end{quote}

The most important quantity in connection with the above paradigm for proving shock formation
is the inverse foliation density.

\begin{definition}[\textbf{Inverse foliation density}]
 \label{D:UPMUDEF}
We define the inverse foliation density $\upmu > 0$ as follows:
\begin{align} \label{E:UPMUDEF}
	\upmu 
	& := 
		\frac{-1}{(g^{-1})^{\alpha \beta} \partial_{\alpha} t \partial_{\beta} u}.
\end{align}
\end{definition}
$
\displaystyle
\frac{1}{\upmu}
$ is a measure of the density of 
the characteristics $\mathcal{P}_u$
relative to the constant-time hypersurfaces $\Sigma_t$. When $\upmu$ 
vanishes, the density becomes infinite, the level sets of $u$
intersect, and, as it turns out, the first Cartesian coordinate partial
derivatives of the density and velocity blow up in finite time.
See Figure~\ref{F:FRAME} on pg.~\pageref{F:FRAME}
for a depiction of a solution for which the characteristics 
have almost intersected.
Note that by \eqref{E:INVERSEACOUSTICALMETRIC} and \eqref{E:EIKIC}, 
we have\footnote{$\upmu|_{t=0}$ depends on the data for the fluid variables.} 
$\upmu|_{t=0} \approx 1$. Christodoulou was the first to introduce $\upmu$ 
in the context of proving shock formation in more than one spatial dimension
\cite{dC2007}.
However, before Christodoulou's work, 
quantities in the spirit of $\upmu$ had been used
in one spatial dimension,
for example, by John in his proof \cite{fJ1974} of blowup
for solutions to a large class of quasilinear hyperbolic systems.
In short, to prove a shock formation result under Christodoulou's approach,
one must control the solution all the way up until the time of first
vanishing of $\upmu$.

\subsection{Summary of the proof of shock formation}
\label{SS:SUMMARYOFPROOF}
Having introduced the geometric coordinates and the inverse foliation density,
we are now ready to summarize the main ideas in the proof of shock formation
for perturbations of simple isentropic plane wave solutions to the compressible Euler
equations in three spatial dimensions with spatial topology $\Sigma = \mathbb{R} \times \mathbb{T}^2$.
For convenience, we will study solutions with very small initial data given along
a portion of the characteristic $\mathcal{P}_0$
and ``interesting'' data (whose derivatives can be large in directions transversal to the characteristics)
along a portion of $\Sigma_0 \simeq \mathbb{R} \times \mathbb{T}^2$;
see Figure~\ref{F:FRAME} below for a schematic depiction of the setup.

Given the structures revealed by Theorems~\ref{T:GEOMETRICWAVETRANSPORTSYSTEM} and \ref{T:STRONGNULL},
most of the proof is based on a framework developed in prior works, 
as we now quickly summarize.
The bulk of the framework originated in Christodoulou's groundbreaking work \cite{dC2007}
in the irrotational case, with some key contributions 
(especially the idea to rely on an eikonal function)
coming from Alinhac's earlier work \cites{sA1999a,sA1999b,sA2001b,sA2002}
on scalar wave equations. The relevance of the strong null condition 
in the context of proving shock formation was first recognized in \cites{gHsKjSwW2016,jS2016b}.
The crucial new ideas needed to handle the transport equations and the elliptic operators/estimates
originated in \cites{jLjS2016a,jLjS2016b}. 
A key contribution of the present work 
is realizing \textbf{i)} that one can gain a derivative for the entropy $\Ent$ 
and \textbf{ii)} that in the context of shock formation, 
one needs to rely on transport-$\Flatdiv$-$\Flatcurl$ estimates for the 
entropy gradient $\GradEnt$ in order to avoid uncontrollable error terms;
see Remark~\ref{R:NEEDFORGRADENTANDDIVCURL} and Step (2) below for further
discussion on this point.

We now summarize the main ideas behind our forthcoming proof of shock formation.
Most of the discussion will be at a rough, schematic level.

\begin{enumerate}
	\item \textbf{(Commutation vectorfields adapted to the characteristics).}
	With the help of the eikonal function $u$ (see Subsect.~\ref{SS:SHOCKFORMATIONPROOFSUMMARY}), 
	construct a set of geometric vectorfields
	\begin{align} \label{E:COMMSET}
			\Fullset
			& := \lbrace \Lunit, \Rad, \GeoAng_1, \GeoAng_2 \rbrace
		\end{align}
		that are adapted to the characteristics $\mathcal{P}_u$; see Figure~\ref{F:FRAME}.
		Readers may consult
		\cites{jLjS2016a,jLjS2016b,gHsKjSwW2016} for details on how to use $u$ to construct $\Fullset$.
		Here, we only note some basic properties of these vectorfields.
		The subset
		\begin{align} \label{E:TANSET}
			\Tanset
			& := \lbrace \Lunit, \GeoAng_1, \GeoAng_2 \rbrace
		\end{align}
		spans the tangent space of $\mathcal{P}_u$
		while the vectorfield $\Rad$ is \emph{transversal} to $\mathcal{P}_u$.
		$\Lunit$ is a $g$-null (that is, $g(\Lunit,\Lunit) = 0$)
		generator of the $\mathcal{P}_u$ normalized by
		$\Lunit t = 1$, while 
		$
		\displaystyle
		\Rad
		=
		\frac{\partial}{\partial u}
		+ 
		\mbox{\upshape Error} 
		$,
		where
		$
		\mbox{\upshape Error}
		$
		is a small vectorfield tangent to the co-dimension-two tori $\mathcal{P}_u \cap \Sigma_t$.
		The vectorfields 
		$\lbrace \GeoAng_1, \GeoAng_2 \rbrace$ span the tangent space of $\mathcal{P}_u \cap \Sigma_t$.
		
		\begin{center}
\begin{overpic}[scale=.35]{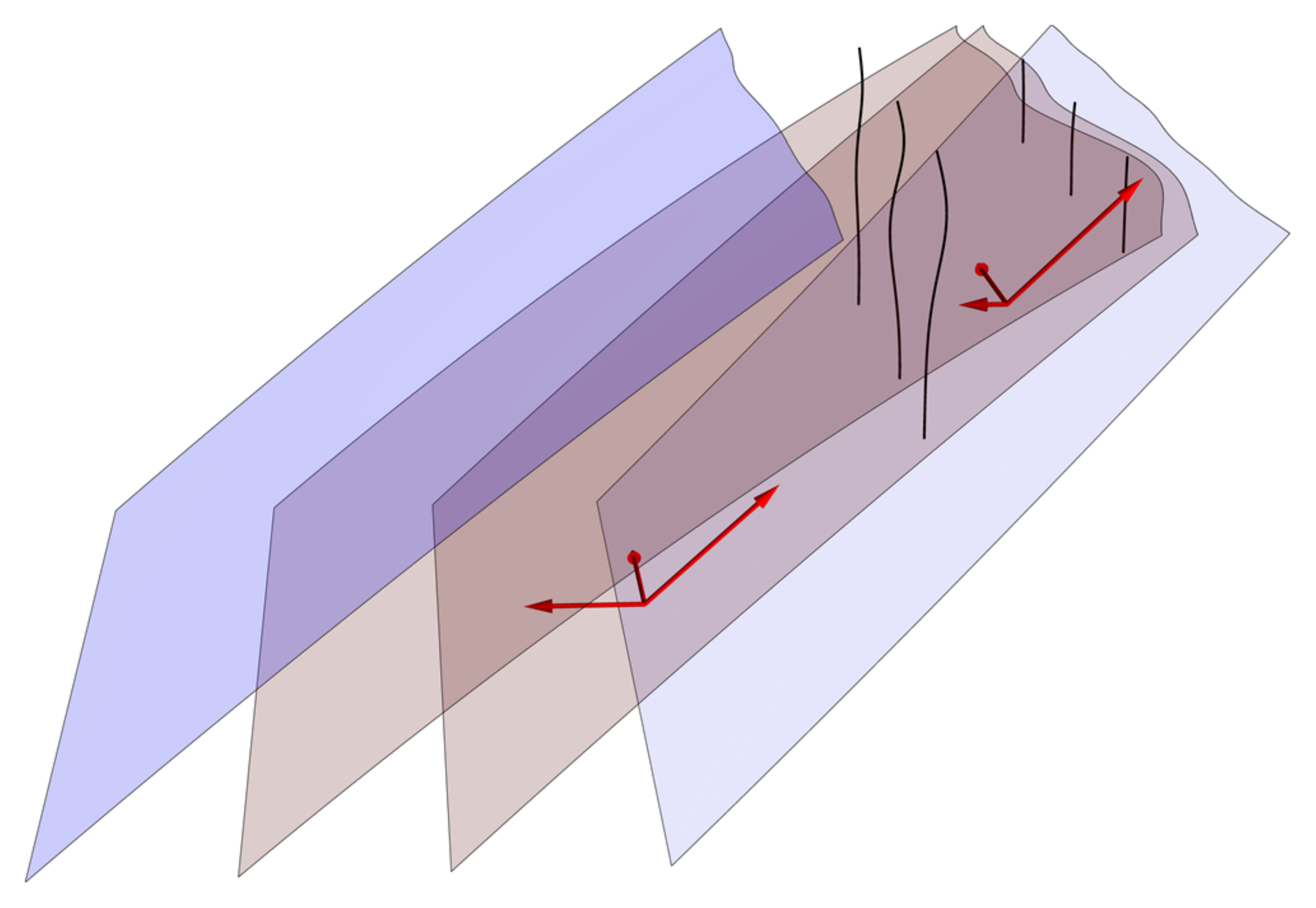} 
\put (87.1,55.6) {\large$\displaystyle \Lunit$}
\put (72,41) {\large$\displaystyle \Rad$}
\put (72.5,50) {\large$\displaystyle \GeoAng_1$}
\put (59.5,32) {\large$\displaystyle \Lunit$}
\put (39,18) {\large$\displaystyle \Rad$}
\put (46,28) {\large$\displaystyle \GeoAng_1$}
\put (51,13) {\large$\displaystyle \mathcal{P}_0^t$}
\put (37,13) {\large$\displaystyle \mathcal{P}_u^t$}
\put (7,13) {\large$\displaystyle \mathcal{P}_1^t$}
\put (22,25) {\large$\displaystyle \upmu \approx 1$}
\put (67,68) {\large$\displaystyle \upmu \ \mbox{\upshape near } 0$}
\put (15,6) {\large$\mbox{``interesting'' data}$}
\put (57,18) {\large \rotatebox{45}{$\displaystyle \mbox{very small data}$}}
\end{overpic}
\captionof{figure}{The vectorfield frame $\Fullset$ at two distinct points in 
$\mathcal{P}_u$ and the integral curves of $\Transport$,
with one spatial dimension suppressed.}
\label{F:FRAME}
\end{center}

		The elements of $\Fullset$ are designed to have good commutation properties with each other
		and also, as we describe below $\upmu \square_g$. In particular, one can show that
		we have the following
		schematic relations:\footnote{A more precise statement would indicate that
		the coefficients on RHS~\eqref{E:FULLSETFULLSETCOMMUTATIONYIELDSTANSET} depend
		on the second derivatives of the eikonal function, but we suppress this issue here.}
		\begin{align} \label{E:FULLSETFULLSETCOMMUTATIONYIELDSTANSET}
			[\Fullset, \Fullset] \sim \Tanset.
		\end{align}
		In the rest of the discussion, $Z$ denotes a generic element of $\Fullset$ and $\Singletan$
		denotes a generic element of $\Tanset$ or, more generally, a $\mathcal{P}_u$-tangent differential 
		operator.
		
		It is straightforward to derive the following relationships, which are key to understanding the shock formation, 
		where $\partial$ schematically denotes\footnote{Throughout, we use the notation $A \sim B$
		to imprecisely indicate that $A$ is well approximated by $B$.} 
		linear combinations of 
		the Cartesian coordinate partial derivative vectorfields:
		\begin{align} \label{E:GEOMETRICINTERMSOFCARTESIAN}
			\Singletan & \sim \partial, 
			&& \Rad \sim \upmu \partial.
		\end{align}
		We also note the complementary schematic relation
		\begin{align} \label{E:CARTESIANINTERMSOFGEOMETRIC}
			\upmu \partial \sim \Rad + \upmu \Singletan,
		\end{align}
		which we will refer to in Step (2).
		At the end of Step (3),
		we will clarify the role of 
		the second relation in \eqref{E:GEOMETRICINTERMSOFCARTESIAN} in tying the vanishing of 
		the inverse foliation density $\upmu$ (see Def.~\ref{D:UPMUDEF})
		to the blowup of the solution's Cartesian coordinate partial derivatives.
		In the proof of shock formation, one uses
		the elements of $\Fullset$ to differentiate 
		the equations and to obtain estimates
		for the solution's derivatives. The goal is to show
		that up to a sufficiently high order, 
		the $\Fullset$ derivatives of the solution 
		remain uniformly bounded, all the way up to the time of first vanishing of $\upmu$.
		Note that by \eqref{E:GEOMETRICINTERMSOFCARTESIAN}, we have
		$|\Singletan| = \mathcal{O}(1)$ while $|\Rad| = \mathcal{O}(\upmu)$.
		The latter estimate implies that deriving uniform bounds for the $\Rad$ derivatives of the solution
		yields only very weak information about the solution's $\mathcal{P}_u$-transversal derivatives
		in regions where $\upmu$ is small. 
		We now note that one can derive the relations
		$
		\displaystyle
		\Lunit = \frac{\partial}{\partial t}
		$,
		$
		\displaystyle
		\Rad 
			= 
			\frac{\partial}{\partial u}
			+
			\mbox{\upshape Error}	
		$,
		$
		\displaystyle
		\GeoAng_A  
		=
		\frac{\partial}{\partial \vartheta^A}
		+
		\mbox{\upshape Error}
		$,
		$A=1,2$,
		where $\mbox{\upshape Error}$ denotes small vectorfields.
		Hence, deriving estimates for the 
		$\Fullset$ derivatives of the solution is equivalent
		to deriving estimates for the derivatives of the solution relative 
		to the geometric coordinates.
		The elements of \eqref{E:COMMSET} are replacements
		for the geometric coordinate 
		partial derivative vectorfields \eqref{E:GEOMETRICCOORDINATEPARTIALDERIVATIVEVECTORFIELDS}
		that, as it turns out, enjoy better regularity properties.
		Specifically, an important point, which is not at all obvious,
		is that the elements of
		$
		\displaystyle
		\left\lbrace
			\frac{\partial}{\partial u},
			\frac{\partial}{\partial \vartheta^1},
			\frac{\partial}{\partial \vartheta^2}
		\right\rbrace
 		$,
		when commuted through the covariant wave operator $\square_g$
		from LHSs \eqref{E:VELOCITYWAVEEQUATION}-\eqref{E:RENORMALIZEDDENSITYWAVEEQUATION},
		generate error terms that lose a derivative and thus
		are uncontrollable at the top-order. 
		In contrast, the elements 
		$Z \in \Fullset$
		are adapted to the acoustical metric $g$ 
		in such a way that the commutator operator $[\upmu \square_g,Z]$ 
		generates controllable error terms.
		We note that one includes the factor of $\upmu$ in the previous commutator because 
		it leads to essential cancellations.
		However, achieving control of the commutator error terms
		at the top-order derivative level
		is difficult and in fact constitutes
		the main step in the proof.
		The difficulty is that the Cartesian components of $Z \in \Fullset$
		depend on the Cartesian coordinate partial derivatives of $u$, which we can schematically depict as follows:
		$Z^{\alpha} \sim \partial u$. Therefore, the regularity of the vectorfields $Z$ themselves
		depends on the regularity of the fluid solution through the dependence of the eikonal
		equation on the fluid variables. 
		In fact, some of the commutator terms generated by
		$[\upmu \square_g,Z]$ appear to suffer from the loss of a derivative.
		The derivative loss can be overcome using ideas originating in 
		\cites{dCsK1993,sKiR2003} and, in the context of shock formation, in \cite{dC2007}.
		However, as we explain in Step (7), one pays a steep price in overcoming the loss of a derivative:
		the only known procedure leads to degenerate estimates in which the high-order
		energies are allowed to blow up as $\upmu \to 0$. On the other hand, to close the proof
		and show that the shock forms, one must prove that the low-order energies remain
		bounded all the way up to the singularity. Establishing this hierarchy of energy estimates
		is the main technical step in the proof.
	 \item \textbf{(Multiple speeds and commuting geometric vectorfields through first-order operators).}
	 	The compressible Euler equations with entropy feature two kinds of characteristics: the acoustic characteristics
	 	$\mathcal{P}_u$ and the integral curves of the material derivative vectorfield $\Transport$;
	 	see Figure~\ref{F:FRAME}. 
	 	That is, the system features multiple characteristic speeds, which creates new difficulties compared to the case of
	 	the scalar wave equations treated in the works
	 	\cites{sA1999a,sA1999b,sA2001b,sA2002,dC2007,dCsM2014,jLjS2016a,jLjS2016b,jSgHjLwW2016,sMpY2017}. 
	 	Another new difficulty compared to the scalar wave equation case
	 	is the presence of the operators $\Flatdiv$ and $\Flatcurl$
	 	in the equations of Theorem~\ref{T:GEOMETRICWAVETRANSPORTSYSTEM}.
	 	The first proof of shock formation for a quasilinear hyperbolic
	 	system in more than one spatial dimension featuring multiple speeds
	 	and the operators $\Flatdiv$ and $\Flatcurl$
	 	was our prior work \cites{jLjS2016a,jLjS2016b} on the compressible Euler equations in the barotropic case.
	 	We now review the main difficulties corresponding to the presence multiple speeds
	 	and the operators $\Flatdiv$ and $\Flatcurl$. We will then explain how to
	 	overcome them; it turns out that essentially the same strategy can be used to handle all of these
	 	first-order operators.
	 	Since the formation of a shock is tied to the intersection of the wave characteristic $\mathcal{P}_u$
	 	(as we clarify in Step (5)),
	 	our construction of the geometric vectorfields $Z \in \Fullset$ from Step (1) was,
	 	by necessity, adapted to $g$;
	 	indeed, this seems to be the only way to ensure that the commutator terms $[\upmu \square_g,Z]$ 
	 	are controllable up to the shock.
	 	This begs the question of what kind of commutation error terms are generated
	 	upon commuting them through first-order operators such as $\Transport$, $\Flatdiv$, and $\Flatcurl$.
	 	The resolution was provided by the following key insight from \cites{jLjS2016a,jLjS2016b}:
	 	the elements of $\Fullset$ have just enough structure 
		such that their commutator with an 
		appropriately weighted, but otherwise arbitrary,
		first-order differential operator\footnote{Here, by first-order differential operator,
		we mean one equal to a regular function times a Cartesian coordinate partial derivative.} 
		produces controllable error terms, consistent with ``hiding the singularity'' 
		relative to the geometric coordinates
		at the lower derivative levels.
		Specifically, one can show that
		we have the schematic commutation relation
		\begin{align} \label{E:UPMUPARTIALZCOMMUTATOR}
			[\upmu \partial_{\alpha}, \Fullset] \sim \Rad 
			+ 
			\Singletan,
		\end{align}
		which is suggested by the schematic relations 
		\eqref{E:FULLSETFULLSETCOMMUTATIONYIELDSTANSET}
		and
		\eqref{E:CARTESIANINTERMSOFGEOMETRIC}.
		The important point is that RHS~\eqref{E:UPMUPARTIALZCOMMUTATOR}
		does not feature any singular factor of
		$
		\displaystyle
		\frac{1}{\upmu}
		$.
		
		The above discussion suggests the following strategy for treating the first-order
		equations of Theorem~\ref{T:GEOMETRICWAVETRANSPORTSYSTEM}:
		weight them with a factor of $\upmu$ 
		so that the principal part is of the schematic form $\upmu \partial$.
		Then by \eqref{E:UPMUPARTIALZCOMMUTATOR},
		upon commuting the weighted equation with elements of $\Fullset$,
		we generate only commutator terms that do not feature the
		$
		\displaystyle
		\frac{1}{\upmu}
		$.
		We stress that the property \eqref{E:UPMUPARTIALZCOMMUTATOR} does not generalize
		to second-order operators. That is, we have the schematic relation
		$
		\displaystyle
		[\upmu \partial_{\alpha} \partial_{\beta}, \Fullset] 
		\sim 
		\frac{1}{\upmu} 
		\Fullset \Fullset
		+ 
		\cdots
		$,
		which features uncontrollable factors of 
		$
		\displaystyle
		\frac{1}{\upmu}
		$.
		This is the reason that in deriving elliptic estimates 
		for the entropy $\Ent$, we work the 
		divergence and curl of the
		entropy gradient vectorfield 
		$\GradEnt^i = \partial_i \Ent$ instead
		of $\Delta \Ent$ 
		(see also Remark~\ref{R:NEEDFORGRADENTANDDIVCURL}); the $\Flatdiv$-$\Flatcurl$ formulation
		allows us to avoid commuting the elements of $\Fullset$ 
		through the (second-order) flat Laplacian $\Delta$ 
		and thus avoid the uncontrollable error terms.
		\item \textbf{($L^{\infty}$ bootstrap assumptions).}
		Formulate appropriate \emph{uniform} $L^{\infty}$ bootstrap assumptions for
		the $\Fullset$-derivatives of the solution, up to order approximately $10$,
		on a region on which the solution exists classically.
		In particular, these $\Fullset$ derivatives of the solution will not blow up,
		even as the shock forms. Just below equation \eqref{E:ONETIMECOMMUTEDNONDEGENERATEENERGY}, 
		we explain why the proof requires so many derivatives.
		The bootstrap assumptions are tensorial in nature
		and involve several parameters measuring the size of various directional
		derivatives of the solution. We will not discuss 
		the bootstrap assumptions in detail here.
		Instead, we simply note that they reflect our expectation that the solution remains 
		a small perturbation of a simple isentropic plane wave at appropriate $\Fullset$-derivative levels;
		readers may consult \cites{jLjS2016a,jLjS2016b} for more details on the bootstrap assumptions
		in the barotropic case and note that in our forthcoming work,
		we will make similar bootstrap assumptions, the new feature being 
		smallness assumptions on the departure of $\Ent$ from a constant
		and the derivatives of $\Ent$.
	\item \textbf{(The role of Theorem~\ref{T:STRONGNULL}).}
		We now clarify the importance of the good null structures revealed by Theorem~\ref{T:STRONGNULL},
		thereby fleshing out the discussion from Subsect.~\ref{SS:CONTEXT}.
		Let $\vec{V}$ denote the solution array \eqref{E:SOLUTIONARRAY}.	
		As we alluded to above, before commuting the equations
		of Theorem~\ref{T:GEOMETRICWAVETRANSPORTSYSTEM} with elements of $\Fullset$, 
		we first multiply the equations by a factor of $\upmu$.
		The main point is that by Theorem~\ref{T:STRONGNULL},
		all derivative-quadratic inhomogeneous terms (that is, the type \textbf{iii)} terms in theorem)
		in the $\upmu$-weighted equations can be decomposed in the following schematic form:
	\begin{align} \label{E:STRONGNULLDECOMPOSED}
	\upmu \partial \vec{V} \cdot \partial \vec{V} 
	= 
	\Rad \vec{V} \cdot \Singletan \vec{V}
	+
	\upmu \Singletan \vec{V} \cdot \Singletan \vec{V},
	\end{align}
	where $\Singletan$ is as in Step (1).
	The decomposition \eqref{E:STRONGNULLDECOMPOSED} 
	is precisely what is afforded by the strong null condition,
	which is available in view of Prop.~\ref{P:STANDARDNULLFORMSSATISFYSTRONGNULL}.
	The reader may have noticed that Def.~\ref{D:STRONGNULLCONDITION}
	of the strong null condition is based on decompositions relative to null frames,
	while the terms on RHS~\eqref{E:STRONGNULLDECOMPOSED} are decomposed relative to 
	the elements of $\Fullset$. That is, one needs some minor observations
	in order to translate the strong null condition into the statement \eqref{E:STRONGNULLDECOMPOSED}.
	The main idea is to consider the strong null condition 
	under a null frame \eqref{E:NULLFRAME}
	in which $\Lunit$ is the vectorfield from \eqref{E:COMMSET}
	and the vectorfields $\lbrace e_1, e_2 \rbrace$ have the same span as
	$\lbrace \GeoAng_1, \GeoAng_2 \rbrace$, in which case
	both of the sets $\lbrace \Lunit, e_1, e_2 \rbrace$
	and 
	$\lbrace \Lunit, \GeoAng_1, \GeoAng_2 \rbrace$
	span the tangent space of $\mathcal{P}_u$.
	From these considerations, it is easy to see that 
	given any derivative-quadratic term verifying 
	the strong null condition, we can decompose it 
	into factors such that each factor contains at least one $\mathcal{P}_u$-tangent differentiation,
	which is precisely what is indicated on RHS~\eqref{E:STRONGNULLDECOMPOSED}.
	In particular, on RHS~\eqref{E:STRONGNULLDECOMPOSED},
	there are no terms proportional
	to $\Rad \vec{V} \cdot \Rad \vec{V}$, which, by
	signature considerations, 
	would have to be multiplied by an \emph{uncontrollable factor of} $\displaystyle \frac{1}{\upmu}$
	that would blow up at the shock. Such a term, if present, would
	completely obstruct the goal of obtaining regular estimates for the 
	solution's low-level $\Fullset$ derivatives that hold all the way up to the shock.
	
	We have therefore explained the good structure of
	type \textbf{iii)} terms
	from Theorem~\ref{T:STRONGNULL}.
	The only other kind of inhomogeneous terms that one encounters in the 
	$\upmu$-weighted
	equations of Theorem~\ref{T:GEOMETRICWAVETRANSPORTSYSTEM}
	are at most linear in $\partial \vec{V}$, that is,
	$\upmu$-weighted versions of the type \textbf{i)} and \textbf{ii)} terms 
	from Theorem~\ref{T:STRONGNULL}.
	The linear terms
	$\upmu \partial \vec{V}$ can be decomposed (schematically) as
	\begin{align} \label{E:LINTERDECOMPOSED}
	\upmu \partial \vec{V}  
	= 
	\Rad \vec{V}
	+
	\upmu \Singletan \vec{V},
	\end{align}
	the key point being that
	RHS~\eqref{E:LINTERDECOMPOSED} does not feature any singular factor of
	$
	\displaystyle
	\frac{1}{\upmu}
	$.
	For this reason, all linear terms $\upmu \partial \vec{V}$
	remain uniformly bounded all the way up to the shock and are admissible
	within the scope of our approach.
	Similar remarks of course apply to terms that depend on $\vec{V}$ but not
	$\partial \vec{V}$.
		\item \textbf{(Tying the singularity formation to the vanishing of $\upmu$).}
		Derive the following evolution equation for $\upmu$, written in schematic form:\footnote{It turns out 
		that the coefficient of the term
		$
		\displaystyle
		\frac{\partial}{\partial u} v^1
		$
		on RHS~\eqref{E:UPMUSCHEMATICEVOLUTION} vanishes precisely in the case of the
		Chaplygin gas equation of state, which is
		$\displaystyle
		p = p(\varrho)=C_0-\frac{C_1}{\varrho}$, where $C_0\in \mathbb R$ and $C_1>0
		$. Since the term 
		$
		\displaystyle
		\frac{\partial}{\partial u} v^1
		$ is precisely the one that drives the vanishing of $\upmu$, our proof of shock formation
		does not apply for the Chaplygin gas. 
		This is connected to the following well-known fact:
		in one spatial dimension under the Chaplygin gas equation of state, 
		the compressible Euler equations form a \emph{totally linearly degenerate} 
		PDE system, which is not expected to exhibit shock formation; 
		see \cite{aM1984} for additional discussion on totally linearly degenerate PDEs.
		\label{FN:NOSHOCKSCHAPLYGIN}}
		\begin{align} \label{E:UPMUSCHEMATICEVOLUTION}
		\displaystyle
		\frac{\partial}{\partial t} \upmu
		\sim
		\frac{\partial}{\partial u} v^1
		+ 
		\mbox{\upshape Error}.
		\end{align}
		Then, using the $L^{\infty}$ bootstrap assumptions from Step (2),
		show that for data near that of a simple plane wave,
		$
		\displaystyle
		\frac{\partial}{\partial u} v^1
		$ is negative and approximately constant
		in time (relative to the geometric coordinates)
		and that $\mbox{\upshape Error}$ is small in $L^{\infty}$, all the way up to the shock.
		Thus, from \eqref{E:UPMUSCHEMATICEVOLUTION}, 
		we deduce that $\upmu$ will vanish in finite time.
		Moreover, since $\Rad v^1$ and 
		$
		\displaystyle
		\frac{\partial}{\partial u} v^1
		$
		agree up to small error terms and since
		$
		\displaystyle
		\frac{\partial}{\partial u} v^1
		$
		is strictly non-zero at any point where $\upmu$ vanishes,
		it follows from
		the second relation in \eqref{E:GEOMETRICINTERMSOFCARTESIAN}
		that
		$
		\displaystyle
		\left|
			\partial v^1
		\right|
		\sim
		\frac{1}{\upmu}
		$
		in a past neighborhood of any point where $\upmu$ vanishes.
		In particular, some Cartesian coordinate partial derivative of
		$v^1$ must blow up like 
		$
		\displaystyle
		\frac{1}{\upmu}
		$
		at points where $\upmu$ vanishes.
	\item  \textbf{(Pointwise estimates and sharp estimates for $\upmu$).}
		Commute all of the equations 
		of Theorem~\ref{T:GEOMETRICWAVETRANSPORTSYSTEM}
		up to top-order (up to approximately $20$ times)
		with the elements of $\Fullset$,
		and similarly for\footnote{Deriving estimates 
		for $\upmu$ and the $\Lunit^i$ is essentially equivalent to 
		deriving estimates for the first derivatives of the eikonal function, that is,
		for the first derivatives of solutions to the eikonal equation \eqref{E:FIRSTEIK}.
		For $\upmu$, this is apparent from equation \eqref{E:UPMUDEF}.} 
		the transport equations verified by $\upmu$ and the Cartesian components
		$\Lunit^i$, $(i=1,2,3)$. For brevity, we do not provide these transport equations in detail
		here (we schematically displayed the one for $\upmu$ in \eqref{E:UPMUSCHEMATICEVOLUTION}).
		Instead, we only note that the inhomogeneous terms in the 
		transport equations exhibit good null structures similar
		to the ones enjoyed by the simple type \textbf{i)} and type \textbf{ii)} terms
		from Theorem~\ref{T:STRONGNULL}.
		The reason that we must estimate the derivatives of $\upmu$ and $\Lunit^i$ is that
		they arise as source terms when we commute the 
		equations of Theorem~\ref{T:GEOMETRICWAVETRANSPORTSYSTEM}
		with the elements $Z$ of \eqref{E:COMMSET}.
		After commuting the equations, one uses the $L^{\infty}$ bootstrap assumptions 
		from Step (3) to derive suitable pointwise estimates for
		all of the error terms and inhomogeneous terms in the equations up to top-order.
		A key point is that all good null structures, such as the structure displayed in \eqref{E:STRONGNULLDECOMPOSED},
		are preserved under differentiations of the equations. Moreover,
		since the elements $Z \in \Fullset$ are adapted to $\upmu \square_g$,
		the commutator terms corresponding to the operator $[\upmu \square_g,Z]$ 
		also exhibit a similar good null structure.
		
		Another key step in the proof is to derive very sharp pointwise estimates
		for $\upmu$ capturing exactly how it vanishes. More precisely,
		through a detailed study of equation \eqref{E:UPMUSCHEMATICEVOLUTION},
		one can show that for the solutions under study,
		$
		\displaystyle
		\frac{\partial}{\partial t} \upmu
		$
		is \emph{quantitatively negative} 
		in regions where $\upmu$ is near $0$, which implies that $\upmu$ vanishes linearly.
		It turns out that these facts are crucial for closing the energy estimates.
		\item \textbf{(Energy estimates).}
		Using the pointwise estimates and the sharp estimates for $\upmu$
		from Step (6), 
		derive energy estimates up to top-order. This is the main technical step in the proof.
		Null structures such as 
		\eqref{E:STRONGNULLDECOMPOSED} are again critically 
		important for the energy estimates, since our energies 
		(described below)
		are designed to control error integrals that are generated by terms of the form RHS~\eqref{E:STRONGNULLDECOMPOSED} 
		and their higher-order analogs. To control some of the terms in the energy estimates,
		we also need elliptic estimates along $\Sigma_t$, which we describe in Step (8). 
		As a preliminary step,
		we now briefly describe, from the point of view of regularity, 
		why our proof fundamentally relies on 
		the equations \eqref{E:FLATDIVOFRENORMALIZEDVORTICITY}-\eqref{E:DIVGRADENT}
		and elliptic estimates.
		In reality, we need elliptic estimates
		only to control the solution's top-order derivatives, 
		that is, after commuting the equations many times with the elements of $\Fullset$.
		However, for convenience, 
		here we ignore the need to commute the equations
		and instead focus our discussion on how to derive a consistent 
		amount of Sobolev regularity for solutions
		to the non-commuted equations.
		In proving shock formation, we are primarily interested 
		in deriving estimates for solutions to the wave equations
		\eqref{E:VELOCITYWAVEEQUATION}-\eqref{E:RENORMALIZEDDENSITYWAVEEQUATION};
		given suitable estimates for their solutions, 
		the rest of the proof of the formation of the shock is relatively easy.
		To proceed, we first note that the inhomogeneous terms
		$\CurlofVortrenormalized$
		and
		$\DivofEntrenormalized$
		on the right-hand sides of the wave equations 
		\eqref{E:VELOCITYWAVEEQUATION}-\eqref{E:RENORMALIZEDDENSITYWAVEEQUATION}
		are (see Def.~\ref{D:RENORMALIZEDCURLOFSPECIFICVORTICITY}), from the point of view of regularity, 
		at the level of $\partial \Vortrenormalized$ and $\partial \GradEnt$
		plus easier terms that can be treated using energy estimates for wave equations
		(and that we will therefore ignore in the present discussion).
		On the other hand, 
		the transport equations
		\eqref{E:RENORMALIZEDVORTICTITYTRANSPORTEQUATION} and \eqref{E:GRADENTROPYTRANSPORT}
		for $\Vortrenormalized$ and $\GradEnt$ 
		have source terms that depend on
		$\partial v$ and $\partial \Densrenormalized$.
		This \emph{falsely} suggests that 
		$\Vortrenormalized$ and $\GradEnt$ have the same Sobolev regularity
		as $\partial v$ and $\partial \Densrenormalized$ which,
		from the point of view of regularity, would be inconsistent with 
		the inhomogeneous terms $\partial \Vortrenormalized$ and $\partial \GradEnt$ on 
		the right-hand side of the wave equations; the inconsistency would come from the fact 
		that energy estimates for the wave equations yield control only over
		$\partial v$ and $\partial \Densrenormalized$ and thus
		$\partial v$ and $\partial \Densrenormalized$ cannot have more $L^2$ regularity
		than the wave equation source terms 
		$\partial \Vortrenormalized$ and $\partial \GradEnt$.
		To circumvent this difficulty, one needs to rely on
		the $\Flatdiv$-$\Flatcurl$-transport-type equations
		\eqref{E:FLATDIVOFRENORMALIZEDVORTICITY}-\eqref{E:DIVGRADENT}
		and elliptic estimates to control $\partial \Vortrenormalized$ and $\partial \GradEnt$ 
		in $L^2(\Sigma_t)$, using only that
		$\partial v^i$ and $\partial \Densrenormalized$
		are in $L^2(\Sigma_t)$. We further explain this in Step (8).
		A key reason behind the viability of this approach is that
		even though equations \eqref{E:FLATDIVOFRENORMALIZEDVORTICITY}-\eqref{E:DIVGRADENT}
		are obtained by differentiating the transport equations
		\eqref{E:RENORMALIZEDVORTICTITYTRANSPORTEQUATION} and \eqref{E:GRADENTROPYTRANSPORT}
		(which feature inhomogeneous terms of the schematic form $\partial v$),
		\emph{the inhomogeneous terms on RHSs \eqref{E:FLATDIVOFRENORMALIZEDVORTICITY}-\eqref{E:DIVGRADENT}
		do not feature the terms $\partial^2 v$ or $\partial^2 \Densrenormalized$};
		this is a surprising structural feature of the equations that should not be taken for granted.
		
		The main difficulty that one encounters in the proof of shock formation
		is that the best energy estimates 
		that we know how to derive allow for the possibility that
		the high-order energies might blow up as the shock forms. 
		This makes it difficult to justify the uniform (non-degenerate) $L^{\infty}$
		bootstrap assumptions from Step (3), which play a crucial role in showing
		that the shock forms and in deriving the pointwise estimates from Step (6).
		It turns out that the maximum possible energy blowup rates can be
		expressed in terms of negative powers of
		\begin{align} \label{E:MUSTAR}
			\upmu_{\star}(t)
			& := \min_{\Sigma_t} \lbrace 1, \upmu \rbrace.
		\end{align}
	Note that the formation of the shock corresponds to $\upmu_{\star} \to 0$.
	Just below, we will roughly describe the hierarchy of energy estimates.
	The energy estimates involve energies 
	$\mathbb{E}_{(Wave);Top}$
	for the 
	``wave variables''
	$\lbrace 
			\Densrenormalized, v^1, v^2, v^3
	\rbrace
	$
	as well as energies
	$\mathbb{E}_{(Transport)}$
	for the ``transport variables''
	$\lbrace 
	\Ent, 
	\Vortrenormalized, 
	\GradEnt^1, \GradEnt^2, \GradEnt^3, 
	\CurlofVortrenormalized^1, \CurlofVortrenormalized^2, \CurlofVortrenormalized^3,
	\DivofEntrenormalized
\rbrace$.
	We use the notation $\mathbb{E}_{(Wave);Top}$ to denote a wave energy that
	controls the top-order $\Fullset$ derivatives\footnote{Actually, it is possible 
	(desirable even) to close the proof by deriving energy estimates only for the $\Tanset$-commuted equations;
	we will ignore this technical detail for the rest of the discussion.} 
	of the wave variables
	(here we are not specific about how many derivatives correspond to top-oder),
	$\mathbb{E}_{(Wave);Top-1}$ denote a just-below-top order wave energy,
	$\mathbb{E}_{(Wave);Mid}$ denote a mid-order wave energy
	(we also are not specific about how many derivatives correspond to mid-order),
	$\mathbb{E}_{(Wave);1}$ correspond to the energy after a single commutation,\footnote{It turns out that
 	we can avoid relying on energies corresponding to zero commutations.}
	and similarly for the transport equation energies.
	The hierarchy of energy estimates that one can derive roughly has the following structure,
	where $K \approx 20$ is a constant and 
	$\mathring{\upepsilon}$ is a small parameter representing the size of a seminorm that, 
	roughly speaking, 
	measures how far the initial data are from 
	the data of a simple isentropic plane wave: 
	\begin{subequations}
	\begin{align}
	\mathbb{E}_{(Wave);Top}(t),
	\,
	\mathbb{E}_{(Transport);Top}(t)
	& \lesssim \mathring{\upepsilon}^2 \upmu_{\star}^{-K}(t),
		\label{E:TOPDEGENERATEENERGY} \\
	\mathbb{E}_{(Wave);Top-1}(t),
	\,
	\mathbb{E}_{(Transport);Top-1}(t)
	& \lesssim \mathring{\upepsilon}^2 \upmu_{\star}^{2-K}(t),
		\label{E:JUSTBELOWTOPDEGENERATEENERGY} \\
	\mathbb{E}_{(Wave);Top-2}(t),
	\,
	\mathbb{E}_{(Transport);Top-2}(t)
	& \lesssim \mathring{\upepsilon}^2 \upmu_{\star}^{4-K}(t),
		\label{E:TWOBELOWTOPDEGENERATEENERGY} \\
	& \cdots,
		\\
	\mathbb{E}_{(Wave);Mid}(t),
	\,
	\mathbb{E}_{(Transport);Mid}(t)
	& \lesssim \mathring{\upepsilon}^2,
		\label{E:MIDNONDEGENERATEENERGY} 
		\\
	\mathbb{E}_{(Transport);1}(t)
	& \lesssim \mathring{\upepsilon}^2.
		\label{E:ONETIMECOMMUTEDNONDEGENERATEENERGY} 
\end{align}
\end{subequations}
		
		The difficult parts of the proof are controlling the maximum possible the top-order blowup rate
		$\upmu_{\star}^{-K}(t)$ as well as establishing the descent scheme 
		showing that the below-top-order energies become successively less degenerate
		until one reaches the level \eqref{E:MIDNONDEGENERATEENERGY}, 
		below which the energies do not blow up.
		Descent schemes of this type originated in the works
		\cites{sA1999a,sA1999b,sA2001b,sA2002,dC2007} of Alinhac and Christodoulou
		and have played a key role in all prior works on shock formation in more than one spatial dimension.
		From the non-degenerate energy estimates
		\eqref{E:MIDNONDEGENERATEENERGY}-\eqref{E:ONETIMECOMMUTEDNONDEGENERATEENERGY},
		Sobolev embedding, and a smallness assumption on the data-size parameter $\mathring{\upepsilon}$,
		one can justify (that is, improve) the non-degenerate $L^{\infty}$
		bootstrap assumptions from Step (3).
		To close the proof, we need the energies to remain uniformly bounded up the singularity 
		starting at a level
		representing, roughly, slightly more than half of the top-order number of derivatives.
		Consequently, the proof requires a lot of regularity, and top-order corresponds
		to commuting the equations roughly $20$ times with the elements of $\Fullset$.
		The precise numerology behind the hierarchy 
		\eqref{E:TOPDEGENERATEENERGY}-\eqref{E:ONETIMECOMMUTEDNONDEGENERATEENERGY} 
		is complicated, but the following two features seem fundamental:
		\textbf{i)} The top-order blowup rate $\upmu_{\star}^{-K}(t)$, since,
		as we explain below, the blowup-exponent $K$ is tied to universal structural
		constants in the equations that are independent of the number of times that we commute them.
		\textbf{ii)} An improvement of precisely $\upmu_{\star}^2(t)$ at each step in the descent,
		which is tied to the fact that $\upmu_{\star}(t)$ vanishes linearly
		(as we mentioned at the end of Step (6)).
		 
		To construct energies that result in controllable error terms,
		we must weight various terms in it with factors of $\upmu$,
		a difficulty that lies at the heart of the analysis.
		For example, the energies $\mathbb{E}_{(Wave)}$
		for the wave variables 
		$\Psi \in 
		\lbrace 
			\Densrenormalized, v^1, v^2, v^3
		\rbrace$
		are constructed\footnote{With the help of the vectorfield multiplier method,
		based on the energy-momentum tensor for wave equations and the multiplier
		vectorfield $\Mult := (1 + 2 \upmu) \Lunit + 2 \Rad$; see \cites{jSgHjLwW2016,jLjS2016a,jLjS2016b}.} 
		so that, at the level of the undifferentiated equations, 
		we have, relative to the geometric coordinates \eqref{E:GEOMETRICCOORDINATEs}
		and the vectorfields in \eqref{E:COMMSET}, the
		following schematic relation:
		\begin{align}\label{E:WAVEENERGIESSCHEMATICSTRENGTH}
		\mathbb{E}_{(Wave)}[\Psi](t)
		& \sim  
			\int_{\Sigma_t} 
				\left\lbrace 
					(\Rad\Psi)^2 
					+ 
					\upmu 
					\left(
						(\Lunit \Psi)^2
						+
						(\GeoAng_1 \Psi)^2
						+
						(\GeoAng_2 \Psi)^2
					\right)\right\rbrace 
			\, d \vartheta^1 d \vartheta^2 du.
\end{align}
The energy
$\mathbb{E}_{(Wave);Top}(t)$ 
on LHS~\eqref{E:TOPDEGENERATEENERGY}
schematically represents one of the quantities
$\mathbb{E}_{(Wave)}[\Fullset^{N_{top}} \Psi](t)$,
where $N_{top} \approx 20$ is the maximum number of times 
that we need to commute the equations
in order to close the estimates.
The factor of $\upmu$ in \eqref{E:WAVEENERGIESSCHEMATICSTRENGTH} is chosen so that only controllable
error terms are generated in the energy identities 
(it is true, though not obvious, that RHS~\eqref{E:WAVEENERGIESSCHEMATICSTRENGTH} has the right strength).
Note that some components of the energies become very weak near the shock
(that is, in regions where $\upmu$ is small),
namely the products on RHS~\eqref{E:WAVEENERGIESSCHEMATICSTRENGTH} that 
are $\upmu$-weighted.
This makes it difficult to control the
non-$\upmu$-weighted
error terms that one encounters 
in the energy identities.
To control such ``strong'' error terms, one uses, 
in addition to the energies \eqref{E:WAVEENERGIESSCHEMATICSTRENGTH},
energies along $\mathcal{P}_u$ (known as null fluxes) as well as 
a coercive friction-type spacetime integral, which is
available because 
$
\displaystyle
\frac{\partial}{\partial t} \upmu
$
is quantitatively negative in the difficult region where $\upmu$ is small (as we described in Step (6)).
These aspects of the proof, though of fundamental importance,
have been well-understood since Christodoulou's work \cite{dC2007} 
and are described in more detail in \cites{jSgHjLwW2016,jLjS2016a,jLjS2016b};
for this reason, we will not further discuss these issues here.

We must also derive energy estimates for the transport equations
in Theorem~\ref{T:GEOMETRICWAVETRANSPORTSYSTEM}.
Specifically, to control the transport variables 
$\Psi \in 
\lbrace 
	\Ent, 
	\Vortrenormalized, 
	\GradEnt^1, \GradEnt^2, \GradEnt^3, 
	\CurlofVortrenormalized^1, \CurlofVortrenormalized^2, \CurlofVortrenormalized^3,
	\DivofEntrenormalized
\rbrace$,
we rely on energies with the following strength:
\begin{align}\label{E:TRANSPORTENERGIESSCHEMATICSTRENGTH}
		\mathbb{E}_{(Transport)}[\Psi](t)
		& \sim  
			\int_{\Sigma_t} 
				\upmu \Psi^2
			\, d \vartheta^1 d \vartheta^2 du.
\end{align}
As in the case of the wave variable energies,
the factor of $\upmu$ in \eqref{E:TRANSPORTENERGIESSCHEMATICSTRENGTH} is chosen 
so that only controllable
error terms are generated in the energy identities.

We now sketch the main ideas behind why the top-order energy estimate
\eqref{E:TOPDEGENERATEENERGY} is so degenerate. 
We will focus only on the wave equation energy estimates since the transport
equation energy estimates are much easier to derive.\footnote{It turns out, however, 
that the $g$-timelike nature of the transport operator $\Transport$
(as shown by Lemma~\ref{L:BASICPROPERTIESOFTRANSPORT})
is important for the transport equation energy estimates; see \cite{jLjS2016a} for further discussion on this point.}
The basic difficulty
is that on the basis of energy identities, 
the following integral inequality is the best that we are able to obtain:
\begin{align} \label{E:TOPENERGYGRONWALLREADY}
	\mathbb{E}_{(Wave);Top}(t)
	& 
	\leq C \mathring{\upepsilon}^2
	+
	A
	\int_{s=0}^t
			\sup_{\Sigma_s} 
			\left|
				\frac{\frac{\partial}{\partial s} \upmu}{\upmu}
			\right|
		\mathbb{E}_{(Wave);Top}(s)
	\, ds
	+
	\cdots,
\end{align}
where $A$ is a universal positive constant that is
\textbf{independent of the equation of state and the number of times that the equations are commuted}
and $\cdots$ denotes similar or less degenerate error terms.
Below we explain the origin of the degenerate factor
$
\displaystyle
\left|
	\frac{\frac{\partial}{\partial s} \upmu}{\upmu}
\right|
$
on RHS~\eqref{E:TOPENERGYGRONWALLREADY}, whose presence is tied
to an issue that we highlighted earlier: 
the needed top-order regularity properties of the eikonal function are difficult to derive.
To apply Gronwall's inequality to the inequality \eqref{E:TOPENERGYGRONWALLREADY},
we need the following estimate:
\begin{align} \label{E:KEYGRONWALLFACTORINTEGRALBOUND}
\int_{s=0}^t
	\sup_{\Sigma_s} 
	\left|
		\frac{\frac{\partial}{\partial s} \upmu}{\upmu}
	\right|
\, ds
\sim 
\left| \ln \upmu_{\star}^{-1} \right|.
\end{align}
The proof of \eqref{E:KEYGRONWALLFACTORINTEGRALBOUND}
can be derived with the help
the estimates
\begin{align}
	\upmu_{\star}(t) 
	& \sim 1 - \TrandatasizeWithFactor t,
		\label{E:UPMUSTARSHARP}
		\\
	\left\|
		\frac{\partial}{\partial t} \upmu
	\right\|_{L^{\infty}(\Sigma_t)}
	& \sim \TrandatasizeWithFactor,
	\label{E:LUNITUPMUSTARSHARP}
\end{align}
where $\TrandatasizeWithFactor > 0$ is a data-dependent parameter
that, roughly speaking, measures 
the $L^{\infty}$ size of the 
term 
$
\displaystyle
\frac{\partial}{\partial u} v^1
$
on RHS~\eqref{E:UPMUSCHEMATICEVOLUTION}.
We note that to close the proof, 
one needs to consider initial data such that $\mathring{\upepsilon}$ is small relative to $\TrandatasizeWithFactor$
(though $\TrandatasizeWithFactor$ may be small in an absolute sense).
We also note that the estimates \eqref{E:UPMUSTARSHARP}-\eqref{E:LUNITUPMUSTARSHARP}
fall under the scope of the sharp estimates for $\upmu$ from Step (6).
Moreover, we note that the fact that $\upmu_{\star}$ vanishes linearly
is important for deriving \eqref{E:KEYGRONWALLFACTORINTEGRALBOUND}.
Finally, we note that 
\eqref{E:KEYGRONWALLFACTORINTEGRALBOUND} is just a quasilinear version of
the caricature estimate
$
\displaystyle
\int_{s=t}^1 
	\frac{1}{s}
\, ds
\lesssim \ln t
$,
in which $s=0$ represents the time of first vanishing of $\upmu_{\star}$ and $s=1$ represents 
the ``initial'' data time.

After we have derived \eqref{E:TOPENERGYGRONWALLREADY}
and
\eqref{E:KEYGRONWALLFACTORINTEGRALBOUND},
we can apply Gronwall's inequality (ignoring the terms $\cdots$ on RHS~\eqref{E:TOPENERGYGRONWALLREADY})
to obtain the following bound:
\begin{align} \label{E:TOPENERGYGRONWALLED}
	\mathbb{E}_{(Wave);Top}(t)
	& 
	\leq C 
	\mathring{\upepsilon}^2 
	\upmu_{\star}^{-A}.
\end{align}
The bound \eqref{E:TOPENERGYGRONWALLED}
is essentially the top-order energy estimate \eqref{E:TOPDEGENERATEENERGY}.
However, in reality, the blowup-exponent $K$ on RHS~\eqref{E:TOPDEGENERATEENERGY}
is larger than the blowup-exponent $A$ on RHS~\eqref{E:TOPENERGYGRONWALLED}
because the correct estimate \eqref{E:TOPDEGENERATEENERGY} is influenced by additional
degenerate error terms that we have ignored in deriving \eqref{E:TOPENERGYGRONWALLED}.

We now briefly explain the origin of the difficult error integral on
RHS~\eqref{E:TOPENERGYGRONWALLREADY}. 
Let $\Psi$ schematically denote any of the wave variables
$\lbrace 
	\Densrenormalized, v^1, v^2, v^3
\rbrace
$.
The difficulty arises from the worst commutator error terms that are generated when one commutes
the elements of $\Fullset$ 
(see \eqref{E:COMMSET})
through the wave operator $\upmu \square_g$.
To explain the main ideas, we consider only the wave equation verified by
$\GeoAng^N \Psi$, where $\GeoAng^N$ schematically denote an order $N$ differential
operator corresponding to repeated differentiation with
respect to elements of the set $\lbrace \GeoAng_1, \GeoAng_2 \rbrace$;
similar difficulties arise upon commuting $\upmu \square_g$ with other strings of vectorfields from $\Fullset$.
Specifically, one can show that upon commuting any of the 
$\upmu$-weighted wave equations
\eqref{E:VELOCITYWAVEEQUATION}-\eqref{E:RENORMALIZEDDENSITYWAVEEQUATION}
with $\GeoAng^N$, we obtain an inhomogeneous wave equation of the schematic form
\begin{align}  \label{E:NTIMESCOMMUTEDWAVEMAINTRCHITERM}
	\upmu \square_g \GeoAng^N \Psi
	& =  (\Rad \Psi) \GeoAng^N \mytr \upchi
		+ 
		\cdots.
\end{align}
The term $\upchi$ on RHS~\eqref{E:NTIMESCOMMUTEDWAVEMAINTRCHITERM}
is the null second fundamental form of the co-dimension-two tori
$\mathcal{P}_u \cap \Sigma_t$,
that is, the symmetric type $\binom{0}{2}$
tensorfield with components
$\upchi_{\CoordAng_A \CoordAng_B} = g(\D_{\CoordAng_A} \Lunit, \CoordAng_B)$,
where $\D$ is the Levi-Civita connection of $g$.
Moreover, $\mytr \upchi$ is the trace of $\upchi$
with respect to the Riemannian metric $\gsphere$ induced on $\mathcal{P}_u \cap \Sigma_t$ by $g$.
Geometrically, $\mytr \upchi$ 
is the null mean curvature of the $g$-null hypersurfaces $\mathcal{P}_u$.
Analytically, $\GeoAng^N \mytr \upchi$ is a difficult commutator term in which the
maximum possible number of derivatives falls on 
the eikonal function
(recall that $\Lunit \sim \partial u$ and thus $\upchi \sim \partial^2 u$).
As we mentioned earlier, the main difficulty is that a naive
treatment of terms involving the maximum number of derivatives of the eikonal function
leads to the loss of a derivative.
This difficulty is visible directly from the evolution equation 
satisfied by $\GeoAng^N \mytr \upchi$,
which can be derived from geometric considerations\footnote{The precise version of equation
\eqref{E:RAYCHEVOLUTIONS} that one needs in a detailed proof is essentially
Raychaudhuri's equation for the ``$\Lunit \Lunit$'' component of the Ricci curvature of the 
acoustical metric $g$.}
and which takes the following schematic form
(recall that $\Singletan$ schematically denotes elements of the set \eqref{E:TANSET}):
\begin{align} \label{E:RAYCHEVOLUTIONS}
	\frac{\partial}{\partial t}
	\GeoAng^N \mytr \upchi
	& = 
	\angLap \GeoAng^N \Psi
	+
	\frac{\partial}{\partial t} \Singletan \GeoAng^N \Psi
	+ 
	l.o.t,	
\end{align}
where $\angLap$ is the covariant Laplacian induced on
induced on $\mathcal{P}_u \cap \Sigma_t$ by $g$
and $l.o.t.$ denotes terms involving $\leq 1$ derivative of $\Psi$.
The difficulty with equation \eqref{E:RAYCHEVOLUTIONS} 
is that the two explicitly displayed terms
on RHS~\eqref{E:RAYCHEVOLUTIONS} depend on
$N+2$ derivatives of $\Psi$, which is one more 
than we can control by energy estimates for the wave equation \eqref{E:NTIMESCOMMUTEDWAVEMAINTRCHITERM}.
That is, the two terms on RHS~\eqref{E:RAYCHEVOLUTIONS} seem to lose a derivative.
To overcome this difficulty for the second term
$
\displaystyle
\frac{\partial}{\partial t} \Singletan \GeoAng^N \vec{\Psi}
$,
we can simply bring it over to the left so that equation \eqref{E:RAYCHEVOLUTIONS} becomes
$
\displaystyle
\frac{\partial}{\partial t}
\left\lbrace
	\GeoAng^N \mytr \upchi
	-
	\Singletan \GeoAng^N \vec{\Psi}
\right\rbrace
=
\angLap \GeoAng^N \Psi
	+ 
	l.o.t
$.
To handle the term
$\angLap \GeoAng^N \Psi$,
we can use a similar but more complicated 
strategy first employed in \cite{sKiR2003} in the
context of low regularity well-posedness and later by Christodoulou
\cite{dC2007} in the context of shock formation:
by decomposing the principal parts of the $\GeoAng^N$-commuted wave equations
\eqref{E:VELOCITYWAVEEQUATION}-\eqref{E:RENORMALIZEDDENSITYWAVEEQUATION},
we can obtain the following algebraic relation, 
written in schematic form:
$
\displaystyle
\upmu \angLap \GeoAng^N \Psi
=
\frac{\partial}{\partial t}
Z \GeoAng^N \Psi
+
l.o.t.
$,
where $Z \in \Fullset$.
Bringing the perfect time derivative 
term
$
\displaystyle
\frac{\partial}{\partial t}
Z \GeoAng^N \Psi
$,
over to LHS~\eqref{E:RAYCHEVOLUTIONS} as well,
we obtain
\begin{align} \label{E:RENORMALZIEDRAYCHEVOLUTIONS}
	\frac{\partial}{\partial t}
	\left\lbrace
		\upmu \GeoAng^N \mytr \upchi
		-
		Z \GeoAng^N \Psi
		-
		\upmu \Singletan \GeoAng^N \Psi
	\right\rbrace
	& = 
	l.o.t.
\end{align}
The key point is that all inhomogeneous terms on RHS~\eqref{E:RENORMALZIEDRAYCHEVOLUTIONS}
now feature an allowable amount of regularity,\footnote{In reality, in three or more spatial dimensions,
there remain some additional terms on RHS~\eqref{E:RENORMALZIEDRAYCHEVOLUTIONS}
that depend on the top-order derivatives of the eikonal function. These terms are schematically of the form
of the top-order derivatives of the trace-free part of $\upchi$, traditionally denoted by $\hat{\upchi}$
(note that $\hat{\upchi} \equiv 0$ in two spatial dimensions).
From the prior discussion, one might think that these terms result in the loss of a derivative and obstruct
the closure of the energy estimates. However, it turns out that one can 
avoid the derivative loss for $\hat{\upchi}$ by exploiting 
geometric Codazzi-type identities and elliptic estimates
on the co-dimension-two tori $\mathcal{P}_u \cap \Sigma_t$. Such elliptic estimates for $\hat{\upchi}$ 
have been well-understood since \cite{dCsK1993} and, in the context of shock formation, since \cite{dC2007}. For this reason, we
do not further discuss this technical issue here.} 
which implies that we can gain back the 
derivative by working with the ``modified'' quantity
\begin{align} \label{E:TRCHIMOD}		
\upmu \GeoAng^N \mytr \upchi
		-
		Z \GeoAng^N \Psi
		-
		\upmu \Singletan \GeoAng^N \Psi.
\end{align}
We have therefore explained how to avoid the derivative loss 
that was threatened by the term $\GeoAng^N \mytr \upchi$
on RHS~\eqref{E:NTIMESCOMMUTEDWAVEMAINTRCHITERM}.
However, our approach comes with a large price: 
the inhomogeneous term on RHS~\eqref{E:NTIMESCOMMUTEDWAVEMAINTRCHITERM}
involves the factor $\GeoAng^N \mytr \upchi$,
while \eqref{E:RENORMALZIEDRAYCHEVOLUTIONS}
yields an evolution equation only for the modified version 
of $\upmu \GeoAng^N \mytr \upchi$ stated in \eqref{E:TRCHIMOD}; 
this discrepancy factor of 
$\upmu$ is what leads to the dangerous factor of
$
\displaystyle
\frac{1}{\upmu}
$
on RHS~\eqref{E:TOPENERGYGRONWALLREADY}.
Moreover, from a careful analysis that takes into
account the evolution equation for $\upmu$
as well as the precise structure of the factor $\Rad \Psi$
on RHS~\eqref{E:NTIMESCOMMUTEDWAVEMAINTRCHITERM}
and the terms on LHS~\eqref{E:RENORMALZIEDRAYCHEVOLUTIONS},
one can deduce the presence of the factor
$
\displaystyle
\frac{\partial}{\partial s} \upmu
$
on RHS~\eqref{E:TOPENERGYGRONWALLREADY},
whose precise form is important in the proof of 
the estimate \eqref{E:KEYGRONWALLFACTORINTEGRALBOUND}.
We have therefore explained the main ideas behind the origin of the
main error integral displayed on RHS~\eqref{E:TOPENERGYGRONWALLREADY}.

Having provided an overview of the derivation of the top-order energy estimate \eqref{E:TOPDEGENERATEENERGY},
we now describe why the below-top-order energies become successively less singular
as one descends below top-order, that is, how to implement the energy estimate descent scheme
resulting in the estimates \eqref{E:JUSTBELOWTOPDEGENERATEENERGY}-\eqref{E:ONETIMECOMMUTEDNONDEGENERATEENERGY}; 
recall that the non-degenerate energy estimates \eqref{E:MIDNONDEGENERATEENERGY}-\eqref{E:ONETIMECOMMUTEDNONDEGENERATEENERGY} 
are needed to improve, by Sobolev embedding and a small-data assumption, 
the $L^{\infty}$ bootstrap assumptions from Step (2), which are central to the whole process.
A key ingredient in the energy estimate descent scheme is the following estimate, valid for constants $b > 0$, 
which shows that integrating the singularity in time reduces its strength:
\begin{align} \label{E:REDUCETHESTRENTHOFTHESINGULARITY}
	\int_{s=0}^t
			\upmu_{\star}^{-b}(s)
		\, ds
		\lesssim
		\upmu_{\star}^{1-b}(t).
\end{align}
The estimate \eqref{E:REDUCETHESTRENTHOFTHESINGULARITY}
is easy to obtain thanks to the sharp information that we have about the linear vanishing rate
of $\upmu_{\star}$ (see \eqref{E:UPMUSTARSHARP}).
We note that \eqref{E:REDUCETHESTRENTHOFTHESINGULARITY} is just a quasilinear version of the estimate 
$\int_{s=t}^1 s^{-b} \, ds \lesssim t^{1 - b}$ for $0 < t < 1$, 
where $s=0$ represents the vanishing of $\upmu_{\star}$.
A second key ingredient in implementing the descent scheme is to exploit that
below top-order, we can estimate the difficult term $\GeoAng^N \mytr \upchi$ on RHS~\eqref{E:NTIMESCOMMUTEDWAVEMAINTRCHITERM}
in a different way; recall that this term was the main driving force behind the degenerate top-order
energy estimates. Specifically, for $N$ below top-order,  
we can directly estimate $\GeoAng^N \mytr \upchi$ by integrating the transport equation
\eqref{E:RAYCHEVOLUTIONS} in time, without going through the procedure that
led to equation \eqref{E:RENORMALZIEDRAYCHEVOLUTIONS} in the top-order case.
This approach results in a loss of one derivative caused by the two explicitly displayed
terms on RHS~\eqref{E:RAYCHEVOLUTIONS} and therefore couples the below-top-order energy estimates
to the top-order ones. However, the integration in time allows one to employ the estimate
\eqref{E:REDUCETHESTRENTHOFTHESINGULARITY}, which implies that 
below top-order, $\GeoAng^N \mytr \upchi$ is less singular than RHS~\eqref{E:RAYCHEVOLUTIONS};
this is the crux of the descent scheme.
We also note that this procedure allows one to avoid the difficult factor of
$\upmu$, which in the top-order case appeared on LHS~\eqref{E:RENORMALZIEDRAYCHEVOLUTIONS} 
and which drove the blowup-rate of the top-order energies.

We have thus explained one step in the descent. One can continue the descent, noting
that at each stage, we can directly estimate 
the difficult term $\GeoAng^N \mytr \upchi$ by integrating the transport equation
\eqref{E:RAYCHEVOLUTIONS} in time and allowing the loss of one derivative
coming from the terms on RHS~\eqref{E:RAYCHEVOLUTIONS}.
This procedure couples the energy estimates at a given derivative level to
the estimates for the (already controlled) next-highest-energy, 
but it nonetheless allows one to derive the desired improvement in the energy blowup-rate by downward induction,
thanks to the integration in time and the estimate \eqref{E:REDUCETHESTRENTHOFTHESINGULARITY}.
\item \textbf{(Elliptic estimates along $\Sigma_t$).}
We now confront an important issue that we ignored in Step (7):
to close the energy estimates,
we are forced to control some of the inhomogeneous terms in the equations
using elliptic estimates along $\Sigma_t$. 
This major difficulty is not present in works on 
shock formation for wave equations; it was encountered
for the first time in our earlier work on shock formation
\cites{jLjS2016a} for barotropic fluids with vorticity.
A key aspect of the difficulty is that elliptic estimates
along $\Sigma_t$ necessarily involve controlling the derivatives
of the solution in a direction transversal to 
the acoustic characteristics $\mathcal{P}_u$, that is,
in the singular direction.
We need elliptic estimates to control 
the source terms 
on RHSs
\eqref{E:EVOLUTIONEQUATIONFLATCURLRENORMALIZEDVORTICITY} 
and 
\eqref{E:TRANSPORTFLATDIVGRADENT}
that depend on 
$\underline{\partial} \Vortrenormalized$
and
$\underline{\partial} \GradEnt$,
where $\underline{\partial}$ denotes the gradient with respect to the Cartesian spatial coordinates.
More precisely, we need the elliptic estimates only at the top derivative level,
but we will ignore that issue here and focus instead on the degeneracy of the elliptic estimates
with respect to $\upmu$.

The elliptic estimates can easily be \emph{derived} relative to the Cartesian coordinates
and the Euclidean volume form $dx^1 dx^2 dx^3$ on $\Sigma_t$. 
However, in order to compare the strength of the elliptic estimates to that of the
wave energies
\eqref{E:WAVEENERGIESSCHEMATICSTRENGTH}
and the transport energies
\eqref{E:TRANSPORTENERGIESSCHEMATICSTRENGTH},
we need to understand the relationship between
the Euclidean volume form 
and the volume form $du d \vartheta^1 d \vartheta^2$ featured in the energies.
Specifically, by studying the Jacobian of the change of variables map between
the geometric and the Cartesian coordinates, 
one can show that there is an
$\mathcal{O}(\upmu)$ discrepancy factor between the two forms:
\begin{align} \label{E:VOLFORMRELATION}
	dx^1 dx^2 dx^3 
	& \sim \upmu \, du d \vartheta^1 d \vartheta^2.
\end{align}
In the rest of this discussion, 
our notion of an $L^2(\Sigma_t)$ norm
is in terms of the volume form $du d \vartheta^1 d \vartheta^2$.
That is, we set
\begin{align} \label{E:GEOMETRICL2NORMALONGSIGMAT}
\| f \|_{L^2(\Sigma_t)}^2
:=
\int
	f^2(t,u,\vartheta^1,\vartheta^2)
\, du d \vartheta^1 d \vartheta^2.
\end{align}

We now explain some aspects of the elliptic estimates that 
yield control over $\underline{\partial} \GradEnt$, 
where as before, $\underline{\partial}$ denotes the gradient with respect to the Cartesian spatial coordinates.
One also needs similar elliptic estimates to obtain control over
$\underline{\partial} \Vortrenormalized$, but we omit those details;
see \cite{jLjS2016a} for an overview of how to control
$\underline{\partial} \Vortrenormalized$ in the barotropic case.
Our elliptic estimates are essentially standard div-curl
estimates of the form
\[
	\int_{\Sigma_t}
		\left|
			\underline{\partial} \GradEnt 
		\right|^2
	\, dx^1 dx^2 dx^3 
	\lesssim
	\int_{\Sigma_t}
		\left|
			\Flatdiv \GradEnt 
		\right|^2
	\, dx^1 dx^2 dx^3
	+
	\int_{\Sigma_t}
		\left|
			\Flatcurl \GradEnt 
		\right|^2
	\, dx^1 dx^2 dx^3.
\]
With the help of \eqref{E:VOLFORMRELATION} and \eqref{E:GEOMETRICL2NORMALONGSIGMAT},
we can re-express the above div-curl estimate as follows:
\begin{align} \label{E:STANDARDHODGE}
	\| \sqrt{\upmu} \underline{\partial} \GradEnt \|_{L^2(\Sigma_t)}
	& 
	\lesssim
	\| \sqrt{\upmu} \Flatdiv \GradEnt \|_{L^2(\Sigma_t)}
	+
	\| \sqrt{\upmu} \Flatcurl \GradEnt \|_{L^2(\Sigma_t)}.
\end{align}

We now explain the role that \eqref{E:STANDARDHODGE} plays in closing
the energy estimates. Our main goal is to show how to derive the 
bound 
\begin{align} \label{E:FINALSTANDARDHODGE}
	\| \sqrt{\upmu} \underline{\partial} \GradEnt \|_{L^2(\Sigma_t)}^2
	& 
	\lesssim
	\mathring{\upepsilon}^2 \upmu_{\star}^{- \widetilde{m}}(t)
	+ 
	\cdots,
\end{align}
where $\widetilde{m}$ is a \emph{small} positive constant and 
$\cdots$ denote error terms that can be controlled without elliptic estimates
(for example, via the wave energies).
We note that since \eqref{E:TOPDEGENERATEENERGY} implies that the top-order wave energies
can be very degenerate, some of the terms in $\cdots$ on RHS~\eqref{E:FINALSTANDARDHODGE}
can in fact blow up at a much worse rate than the one $\upmu_{\star}^{- \widetilde{m}}(t)$
that we have explicitly displayed. 
The point of writing the estimate 
for $\| \sqrt{\upmu} \underline{\partial} \GradEnt \|_{L^2(\Sigma_t)}^2$
in the form \eqref{E:FINALSTANDARDHODGE} is that
this form emphasizes the following point:
\emph{the self-interaction terms in the elliptic estimates are not the ones
driving the blowup rate of the top-order derivatives of} $\GradEnt$;
Instead the blowup-rate of 
$\| \sqrt{\upmu} \underline{\partial} \GradEnt \|_{L^2(\Sigma_t)}^2$
is driven by the blowup-rate of
the top-order derivatives of the wave variables
$\lbrace 
			\Densrenormalized, v^1, v^2, v^3
	\rbrace
$,
which are hidden in the $\cdots$ terms on RHS~\eqref{E:FINALSTANDARDHODGE}.
It turns out that as a consequence, 
the blowup rates for the top-order wave energies are exactly the same 
as they are in the isentropic irrotational case.
That is, our approach to energy estimates yields the same blowup-exponent $K$
in the energy hierarchy \eqref{E:TOPDEGENERATEENERGY}-\eqref{E:ONETIMECOMMUTEDNONDEGENERATEENERGY}
compared to the exponent that our approach would yield in the
isentropic irrotational case.

To explain how to derive \eqref{E:FINALSTANDARDHODGE},
we start by discussing energy 
estimates for the transport equation \eqref{E:TRANSPORTFLATDIVGRADENT}
for $\DivofEntrenormalized$. 
We again remind the reader that the elliptic estimate approach to deriving 
\eqref{E:FINALSTANDARDHODGE} is needed mainly at the top-order, but for convenience, we discuss here only
the non-differentiated equations.
Specifically, by deriving standard transport equation energy estimates for the
weighted equation
$\upmu \times \mbox{\eqref{E:TRANSPORTFLATDIVGRADENT}}$
and by using the $L^{\infty}$ bootstrap assumptions of Step (2)
(which in particular can be used to derive the bound $\| \upmu \partial_a v^b \|_{L^{\infty}(\Sigma_t)} \lesssim 1$),
one can obtain the following integral inequality:
\begin{align} \label{E:ENTROPYDIVERGENCETRANSPORTEQUATIONESTIMATE}
	\| \sqrt{\upmu} \DivofEntrenormalized \|_{L^2(\Sigma_t)}^2
	& \leq
		k
		\int_{s=0}^t
			\frac{1}{\upmu_{\star}(s)}
			\| \sqrt{\upmu} \underline{\partial} \GradEnt \|_{L^2(\Sigma_s)}^2	
		\, ds
		+
		\cdots.
\end{align}
In \eqref{E:ENTROPYDIVERGENCETRANSPORTEQUATIONESTIMATE},
$\cdots$ denotes simpler error terms that can be treated without elliptic estimates
and, by judicious use of Young's inequality, it can be arranged that $k$ is a small positive constant.\footnote{$k$ can be chosen to be small by using Young's inequality  
in the form 
$
\displaystyle
AB \lesssim k A^2 + \frac{1}{k} B^2
$ 
on the relevant error integrals. It turns out that the large-coefficient error integral,
which is represented by
$
\displaystyle
\frac{1}{k} B^2
$
and which we have relegated to the terms $\cdots$ on RHS~\eqref{E:ENTROPYDIVERGENCETRANSPORTEQUATIONESTIMATE},
is much less degenerate than the one we have explicitly displayed 
on RHS~\eqref{E:ENTROPYDIVERGENCETRANSPORTEQUATIONESTIMATE}
and in particular does not contribute to the blowup-rate of the top-order energies.
A full discussion of this issue would involve a lengthy interlude in which we
describe the need to rely, in addition to energies along $\Sigma_t$, 
energies along the acoustic characteristics $\mathcal{P}_u$.
For this reason, we omit this aspect of the discussion.}
Substituting the estimate \eqref{E:STANDARDHODGE}
into \eqref{E:ENTROPYDIVERGENCETRANSPORTEQUATIONESTIMATE},
using equation \eqref{E:DIVGRADENT},
and referring to definition \eqref{E:MUSTAR},
we obtain
\begin{align} \label{E:SECONDENTROPYDIVERGENCETRANSPORTEQUATIONESTIMATE}
	\| \sqrt{\upmu} \DivofEntrenormalized \|_{L^2(\Sigma_t)}^2
	& \leq
		k
		\int_{s=0}^t
			\frac{1}{\upmu_{\star}(s)}
			\| \sqrt{\upmu} \Flatdiv \GradEnt \|_{L^2(\Sigma_s)}^2		
		\, ds
		+
		\cdots,
\end{align}
where $\cdots$ is as above.
From the $L^{\infty}$ bootstrap assumptions of Step(2),
we have
$\exp(2 \Densrenormalized) \lesssim 1$
Thus, from \eqref{E:RENORMALIZEDDIVOFENTROPY}, we have
$\| \sqrt{\upmu} \Flatdiv \GradEnt \|_{L^2(\Sigma_s)}^2
\lesssim
\| \sqrt{\upmu} \DivofEntrenormalized \|_{L^2(\Sigma_s)}^2
+
\cdots
$,
where $\cdots$ denotes terms that can be controlled 
without elliptic estimates, that is, via energy estimates
for the wave equations
\eqref{E:VELOCITYWAVEEQUATION}-\eqref{E:RENORMALIZEDDENSITYWAVEEQUATION}
and the transport equations
\eqref{E:RENORMALIZEDVORTICTITYTRANSPORTEQUATION}-\eqref{E:GRADENTROPYTRANSPORT}.
From these estimates and \eqref{E:SECONDENTROPYDIVERGENCETRANSPORTEQUATIONESTIMATE}, 
we deduce
\begin{align} \label{E:GRONWALLABLEENTROPYDIVERGENCETRANSPORTEQUATIONESTIMATE}
	\| \sqrt{\upmu} \DivofEntrenormalized \|_{L^2(\Sigma_t)}^2
	& \leq
		k
		\int_{s=0}^t
			\frac{1}{\upmu_{\star}(s)}
			\| \sqrt{\upmu} \DivofEntrenormalized \|_{L^2(\Sigma_s)}^2		
		\, ds
		+
		\cdots,
\end{align}
where we again emphasize that $k$ can be chosen to be a small constant.
Then from 
\eqref{E:UPMUSTARSHARP},
\eqref{E:GRONWALLABLEENTROPYDIVERGENCETRANSPORTEQUATIONESTIMATE},
Gronwall's inequality, 
and an appropriate $\mathcal{O}(\mathring{\upepsilon})$-size small-data assumption,\footnote{We again note that the 
smallness assumption guarantees, roughly, that the data are near that of a simple isentropic plane wave solution.}
we deduce that
\begin{align} \label{E:GRONWALLEDENTROPYDIVERGENCETRANSPORTEQUATIONESTIMATE}
	\| \sqrt{\upmu} \DivofEntrenormalized \|_{L^2(\Sigma_t)}^2
	& \lesssim
		\mathring{\upepsilon}^2
		\upmu_{\star}^{- \widetilde{k}}(t)
		+ 
		\cdots,
\end{align}
where $\widetilde{k}$ is a small constant whose smallness is controlled by $k$.
Again using \eqref{E:RENORMALIZEDDIVOFENTROPY}
and the bound $\exp(2 \Densrenormalized) \lesssim 1$,
we deduce from \eqref{E:GRONWALLEDENTROPYDIVERGENCETRANSPORTEQUATIONESTIMATE}
that
\begin{align} \label{E:GRONWALLEDENTROPYDIVERGE}
	\| \sqrt{\upmu} \Flatdiv \GradEnt \|_{L^2(\Sigma_t)}^2
	& \lesssim
		\mathring{\upepsilon}^2
		\upmu_{\star}^{- \widetilde{k}}(t)
		+ 
		\cdots.
\end{align}
One can obtain similar but much simpler estimates for
$\| \sqrt{\upmu} \Flatcurl \GradEnt \|_{L^2(\Sigma_t)}^2$
directly from equation \eqref{E:DIVGRADENT}.\footnote{The bound is completely trivial in the case
of $\Flatcurl \GradEnt$ since $\Flatcurl \GradEnt = 0$. The needed bound is less trivial
after one commutes equation \eqref{E:DIVGRADENT} with the elements of $\Fullset$ since
one must control the commutator terms.}
Then inserting these bounds into \eqref{E:STANDARDHODGE}, we finally obtain
the desired bound \eqref{E:FINALSTANDARDHODGE}. We have therefore given the main
ideas behind the elliptic estimates. This completes our overview of our forthcoming
proof of shock formation.

\end{enumerate}

\section{Proof of Theorem~\ref{T:GEOMETRICWAVETRANSPORTSYSTEM}}
\label{S:PROOFOFMAINTHEOREM}
\setcounter{equation}{0}
In this section, we prove Theorem~\ref{T:GEOMETRICWAVETRANSPORTSYSTEM}.
The theorem is a conglomeration of
Lemmas~\ref{L:RENORMALIZEDVORTICITYEVOLUTIONEQUATION},
\ref{L:TRANSPORTFORENTANDGRADENT},
\ref{L:DIVANDCURLEQUATIONSFORENTROPY},
\ref{L:DIVANDCURLEQUATIONSFORVORTICITY},
\ref{L:WAVEEQUATIONFORLOGDENSITY},
and
\ref{L:WAVEEQUATIONFORV},
in which we separately derive the equations stated in the theorem.
Actually, to obtain Theorem~\ref{T:GEOMETRICWAVETRANSPORTSYSTEM} 
from the lemmas, one must reorganize the terms in the equations; 
we omit these minor details.

Throughout Sect.~\ref{S:PROOFOFMAINTHEOREM}, 
we freely use the following identity (see \eqref{E:RESCALEDVARIABLES}):
\begin{align} \label{E:VORTICIITYRENORMALIZEDVORTICITYANDDENSITYRELATION}
	\omega^i 
	& =
	\exp(\Densrenormalized) \Vortrenormalized^i.
\end{align}

\subsection{Deriving the transport equations}
\label{SS:MAINTHEOREMTRANSPORTEQUATIONS}
We start by deriving the evolution equation 
\eqref{E:RENORMALIZEDVORTICTITYTRANSPORTEQUATION} for 
$\Vortrenormalized$.

\begin{lemma}[\textbf{Transport equation for} $\Vortrenormalized$]
\label{L:RENORMALIZEDVORTICITYEVOLUTIONEQUATION}
The compressible Euler equations 
\eqref{E:TRANSPORTDENSRENORMALIZEDRELATIVETORECTANGULAR}-\eqref{E:ENTROPYTRANSPORT}
imply the following evolution equation for the Cartesian components
$\lbrace \Vortrenormalized^i \rbrace_{i=1,2,3}$
of the specific vorticity vectorfield from Def.~\ref{D:RESCALEDVARIABLES}:
\begin{align} \label{E:RESTATEDRENORMALIZEDVORTICITYEVOLUTIONEQUATION}
	\Transport \Vortrenormalized^i
	& = 
		\Vortrenormalized^a \partial_a v^i
		-
		\exp(-2 \Densrenormalized) \Speed^{-2} \frac{p_{;\Ent}}{\bar{\varrho}} \epsilon_{iab} (\Transport v^a) \GradEnt^b.
\end{align}
\end{lemma}

\begin{proof}
We first note the following chain rule identity,
which follows easily from 
definitions \eqref{E:RESCALEDVARIABLES}
and \eqref{E:SOUNDSPEED}:
$
\displaystyle
\exp(-\Densrenormalized) \partial_i p
=
\Speed^2 \delta^{ia} \partial_a \Densrenormalized
+ 
\exp(-\Densrenormalized) \frac{p_{;\Ent}}{\bar{\varrho}} \delta^{ia} \partial_a \Ent
$.
It follows that
$
\displaystyle
\left(\mbox{{\upshape RHS}~\eqref{E:TRANSPORTVELOCITYRELATIVETORECTANGULAR}}\right)^i
= - 
\frac{1}{\bar{\varrho}}
\exp(-\Densrenormalized) \partial_i p
$.
Hence, applying $\exp(-\Densrenormalized) \Flatcurl$ to \eqref{E:TRANSPORTVELOCITYRELATIVETORECTANGULAR}
and using 
definition \eqref{E:ENTROPYGRADIENT},
the antisymmetry of $\epsilon_{\dots}$,
and the symmetry $\partial_a \partial_b p = \partial_b \partial_a p$,
we deduce the following identity:
\begin{align} \label{E:CURLSOURCETERMID}
\exp(-\Densrenormalized) 
\left(
	\Flatcurl \mbox{{\upshape RHS}\eqref{E:TRANSPORTVELOCITYRELATIVETORECTANGULAR}} 
\right)^i
& = 
\exp(-2 \Densrenormalized)
\frac{p_{;\Ent}}{\bar{\varrho}}
\epsilon_{iab}
(\partial_a \Densrenormalized)
 \partial_b \Ent
 	\\
& = \exp(-2 \Densrenormalized)
\frac{p_{;\Ent}}{\bar{\varrho}}
\epsilon_{iab}
(\partial_a \Densrenormalized)
\GradEnt^b.
 \notag
\end{align}
Next, in view of the definition \eqref{E:RESCALEDVARIABLES} of $\Vortrenormalized$,
we commute equation \eqref{E:TRANSPORTVELOCITYRELATIVETORECTANGULAR}
with the operator 
$
\displaystyle
\exp(-\Densrenormalized)\Flatcurl
$
and use equations  
\eqref{E:TRANSPORTDENSRENORMALIZEDRELATIVETORECTANGULAR} and \eqref{E:MATERIALVECTORVIELDRELATIVETORECTANGULAR},
the antisymmetry of $\epsilon_{\dots}$,
the identity 
$\epsilon_{iab} \epsilon_{deb}
=
\delta_{id} \delta_{ae} - \delta_{ie} \delta_{ad}
$,
and the identity \eqref{E:CURLSOURCETERMID} to deduce
\begin{align} \label{E:FIRSTSTEPRENORMALIZEDVORTICITYEVOLUTIONEQUATION}
	\Transport \Vortrenormalized^i
	& = -
		 	\exp(-\Densrenormalized) \epsilon_{iab} (\partial_a v^d) \partial_d v^b
		  - 
		  \exp(-\Densrenormalized) (\Transport \Densrenormalized) \omega^i
		  + 
		  \exp(-2 \Densrenormalized) \frac{p_{;\Ent}}{\bar{\varrho}} \epsilon_{iab} (\partial_a \Densrenormalized) \GradEnt^b
		\\
	& = 
		- 
		\exp(-\Densrenormalized) \epsilon_{iab} (\partial_a v^d) \partial_d v^b
		+ 
		(\partial_a v^a) \Vortrenormalized^i
		+ 
		\exp(-2 \Densrenormalized) \frac{p_{;\Ent}}{\bar{\varrho}} \epsilon_{iab} (\partial_a \Densrenormalized) \GradEnt^b
		\notag \\
	& = 
		- 
		\exp(-\Densrenormalized) \epsilon_{iab} (\partial_a v^d) (\partial_d v^b - \partial_b v^d)
		+ 
		(\partial_a v^a) \Vortrenormalized^i
		+ 
		\exp(-2 \Densrenormalized) \frac{p_{;\Ent}}{\bar{\varrho}} \epsilon_{iab} (\partial_a \Densrenormalized) \GradEnt^b
		\notag \\
	& = - 
		\epsilon_{iab} \epsilon_{dbe} \Vortrenormalized^e (\partial_a v^d) 
		+ 
		(\partial_a v^a) \Vortrenormalized^i
		+ 
		\exp(-2 \Densrenormalized) \frac{p_{;\Ent}}{\bar{\varrho}} \epsilon_{iab} (\partial_a \Densrenormalized) \GradEnt^b
	\notag \\
	& = 
	\epsilon_{iab} \epsilon_{deb} \Vortrenormalized^e (\partial_a v^d) 
	+ 
	(\partial_a v^a) \Vortrenormalized^i
	+ 
	\exp(-2 \Densrenormalized) \frac{p_{;\Ent}}{\bar{\varrho}} \epsilon_{iab} (\partial_a \Densrenormalized) \GradEnt^b
	\notag \\
	& = (\delta_{id} \delta_{ae} - \delta_{ie} \delta_{ad}) \Vortrenormalized^e (\partial_a v^d) 
	+ 
	(\partial_a v^a) \Vortrenormalized^i
	+ 
	\exp(-2 \Densrenormalized) \frac{p_{;\Ent}}{\bar{\varrho}} \epsilon_{iab} (\partial_a \Densrenormalized) \GradEnt^b
	\notag 	
		\\
	& =
	\Vortrenormalized^a (\partial_a v^i) 
	+ 
	\exp(-2 \Densrenormalized) \frac{p_{;\Ent}}{\bar{\varrho}} \epsilon_{iab} (\partial_a \Densrenormalized) \GradEnt^b.
	\notag
\end{align}
Next, we use equation \eqref{E:TRANSPORTVELOCITYRELATIVETORECTANGULAR},
the definition $\GradEnt^i = \partial_i \Ent$,
and the antisymmetry of $\epsilon_{\dots}$ 
to derive the identity
$\epsilon_{iab} (\partial_a \Densrenormalized) \GradEnt^b 
=
- \Speed^{-2} \epsilon_{iab} (\Transport v^a) \GradEnt^b
$.
Substituting this identity into the last product on RHS~~\eqref{E:FIRSTSTEPRENORMALIZEDVORTICITYEVOLUTIONEQUATION},
we arrive at \eqref{E:RESTATEDRENORMALIZEDVORTICITYEVOLUTIONEQUATION}.
\end{proof}

We now establish equations 
\eqref{E:ENTROPYTRANSPORTMAINSYSTEM}-\eqref{E:GRADENTROPYTRANSPORT}.

\begin{lemma}[\textbf{Transport equations for} $\Ent$ \textbf{and} $\GradEnt^i$]
\label{L:TRANSPORTFORENTANDGRADENT}
The compressible Euler equations 
\eqref{E:TRANSPORTDENSRENORMALIZEDRELATIVETORECTANGULAR}-\eqref{E:ENTROPYTRANSPORT}
imply the following transport equations for $\Ent$ and
the Cartesian components
$\lbrace \GradEnt^i \rbrace_{i=1,2,3}$
of the entropy gradient vectorfield from Def.~\ref{D:ENTROPYGRADIENT}:
\begin{subequations}
\begin{align} \label{E:PROOFENTROPYTRANSPORTMAINSYSTEM}
\Transport \Ent & = 0,
	\\
\Transport \GradEnt^i
& = - \GradEnt^a \partial_a v^i
		+ 
		\exp(\Densrenormalized) \epsilon_{iab} \Vortrenormalized^a \GradEnt^b.
\label{E:PROOFGRADENTROPYTRANSPORT}
\end{align}
\end{subequations}
\end{lemma}

\begin{proof}
	Equation \eqref{E:PROOFENTROPYTRANSPORTMAINSYSTEM} is just a restatement of \eqref{E:ENTROPYTRANSPORT}.
	
	To derive \eqref{E:PROOFGRADENTROPYTRANSPORT}, we first commute equation
	\eqref{E:PROOFENTROPYTRANSPORTMAINSYSTEM} with $\partial_i$ and use  
	definition \eqref{E:MATERIALVECTORVIELDRELATIVETORECTANGULAR},
	the definition
	$\GradEnt^i = \partial_i \Ent$,
	and the identity \eqref{E:VORTICIITYRENORMALIZEDVORTICITYANDDENSITYRELATION}
	to obtain
	\begin{align}
		\Transport \GradEnt^i
		& = - \delta_{ab}(\partial_i v^a) \GradEnt^b
				\label{E:ENTROPYTRANSPORTONCECARTESIANCOMMMUTED}
					\\
		& = - \delta_{ab}(\partial_a v^i) \GradEnt^b
				+ \delta_{ab}(\partial_a v^i - \partial_i v^a) \GradEnt^b
				\notag \\
		& = - \GradEnt^a \partial_a v^i
			+ \epsilon_{aij} \omega^j \GradEnt^a
				\notag \\
		& = - \GradEnt^a \partial_a v^i
				+ \exp(\Densrenormalized) \epsilon_{aij} \Vortrenormalized^j \GradEnt^a.
			\notag
	\end{align}
	Equation \eqref{E:PROOFGRADENTROPYTRANSPORT} now follows from 
	\eqref{E:ENTROPYTRANSPORTONCECARTESIANCOMMMUTED}
	and the antisymmetry of $\epsilon_{\dots}$.
\end{proof}

We now establish equations \eqref{E:TRANSPORTFLATDIVGRADENT}-\eqref{E:DIVGRADENT}.

\begin{lemma}[\textbf{Equations for} $\DivofEntrenormalized$ \textbf{and} $\Flatcurl \GradEnt$]
\label{L:DIVANDCURLEQUATIONSFORENTROPY}
The compressible Euler equations 
\eqref{E:TRANSPORTDENSRENORMALIZEDRELATIVETORECTANGULAR}-\eqref{E:ENTROPYTRANSPORT}
imply the following equations
for the modified fluid variable $\DivofEntrenormalized$ defined in \eqref{E:RENORMALIZEDDIVOFENTROPY}
and for the Cartesian components $(\Flatcurl \GradEnt)^i$, 
$(i=1,2,3)$,
where $\GradEnt$ is the entropy gradient vectorfield defined 
in \eqref{E:ENTROPYGRADIENT}:
\begin{subequations}
\begin{align}
	\Transport \DivofEntrenormalized
	& =
		2 \exp(-2 \Densrenormalized)
		\left\lbrace
			(\partial_a v^a) \partial_b \GradEnt^b
			- 
			(\partial_a v^b) \partial_a \GradEnt^b
		\right\rbrace
			\label{E:PROOFTRANSPORTFLATDIVGRADENT} \\
	& \ \
		+ 
		2 \exp(-2 \Densrenormalized) 
		\left\lbrace
			(\GradEnt^a \partial_a v^b) \partial_b \Densrenormalized
			-
			(\partial_a v^a) \GradEnt^b \partial_b \Densrenormalized
		\right\rbrace
		\notag \\
	& \ \
		+ 
		\exp(-\Densrenormalized) \delta_{ab} (\Flatcurl \Vortrenormalized)^a \GradEnt^b,
		\notag \\
	(\Flatcurl \GradEnt)^i
	& = 0.
	\label{E:PROOFDIVGRADENT}
\end{align}
\end{subequations}

\end{lemma}

\begin{proof}
Equation \eqref{E:PROOFDIVGRADENT} is a simple consequence of the fact that
$\GradEnt$ is a (spatial) gradient vectorfield.

To derive \eqref{E:PROOFTRANSPORTFLATDIVGRADENT},
we first commute the already established equation \eqref{E:PROOFGRADENTROPYTRANSPORT}
with $\partial_i$ 
and use 
definition \eqref{E:MATERIALVECTORVIELDRELATIVETORECTANGULAR} and
equation \eqref{E:VORTICIITYRENORMALIZEDVORTICITYANDDENSITYRELATION}
to deduce
\begin{align}  \label{E:GRADENTROPYTRANSPORTFLATDIVTAKEN}
\Transport \Flatdiv \GradEnt
	& = 
		- 2 (\partial_a v^b) \partial_b \GradEnt^a
		- \GradEnt^a \partial_a \Flatdiv v
		+ \exp(\Densrenormalized) \epsilon_{iab} (\partial_i \Densrenormalized) \Vortrenormalized^a \GradEnt^b
		+ \exp(\Densrenormalized) \epsilon_{iab} (\partial_i \Vortrenormalized^a) \GradEnt^b
			\\
	& \ \
		+ \exp(\Densrenormalized)\epsilon_{iab} \Vortrenormalized^a \partial_i \GradEnt^b.
		\notag
\end{align}
From equation \eqref{E:PROOFDIVGRADENT},
we see that the last product on RHS~\eqref{E:GRADENTROPYTRANSPORTFLATDIVTAKEN} vanishes.
Also noting that $\epsilon_{iab} (\partial_i \Vortrenormalized^a) = (\Flatcurl \Vortrenormalized)^b$,
we deduce from \eqref{E:GRADENTROPYTRANSPORTFLATDIVTAKEN} that
\begin{align} \label{E:ANOTHERGRADENTROPYTRANSPORTFLATDIVTAKEN}
	\Transport \Flatdiv \GradEnt
	& = 
		- 
		\GradEnt^a \partial_a \Flatdiv v
		- 
		2 (\partial_a v^b) \partial_a \GradEnt^b
		+ 
		\exp(\Densrenormalized) \epsilon_{iab} (\partial_i \Densrenormalized) \Vortrenormalized^a \GradEnt^b
		+ 
		\exp(\Densrenormalized) \delta_{ab} (\Flatcurl \Vortrenormalized)^a \GradEnt^b.
\end{align}
Next, using equation \eqref{E:TRANSPORTDENSRENORMALIZEDRELATIVETORECTANGULAR} to substitute
for the term $\Flatdiv v$ on RHS~\eqref{E:ANOTHERGRADENTROPYTRANSPORTFLATDIVTAKEN}
and using equation \eqref{E:MATERIALVECTORVIELDRELATIVETORECTANGULAR},
we find that
\begin{align} \label{E:THIRDANOTHERGRADENTROPYTRANSPORTFLATDIVTAKEN}
\Transport \Flatdiv \GradEnt		
& = \GradEnt^a \partial_a (\Transport \Densrenormalized)
			- 2 (\partial_a v^b) \partial_b \GradEnt^a
			+ \exp(\Densrenormalized) \epsilon_{iab} (\partial_i \Densrenormalized) \Vortrenormalized^a \GradEnt^b
			+ \exp(\Densrenormalized) \delta_{ab} (\Flatcurl \Vortrenormalized)^a \GradEnt^b
				\\
	& =  \Transport (\GradEnt^a \partial_a \Densrenormalized)
		+ (\GradEnt^a \partial_a v^b) \partial_b \Densrenormalized
		- (\Transport \GradEnt^a) \partial_a \Densrenormalized
		- 2 (\partial_a v^b) \partial_b \GradEnt^a
		+ \exp(\Densrenormalized) \epsilon_{iab} (\partial_i \Densrenormalized) \Vortrenormalized^a \GradEnt^b
			\notag \\
	& \ \ 
		+ \exp(\Densrenormalized) \delta_{ab} (\Flatcurl \Vortrenormalized)^a \GradEnt^b.
		\notag
\end{align}
Using equation \eqref{E:PROOFGRADENTROPYTRANSPORT} to substitute 
for the factor $\Transport \GradEnt^a$ on 
the second line of RHS~\eqref{E:THIRDANOTHERGRADENTROPYTRANSPORTFLATDIVTAKEN}, 
we deduce
\begin{align} \label{E:FOURTHANOTHERGRADENTROPYTRANSPORTFLATDIVTAKEN}
	\Transport \Flatdiv \GradEnt 
	& =  
		\Transport (\GradEnt^a \partial_a \Densrenormalized)
		- 2 (\partial_a v^b) \partial_b \GradEnt^a
		+ 2 (\GradEnt^a \partial_a v^b) \partial_b \Densrenormalized
		+ \exp(\Densrenormalized) \delta_{ab} (\Flatcurl \Vortrenormalized)^a \GradEnt^b.
\end{align}
Bringing the term $\Transport (\GradEnt^a \partial_a \Densrenormalized)$
on RHS~\eqref{E:FOURTHANOTHERGRADENTROPYTRANSPORTFLATDIVTAKEN}
over to the left and then commuting the equation with $\exp(-2\Densrenormalized)$, 
we obtain
\begin{align} \label{E:ALMOSTFINALGRADENTROPYTRANSPORTFLATDIVTAKEN}
	\Transport 
	\left\lbrace
		\exp(-2 \Densrenormalized) \Flatdiv \GradEnt 
		-
		\exp(-2 \Densrenormalized) \GradEnt^a \partial_a \Densrenormalized
	\right\rbrace
	& =  
		- 
		2 \exp(-2 \Densrenormalized) (\partial_a v^b) \partial_b \GradEnt^a
		- 
		2 \exp(-2 \Densrenormalized) (\Transport \Densrenormalized) \Flatdiv \GradEnt
			\\
	& \ \
		+ 
		2 \exp(-2 \Densrenormalized) (\GradEnt^a \partial_a v^b) \partial_b \Densrenormalized
		+ 
		2 \exp(-2 \Densrenormalized) (\Transport \Densrenormalized) \GradEnt^a \partial_a \Densrenormalized
		\notag \\
	& \ \
		+ 
		\exp(-\Densrenormalized) \delta_{ab} (\Flatcurl \Vortrenormalized)^a \GradEnt^b.
		\notag
\end{align}
Finally, using equation \eqref{E:TRANSPORTDENSRENORMALIZEDRELATIVETORECTANGULAR}
to substitute for the two factors of $\Transport \Densrenormalized$
on RHS~\eqref{E:ALMOSTFINALGRADENTROPYTRANSPORTFLATDIVTAKEN}
and referring to definition \eqref{E:RENORMALIZEDDIVOFENTROPY},
we arrive at the desired equation \eqref{E:PROOFTRANSPORTFLATDIVGRADENT}.
\end{proof}

We now establish equations 
	\eqref{E:FLATDIVOFRENORMALIZEDVORTICITY}-\eqref{E:EVOLUTIONEQUATIONFLATCURLRENORMALIZEDVORTICITY}.
	
\begin{lemma}[\textbf{Equations for} $\Flatdiv \Vortrenormalized$ \textbf{and}
	$\CurlofVortrenormalized^i$]
\label{L:DIVANDCURLEQUATIONSFORVORTICITY}
Let $\lbrace \CurlofVortrenormalized^i \rbrace$ be the
Cartesian components of the modified fluid variable from
from Def.~\ref{D:RENORMALIZEDCURLOFSPECIFICVORTICITY}.
The compressible Euler equations 
\eqref{E:TRANSPORTDENSRENORMALIZEDRELATIVETORECTANGULAR}-\eqref{E:ENTROPYTRANSPORT}
imply the following equation for
$\Flatdiv \Vortrenormalized$
and transport equation for the Cartesian components
$\lbrace \Vortrenormalized^i \rbrace_{i=1,2,3}$,
where $\Vortrenormalized$ is the specific vorticity vectorfield from Def.~\ref{D:RESCALEDVARIABLES}:
\begin{subequations}
	\begin{align} \label{E:PROOFFLATDIVOFRENORMALIZEDVORTICITY}
	\Flatdiv \Vortrenormalized
	& = - \Vortrenormalized^a \partial_a \Densrenormalized,
	\end{align}
	
\begin{align} \label{E:PROOFEVOLUTIONEQUATIONFLATCURLRENORMALIZEDVORTICITY}
	\Transport \CurlofVortrenormalized^i
	& 	
		=
		- 
		2 \delta_{jk} \epsilon_{iab} \exp(-\Densrenormalized) (\partial_a v^j) \partial_b \Vortrenormalized^k
		+
		\epsilon_{ajk}
		\exp(-\Densrenormalized)
		(\partial_a v^i) 
		\partial_j \Vortrenormalized^k
			\\
&  \ \
		+
		\exp(-3 \Densrenormalized) \Speed^{-2} \frac{p_{;\Ent}}{\bar{\varrho}} 
		\left\lbrace 
			(\Transport \GradEnt^a) \partial_a v^i
			-
			(\Transport v^i) \partial_a \GradEnt^a
		\right\rbrace
		\notag \\
	& \ \
		+
		\exp(-3 \Densrenormalized) \Speed^{-2} \frac{p_{;\Ent}}{\bar{\varrho}}  
		\left\lbrace
			(\Transport v^a) \partial_a \GradEnt^i
			- 
			(\partial_a v^a) (\Transport \GradEnt^i)
		\right\rbrace
		\notag \\
& \ \
		+ 
		\exp(-3 \Densrenormalized) \Speed^{-2} \frac{p_{;\Ent}}{\bar{\varrho}}
		\left\lbrace
			(\partial_a v^b) (\partial_b v^a) \GradEnt^i
			-
			(\partial_a v^a) (\partial_b v^b) \GradEnt^i
		\right\rbrace
			\notag \\
& \ \
		+ 
		\exp(-3 \Densrenormalized) \Speed^{-2} \frac{p_{;\Ent}}{\bar{\varrho}}
		\left\lbrace
			(\partial_a v^a) \GradEnt^b \partial_b v^i 
			- 
			(\GradEnt^a \partial_a v^b) \partial_b v^i 
		\right\rbrace
		\notag \\
	& \ \
		+ 
		2 \exp(-3 \Densrenormalized) \Speed^{-2} \frac{p_{;\Ent}}{\bar{\varrho}} 
		\left\lbrace
			(\GradEnt^a \partial_a \Densrenormalized)  \Transport v^i
			- 
			(\Transport \Densrenormalized) \GradEnt^a \partial_a v^i
		\right\rbrace
		\notag \\
	&  \ \
			+ 
			2 \exp(-3 \Densrenormalized) \Speed^{-3} \Speed_{;\Densrenormalized} \frac{p_{;\Ent}}{\bar{\varrho}} 
			\left\lbrace
				(\GradEnt^a \partial_a \Densrenormalized)  \Transport v^i
				- 
				(\Transport \Densrenormalized) \GradEnt^a \partial_a v^i
			\right\rbrace
		\notag \\
	& \ \
		+ 
		\exp(-3 \Densrenormalized) \Speed^{-2} \frac{p_{;\Ent;\Densrenormalized}}{\bar{\varrho}} 
		\left\lbrace
			(\Transport \Densrenormalized) \GradEnt^a \partial_a v^i
			- 
			(\GradEnt^a \partial_a \Densrenormalized)  \Transport v^i
		\right\rbrace
		\notag \\
	& \ \
		+
		\exp(-3 \Densrenormalized) \Speed^{-2} \frac{p_{;\Ent;\Densrenormalized}}{\bar{\varrho}}  
		\GradEnt^i
		\left\lbrace
			(\Transport v^a) (\partial_a \Densrenormalized) 
			-
			(\Transport \Densrenormalized) (\partial_a v^a) 
		\right\rbrace
		\notag \\
	& \ \
		+
		2 \exp(-3 \Densrenormalized) \Speed^{-2} \frac{p_{;\Ent}}{\bar{\varrho}} 
		\GradEnt^i
		\left\lbrace
			(\Transport \Densrenormalized)(\partial_a v^a) 
			- 
			(\Transport v^a) (\partial_a \Densrenormalized) 
		\right\rbrace
		\notag \\
	& \ \
		 	+
			2 \exp(-3 \Densrenormalized) \Speed^{-3} \Speed_{;\Densrenormalized} \frac{p_{;\Ent}}{\bar{\varrho}} 
			\GradEnt^i
			\left\lbrace 
				(\Transport \Densrenormalized)(\partial_a v^a) 
		 		- 
		 		(\Transport v^a) (\partial_a \Densrenormalized) 
		 	\right\rbrace
		 		\notag \\
		& \ \
				+ 
				2 \exp(-3 \Densrenormalized) \Speed^{-3} \Speed_{;\Ent} \frac{p_{;\Ent}}{\bar{\varrho}} 
				(\Transport v^i) \delta_{ab} \GradEnt^a \GradEnt^b
				-
				2 \exp(-3 \Densrenormalized) \Speed^{-3} \Speed_{;\Ent} \frac{p_{;\Ent}}{\bar{\varrho}} 
				\delta_{ab} \GradEnt^a (\Transport v^b) \GradEnt^i
			\notag \\
		& \ \
			+ 
			\exp(-3 \Densrenormalized) \Speed^{-2} \frac{p_{;\Ent;\Ent}}{\bar{\varrho}} (\Transport v^i) \delta_{ab} \GradEnt^a \GradEnt^b
			- 
			\exp(-3 \Densrenormalized) \Speed^{-2} \frac{p_{;\Ent;\Ent}}{\bar{\varrho}} \delta_{ab} (\Transport v^a) \GradEnt^b \GradEnt^i.
			\notag 
\end{align}	
\end{subequations}

\end{lemma}

\begin{proof}
	Equation \eqref{E:PROOFFLATDIVOFRENORMALIZEDVORTICITY} follows easily from
	applying the operator $\Flatdiv$ to equation
	\eqref{E:VORTICIITYRENORMALIZEDVORTICITYANDDENSITYRELATION}
	and noting that since $\omega = \Flatcurl v$,
	we have $\Flatdiv \omega = 0$.
	
	We now derive \eqref{E:PROOFEVOLUTIONEQUATIONFLATCURLRENORMALIZEDVORTICITY}.
	First, commuting the already established equation \eqref{E:RESTATEDRENORMALIZEDVORTICITYEVOLUTIONEQUATION}
	with the operator $\Flatcurl$
	and using the definitions
	\eqref{E:ENTROPYGRADIENT}
	and
	\eqref{E:MATERIALVECTORVIELDRELATIVETORECTANGULAR}
	as well as equations 
	\eqref{E:VORTICIITYRENORMALIZEDVORTICITYANDDENSITYRELATION}
	and
	\eqref{E:PROOFFLATDIVOFRENORMALIZEDVORTICITY},
	we compute that
	\begin{align} \label{E:FIRSTSTEPEVOLUTIONEQUATIONFLATCURLRENORMALIZEDVORTICITY}
	\Transport (\Flatcurl \Vortrenormalized)^i
	& 	= \Vortrenormalized^a \partial_a \omega^i
			+ 
			\epsilon_{iab} (\partial_a \Vortrenormalized^d) \partial_d v^b
			- 
			\epsilon_{iab} (\partial_a v^d) \partial_d \Vortrenormalized^b
				\\
	& \ \
			- 
			\epsilon_{iab} \epsilon_{bjk} \partial_a
				\left\lbrace
					\exp(-2 \Densrenormalized) \Speed^{-2}\frac{p_{;\Ent}}{\bar{\varrho}} (\Transport v^j) \GradEnt^k
				\right\rbrace
				\notag \\
	& = (\exp \Densrenormalized) \Vortrenormalized^a \partial_a \Vortrenormalized^i
			+ 
			(\exp \Densrenormalized) \Vortrenormalized^i \Vortrenormalized^a \partial_a \Densrenormalized
			+ 
			\epsilon_{iab} (\partial_a \Vortrenormalized^d) \partial_d v^b
			- 
			\epsilon_{iab} (\partial_a v^d) \partial_d \Vortrenormalized^b
				\notag \\
	& \ \
			+ 2 \epsilon_{iab} \epsilon_{bjk} (\partial_a \Densrenormalized)
				\left\lbrace
					\exp(-2 \Densrenormalized) \Speed^{-2} \frac{p_{;\Ent}}{\bar{\varrho}} (\Transport v^j) \GradEnt^k
				\right\rbrace
				\notag \\
	& \ \
			+ 2 \epsilon_{iab} \epsilon_{bjk} (\partial_a \Densrenormalized)
				\left\lbrace
					\exp(-2 \Densrenormalized) \Speed^{-3} \Speed_{;\Densrenormalized} \frac{p_{;\Ent}}{\bar{\varrho}} 
					(\Transport v^j) \GradEnt^k
				\right\rbrace
				\notag \\
	& \ \
				+ 2 \epsilon_{iab} \epsilon_{bjk} \GradEnt^a
				\left\lbrace
					\exp(-2 \Densrenormalized) \Speed^{-3} \Speed_{;\Ent} \frac{p_{;\Ent}}{\bar{\varrho}} (\Transport v^j) \GradEnt^k
				\right\rbrace
				\notag \\
	& \ \
			- \epsilon_{iab} \epsilon_{bjk} (\partial_a \Densrenormalized)
				\left\lbrace
					\exp(-2 \Densrenormalized) \Speed^{-2} \frac{p_{;\Ent;\Densrenormalized}}{\bar{\varrho}}  (\Transport v^j) \GradEnt^k
				\right\rbrace
				\notag \\
	& \ \
			- \epsilon_{iab} \epsilon_{bjk} \GradEnt^a
				\left\lbrace
					\exp(-2 \Densrenormalized) \Speed^{-2} \frac{p_{;\Ent;\Ent}}{\bar{\varrho}}  (\Transport v^j) \GradEnt^k
				\right\rbrace
				\notag \\
	& \ \
			- \epsilon_{iab} \epsilon_{bjk} 
				\left\lbrace
					\exp(-2 \Densrenormalized) \Speed^{-2} \frac{p_{;\Ent}}{\bar{\varrho}} (\partial_a v^d) (\partial_d v^j) \GradEnt^k
				\right\rbrace
				\notag \\
	& \ \
			- \epsilon_{iab} \epsilon_{bjk} 
				\left\lbrace
					\exp(-2 \Densrenormalized) \Speed^{-2} \frac{p_{;\Ent}}{\bar{\varrho}} (\Transport \partial_a v^j) \GradEnt^k
				\right\rbrace
				\notag \\
	& \ \
			- \epsilon_{iab} \epsilon_{bjk} 
				\left\lbrace
					\exp(-2 \Densrenormalized) \Speed^{-2} \frac{p_{;\Ent}}{\bar{\varrho}} (\Transport v^j) \partial_a \GradEnt^k
				\right\rbrace.
				\notag 
\end{align}

Next, using the identity
$
\epsilon_{iab} \epsilon_{jkb}
	= \delta_{ij}\delta_{ak} 
		- 
		\delta_{ik} \delta_{aj}
$
and the antisymmetry of $\epsilon_{\cdots}$,
we rewrite the third and fourth terms on RHS~\eqref{E:FIRSTSTEPEVOLUTIONEQUATIONFLATCURLRENORMALIZEDVORTICITY}
as follows:
\begin{align} \label{E:REWRITINGOFCURLVORTICITYNULLFORM}
	\epsilon_{iab} (\partial_a \Vortrenormalized^d) \partial_d v^b
	 - 
	\epsilon_{iab} (\partial_a v^d) \partial_d \Vortrenormalized^b
	& = 
			- 
			2 \delta_{jk} \epsilon_{iab} (\partial_a v^j) \partial_b \Vortrenormalized^k
				\\
	& \ \
			+ 
			\epsilon_{iab} (\partial_a \Vortrenormalized^d) (\partial_d v^b - \partial_b v^d)
			+ 
			\epsilon_{iab} (\partial_a v^d) (\partial_b \Vortrenormalized^d - \partial_d \Vortrenormalized^b)
				\notag \\
		& = 
			- 
			2 \delta_{jk} \epsilon_{iab} (\partial_a v^j) \partial_b \Vortrenormalized^k
			+
			\epsilon_{iab} \epsilon_{jdb} (\partial_a \Vortrenormalized^d) \omega^j
			+ 
			\epsilon_{iab} \epsilon_{jbd}  (\partial_a v^d) (\Flatcurl \Vortrenormalized)^j
			\notag \\
		& = 
			- 
			2 \delta_{jk} \epsilon_{iab} (\partial_a v^j) \partial_b \Vortrenormalized^k
			+
			\exp(\Densrenormalized) (\Flatdiv \Vortrenormalized) \Vortrenormalized^i
			-
			\exp(\Densrenormalized) \Vortrenormalized^a \partial_a \Vortrenormalized^i
				\notag \\
		& \ \
			+ 
			(\Flatcurl \Vortrenormalized)^a \partial_a v^i
			-
			(\partial_a v^a) (\Flatcurl \Vortrenormalized)^i.
			\notag
\end{align}

Substituting RHS~\eqref{E:REWRITINGOFCURLVORTICITYNULLFORM}
for the third and fourth terms on RHS~\eqref{E:FIRSTSTEPEVOLUTIONEQUATIONFLATCURLRENORMALIZEDVORTICITY},
using equation \eqref{E:PROOFFLATDIVOFRENORMALIZEDVORTICITY} for substitution,
and using the identities
$
(\Flatcurl \Vortrenormalized)^a \partial_a v^i
= 
\epsilon_{ajk} (\partial_a v^i) \partial_j \Vortrenormalized^k
$
and 
$\epsilon_{iab} \epsilon_{bjk}
	= \delta_{ij}\delta_{ak} 
		- 
		\delta_{ik} \delta_{aj}
$,
we compute that
\begin{align} \label{E:SECONDSTEPEVOLUTIONEQUATIONFLATCURLRENORMALIZEDVORTICITY}
	\Transport (\Flatcurl \Vortrenormalized)^i
	& = - 
			2 \delta_{jk} \epsilon_{iab} (\partial_a v^j) \partial_b \Vortrenormalized^k
			-
			(\partial_a v^a) (\Flatcurl \Vortrenormalized)^i
			+
			\epsilon_{ajk} (\partial_a v^i) \partial_j \Vortrenormalized^k
			\\
	& \ \
		 	+ 2 \exp(-2 \Densrenormalized) \Speed^{-2} \frac{p_{;\Ent}}{\bar{\varrho}} 
		 		(\GradEnt^a \partial_a \Densrenormalized)  \Transport v^i
		 	- 2 \exp(-2 \Densrenormalized) \Speed^{-2} \frac{p_{;\Ent}}{\bar{\varrho}}
		 		(\Transport v^a) (\partial_a \Densrenormalized) \GradEnt^i
		 	\notag \\
	& \ \
		 	+ 2 \exp(-2 \Densrenormalized) \Speed^{-3} \Speed_{;\Densrenormalized} 
		 			\frac{p_{;\Ent}}{\bar{\varrho}} (\GradEnt^a \partial_a \Densrenormalized)  \Transport v^i
		 	- 2 \exp(-2 \Densrenormalized) \Speed^{-3} \Speed_{;\Densrenormalized} \frac{p_{;\Ent}}{\bar{\varrho}} 
		 			(\Transport v^a) (\partial_a \Densrenormalized) \GradEnt^i
		 		\notag \\
		& \ \
				+ 
				2 \exp(-2 \Densrenormalized) \Speed^{-3} \Speed_{;\Ent} \frac{p_{;\Ent}}{\bar{\varrho}} 
				(\Transport v^i) \delta_{ab} \GradEnt^a \GradEnt^b
				-
				2 \exp(-2 \Densrenormalized) \Speed^{-3} \Speed_{;\Ent} \frac{p_{;\Ent}}{\bar{\varrho}} 
				\delta_{ab} \GradEnt^a (\Transport v^b) \GradEnt^i
			\notag \\
	 & \ \
		 	+ 
		 	\exp(-2 \Densrenormalized) \Speed^{-2} \frac{p_{;\Ent;\Densrenormalized}}{\bar{\varrho}}
		 	(\Transport v^a) (\partial_a \Densrenormalized) \GradEnt^i
		 	- 
		 	\exp(-2 \Densrenormalized) \Speed^{-2} \frac{p_{;\Ent;\Densrenormalized}}{\bar{\varrho}}
		 	(\GradEnt^a \partial_a \Densrenormalized)  \Transport v^i
		 	\notag \\
	 & \ \
			+ \exp(-2 \Densrenormalized) \Speed^{-2} \frac{p_{;\Ent;\Ent}}{\bar{\varrho}} 
				(\Transport v^i) \delta_{ab} \GradEnt^a \GradEnt^b
			- \exp(-2 \Densrenormalized) \Speed^{-2} \frac{p_{;\Ent;\Ent}}{\bar{\varrho}} \delta_{ab} (\Transport v^a) \GradEnt^b \GradEnt^i
			\notag \\
		& \ \
			+ 
			\exp(-2 \Densrenormalized) \Speed^{-2} \frac{p_{;\Ent}}{\bar{\varrho}} (\partial_a v^b) (\partial_b v^a) \GradEnt^i
			- 
			\exp(-2 \Densrenormalized) \Speed^{-2} \frac{p_{;\Ent}}{\bar{\varrho}}(\GradEnt^a \partial_a v^b) \partial_b v^i 
			\notag \\
		& \ \
			+ 
			\exp(-2 \Densrenormalized) \Speed^{-2} \frac{p_{;\Ent}}{\bar{\varrho}} (\Transport \partial_a v^a) \GradEnt^i
			- 
			\exp(-2 \Densrenormalized) \Speed^{-2} \frac{p_{;\Ent}}{\bar{\varrho}} \GradEnt^a \Transport \partial_a v^i
				\notag \\
		& \ \
			+ 
			\exp(-2 \Densrenormalized) \Speed^{-2} \frac{p_{;\Ent}}{\bar{\varrho}}  (\Transport v^a) \partial_a \GradEnt^i
			-
			\exp(-2 \Densrenormalized) \Speed^{-2} \frac{p_{;\Ent}}{\bar{\varrho}}  (\Transport v^i) \partial_a \GradEnt^a.
			\notag
	\end{align}
We now bring the terms
$
\displaystyle
\exp(-2 \Densrenormalized) \Speed^{-2} \frac{p_{;\Ent}}{\bar{\varrho}} (\Transport \partial_a v^a) \GradEnt^i
- 
\exp(-2 \Densrenormalized) \Speed^{-2} \frac{p_{;\Ent}}{\bar{\varrho}} \GradEnt^a \Transport \partial_a v^i
$	
from 
the next-to-last line of RHS~\eqref{E:SECONDSTEPEVOLUTIONEQUATIONFLATCURLRENORMALIZEDVORTICITY} over to the left
under the transport operator $\Transport$, which generates some additional terms on the RHS
(note that by \eqref{E:ENTROPYTRANSPORT}, the additional terms that depend on $\Transport \Ent$
completely vanish).
In total, we obtain
\begin{align} \label{E:THIRDSTEPEVOLUTIONEQUATIONFLATCURLRENORMALIZEDVORTICITY}
&
\Transport 
\left\lbrace
	(\Flatcurl \Vortrenormalized)^i
	+
	\exp(-2 \Densrenormalized) \Speed^{-2} \frac{p_{;\Ent}}{\bar{\varrho}} \GradEnt^a \partial_a v^i
	-
	\exp(-2 \Densrenormalized) \Speed^{-2} \frac{p_{;\Ent}}{\bar{\varrho}} (\partial_a v^a) \GradEnt^i
\right\rbrace
	\\
& = -
		(\partial_a v^a) (\Flatcurl \Vortrenormalized)^i
		\notag \\
& \ \
		- 
		2 \delta_{jk} \epsilon_{iab} (\partial_a v^j) \partial_b \Vortrenormalized^k
		+
		\epsilon_{ajk} (\partial_a v^i) \partial_j \Vortrenormalized^k
		\notag
				\\
	& \ \
		- 
		2 \exp(-2 \Densrenormalized) \Speed^{-2} \frac{p_{;\Ent}}{\bar{\varrho}} (\Transport \Densrenormalized) \GradEnt^a \partial_a v^i
		- 
			2 \exp(-2 \Densrenormalized) \Speed^{-3} \Speed_{;\Densrenormalized} 
			\frac{p_{;\Ent}}{\bar{\varrho}} (\Transport \Densrenormalized) \GradEnt^a \partial_a v^i
		\notag \\
	& \ \
		+ 
		\exp(-2 \Densrenormalized) \Speed^{-2} \frac{p_{;\Ent;\Densrenormalized}}{\bar{\varrho}} 
		(\Transport \Densrenormalized) \GradEnt^a \partial_a v^i
		\notag \\
	& \ \
		+ 
		\exp(-2 \Densrenormalized) \Speed^{-2} \frac{p_{;\Ent}}{\bar{\varrho}} (\Transport \GradEnt^a) \partial_a v^i
		\notag \\
	& \ \
		+
		2 \exp(-2 \Densrenormalized) \Speed^{-2} \frac{p_{;\Ent}}{\bar{\varrho}} (\Transport \Densrenormalized)(\partial_a v^a) \GradEnt^i
		+
		2 \exp(-2 \Densrenormalized) \Speed^{-3} \Speed_{;\Densrenormalized} \frac{p_{;\Ent}}{\bar{\varrho}} 
		(\Transport \Densrenormalized)(\partial_a v^a) \GradEnt^i
		\notag \\
	& \ \
		-
		\exp(-2 \Densrenormalized) \Speed^{-2} \frac{p_{;\Ent;\Densrenormalized}}{\bar{\varrho}} 
		(\Transport \Densrenormalized) (\partial_a v^a) \GradEnt^i
		\notag \\
	& \ \
		- \exp(-2 \Densrenormalized) \Speed^{-2} \frac{p_{;\Ent}}{\bar{\varrho}} (\partial_a v^a) (\Transport \GradEnt^i)
		\notag \\
  & \ \
		 	+ 2 \exp(-2 \Densrenormalized) \Speed^{-2} \frac{p_{;\Ent}}{\bar{\varrho}} 
		 		(\GradEnt^a \partial_a \Densrenormalized)  \Transport v^i
		 	- 2 \exp(-2 \Densrenormalized) \Speed^{-2} \frac{p_{;\Ent}}{\bar{\varrho}}
		 		(\Transport v^a) (\partial_a \Densrenormalized) \GradEnt^i
		 	\notag \\
	& \ \
		 	+ 2 \exp(-2 \Densrenormalized) \Speed^{-3} \Speed_{;\Densrenormalized} 
		 			\frac{p_{;\Ent}}{\bar{\varrho}} (\GradEnt^a \partial_a \Densrenormalized)  \Transport v^i
		 	- 2 \exp(-2 \Densrenormalized) \Speed^{-3} \Speed_{;\Densrenormalized} \frac{p_{;\Ent}}{\bar{\varrho}} 
		 			(\Transport v^a) (\partial_a \Densrenormalized) \GradEnt^i
		 		\notag \\
		& \ \
				+ 
				2 \exp(-2 \Densrenormalized) \Speed^{-3} \Speed_{;\Ent} \frac{p_{;\Ent}}{\bar{\varrho}} 
				(\Transport v^i) \delta_{ab} \GradEnt^a \GradEnt^b
				-
				2 \exp(-2 \Densrenormalized) \Speed^{-3} \Speed_{;\Ent} \frac{p_{;\Ent}}{\bar{\varrho}} 
				\delta_{ab} \GradEnt^a (\Transport v^b) \GradEnt^i
			\notag \\
		& \ \
		 	+ 
		 	\exp(-2 \Densrenormalized) \Speed^{-2} \frac{p_{;\Ent;\Densrenormalized}}{\bar{\varrho}}
		 	(\Transport v^a) (\partial_a \Densrenormalized) \GradEnt^i
		 	- 
		 	\exp(-2 \Densrenormalized) \Speed^{-2} \frac{p_{;\Ent;\Densrenormalized}}{\bar{\varrho}}
		 	(\GradEnt^a \partial_a \Densrenormalized)  \Transport v^i
		 	\notag \\
		& \ \
			+ 
				\exp(-2 \Densrenormalized) \Speed^{-2} \frac{p_{;\Ent;\Ent}}{\bar{\varrho}} 
				(\Transport v^i) \delta_{ab} \GradEnt^a \GradEnt^b
			- 
			\exp(-2 \Densrenormalized) \Speed^{-2} \frac{p_{;\Ent;\Ent}}{\bar{\varrho}} \delta_{ab} (\Transport v^a) \GradEnt^b \GradEnt^i
			\notag \\
		& \ \
			+ 
			\exp(-2 \Densrenormalized) \Speed^{-2} \frac{p_{;\Ent}}{\bar{\varrho}} (\partial_a v^b) (\partial_b v^a) \GradEnt^i
			- 
			\exp(-2 \Densrenormalized) \Speed^{-2} \frac{p_{;\Ent}}{\bar{\varrho}}(\GradEnt^a \partial_a v^b) \partial_b v^i 
			\notag \\
		& \ \
			+ 
			\exp(-2 \Densrenormalized) \Speed^{-2} \frac{p_{;\Ent}}{\bar{\varrho}}  (\Transport v^a) \partial_a \GradEnt^i
			-
			\exp(-2 \Densrenormalized) \Speed^{-2} \frac{p_{;\Ent}}{\bar{\varrho}}  (\Transport v^i) \partial_a \GradEnt^a.
			\notag
\end{align}

Next, we add 
$
\displaystyle
-
\exp(-2 \Densrenormalized) \Speed^{-2} \frac{p_{;\Ent}}{\bar{\varrho}}(\partial_a v^a) \GradEnt^b \partial_b v^i
+
\exp(-2 \Densrenormalized) \Speed^{-2} \frac{p_{;\Ent}}{\bar{\varrho}}(\partial_a v^a)^2 \GradEnt^i
$
to the first line of RHS~\eqref{E:THIRDSTEPEVOLUTIONEQUATIONFLATCURLRENORMALIZEDVORTICITY},
subtract the same terms on a different line,
use equation \eqref{E:TRANSPORTDENSRENORMALIZEDRELATIVETORECTANGULAR}
to substitute $\partial_a v^a = - \Transport \Densrenormalized$ for some factors,
and rearrange the terms to deduce
\begin{align}  \label{E:FOURTHSTEPEVOLUTIONEQUATIONFLATCURLRENORMALIZEDVORTICITY}
&
\Transport 
\left\lbrace
	(\Flatcurl \Vortrenormalized)^i
	+
	\exp(-2 \Densrenormalized) \Speed^{-2} \frac{p_{;\Ent}}{\bar{\varrho}} \GradEnt^a \partial_a v^i
	-
	\exp(-2 \Densrenormalized) \Speed^{-2} \frac{p_{;\Ent}}{\bar{\varrho}} (\partial_a v^a) \GradEnt^i
\right\rbrace
	\\
& = 
		(\Transport \Densrenormalized) (\Flatcurl \Vortrenormalized)^i
		+ 
		\exp(-2 \Densrenormalized) \Speed^{-2} \frac{p_{;\Ent}}{\bar{\varrho}} (\Transport \Densrenormalized) \GradEnt^a \partial_a v^i 
		- 
		\exp(-2 \Densrenormalized) \Speed^{-2} \frac{p_{;\Ent}}{\bar{\varrho}}(\Transport \Densrenormalized)(\partial_a v^a) \GradEnt^i
		\notag \\
& \ \
		+ 
		\exp(-2 \Densrenormalized) \Speed^{-2} \frac{p_{;\Ent}}{\bar{\varrho}} (\Transport \GradEnt^a) \partial_a v^i
		-
		\exp(-2 \Densrenormalized) \Speed^{-2} \frac{p_{;\Ent}}{\bar{\varrho}}  (\Transport v^i) \partial_a \GradEnt^a
		\notag \\
	& \ \
		+
		\exp(-2 \Densrenormalized) \Speed^{-2} \frac{p_{;\Ent}}{\bar{\varrho}}  (\Transport v^a) \partial_a \GradEnt^i
		- 
		\exp(-2 \Densrenormalized) \Speed^{-2} \frac{p_{;\Ent}}{\bar{\varrho}} (\partial_a v^a) (\Transport \GradEnt^i)
		\notag \\
& \ \
		- 
		2 \delta_{jk} \epsilon_{iab} (\partial_a v^j) \partial_b \Vortrenormalized^k
		+
		\epsilon_{ajk}
		(\partial_a v^i) 
		\partial_j \Vortrenormalized^k
		\notag
				\\
& \ \
		+ 
		\exp(-2 \Densrenormalized) \Speed^{-2} \frac{p_{;\Ent}}{\bar{\varrho}}(\partial_a v^b) (\partial_b v^a) \GradEnt^i
		-
		\exp(-2 \Densrenormalized) \Speed^{-2} \frac{p_{;\Ent}}{\bar{\varrho}}(\partial_a v^a)^2 \GradEnt^i
			\notag \\
& \ \
		+ 
		\exp(-2 \Densrenormalized) \Speed^{-2} \frac{p_{;\Ent}}{\bar{\varrho}}(\partial_a v^a) \GradEnt^b \partial_b v^i 
		- 
		\exp(-2 \Densrenormalized) \Speed^{-2} \frac{p_{;\Ent}}{\bar{\varrho}}(\GradEnt^a \partial_a v^b) \partial_b v^i 
		\notag \\
	& \ \
		+ 2 \exp(-2 \Densrenormalized) \Speed^{-2} \frac{p_{;\Ent}}{\bar{\varrho}} (\GradEnt^a \partial_a \Densrenormalized)  \Transport v^i
		- 
		2 \exp(-2 \Densrenormalized) \Speed^{-2} \frac{p_{;\Ent}}{\bar{\varrho}} (\Transport \Densrenormalized) \GradEnt^a \partial_a v^i
		\notag \\
	&  \ \
			+ 
			2 \exp(-2 \Densrenormalized) \Speed^{-3} \Speed_{;\Densrenormalized}  
			\frac{p_{;\Ent}}{\bar{\varrho}} (\GradEnt^a \partial_a \Densrenormalized)  \Transport v^i
			- 
			2 \exp(-2 \Densrenormalized) \Speed^{-3} \Speed_{;\Densrenormalized} 
			\frac{p_{;\Ent}}{\bar{\varrho}} (\Transport \Densrenormalized) \GradEnt^a \partial_a v^i
		\notag \\
	& \ \
		+ 
		\exp(-2 \Densrenormalized) \Speed^{-2} \frac{p_{;\Ent;\Densrenormalized}}{\bar{\varrho}} (\Transport \Densrenormalized) \GradEnt^a \partial_a v^i
		- 
		\exp(-2 \Densrenormalized) \Speed^{-2} \frac{p_{;\Ent;\Densrenormalized}}{\bar{\varrho}} (\GradEnt^a \partial_a \Densrenormalized)  \Transport v^i
		\notag \\
	& \ \
		+
		\exp(-2 \Densrenormalized) \Speed^{-2} \frac{p_{;\Ent;\Densrenormalized}}{\bar{\varrho}}  (\Transport v^a) (\partial_a \Densrenormalized) \GradEnt^i
		-
		\exp(-2 \Densrenormalized) \Speed^{-2} \frac{p_{;\Ent;\Densrenormalized}}{\bar{\varrho}} (\Transport \Densrenormalized) (\partial_a v^a) \GradEnt^i
		\notag \\
	& \ \
		+
		2 \exp(-2 \Densrenormalized) \Speed^{-2} \frac{p_{;\Ent}}{\bar{\varrho}} (\Transport \Densrenormalized)(\partial_a v^a) \GradEnt^i
		- 
		2 \exp(-2 \Densrenormalized) \Speed^{-2} \frac{p_{;\Ent}}{\bar{\varrho}}  (\Transport v^a) (\partial_a \Densrenormalized) \GradEnt^i
		\notag \\
	& \ \
		 	+
			2 \exp(-2 \Densrenormalized) \Speed^{-3} \Speed_{;\Densrenormalized} \frac{p_{;\Ent}}{\bar{\varrho}} 
			(\Transport \Densrenormalized)(\partial_a v^a) \GradEnt^i
		 	- 
		 	2 \exp(-2 \Densrenormalized) \Speed^{-3} \Speed_{;\Densrenormalized} \frac{p_{;\Ent}}{\bar{\varrho}}  
		 	(\Transport v^a) (\partial_a \Densrenormalized) \GradEnt^i
		 		\notag \\
		& \ \
				+ 
				2 \exp(-2 \Densrenormalized) \Speed^{-3} \Speed_{;\Ent} \frac{p_{;\Ent}}{\bar{\varrho}} 
				(\Transport v^i) \delta_{ab} \GradEnt^a \GradEnt^b
				-
				2 \exp(-2 \Densrenormalized) \Speed^{-3} \Speed_{;\Ent} \frac{p_{;\Ent}}{\bar{\varrho}} 
				\delta_{ab} \GradEnt^a (\Transport v^b) \GradEnt^i
			\notag \\
		& \ \
			+ 
			\exp(-2 \Densrenormalized) \Speed^{-2} \frac{p_{;\Ent;\Ent}}{\bar{\varrho}} (\Transport v^i) \delta_{ab} \GradEnt^a \GradEnt^b
			- 
			\exp(-2 \Densrenormalized) \Speed^{-2} \frac{p_{;\Ent;\Ent}}{\bar{\varrho}} \delta_{ab} (\Transport v^a) \GradEnt^b \GradEnt^i.
			\notag 
\end{align}	

We now multiply both sides of \eqref{E:FOURTHSTEPEVOLUTIONEQUATIONFLATCURLRENORMALIZEDVORTICITY}
by $\exp(-\Densrenormalized)$ and bring the factor of 
$\exp(-\Densrenormalized)$ under the operator $\Transport$ on the LHS.
The commutator term $(\Transport \exp(-\Densrenormalized)) \times \cdots$ 
completely cancels the first line of RHS~\eqref{E:FOURTHSTEPEVOLUTIONEQUATIONFLATCURLRENORMALIZEDVORTICITY},
which therefore yields
\begin{align}
&
\Transport 
\left\lbrace
	\exp(-\Densrenormalized) (\Flatcurl \Vortrenormalized)^i
	+
	\exp(-3 \Densrenormalized) \Speed^{-2} \frac{p_{;\Ent}}{\bar{\varrho}} \GradEnt^a \partial_a v^i
	-
	\exp(-3 \Densrenormalized) \Speed^{-2} \frac{p_{;\Ent}}{\bar{\varrho}} (\partial_a v^a) \GradEnt^i
\right\rbrace
	\label{E:PROOFFINALEVOLUTIONEQUATIONFLATCURLRENORMALIZEDVORTICITY} 
	\\
& 	=
		- 
		2 \delta_{jk} \epsilon_{iab} \exp(-\Densrenormalized) (\partial_a v^j) \partial_b \Vortrenormalized^k
		+
		\epsilon_{ajk}
		\exp(-\Densrenormalized)
		(\partial_a v^i) 
		\partial_j \Vortrenormalized^k
		\notag
				\\
&  \ \
		+
		\exp(-3 \Densrenormalized) \Speed^{-2} \frac{p_{;\Ent}}{\bar{\varrho}} (\Transport \GradEnt^a) \partial_a v^i
		-
		\exp(-3 \Densrenormalized) \Speed^{-2} \frac{p_{;\Ent}}{\bar{\varrho}}  (\Transport v^i) \partial_a \GradEnt^a
		\notag \\
	& \ \
		+
		\exp(-3 \Densrenormalized) \Speed^{-2} \frac{p_{;\Ent}}{\bar{\varrho}}  (\Transport v^a) \partial_a \GradEnt^i
		- 
		\exp(-3 \Densrenormalized) \Speed^{-2} \frac{p_{;\Ent}}{\bar{\varrho}} (\partial_a v^a) (\Transport \GradEnt^i)
		\notag \\
& \ \
		+ 
		\exp(-3 \Densrenormalized) \Speed^{-2} \frac{p_{;\Ent}}{\bar{\varrho}}(\partial_a v^b) (\partial_b v^a) \GradEnt^i
		-
		\exp(-3 \Densrenormalized) \Speed^{-2} \frac{p_{;\Ent}}{\bar{\varrho}}(\partial_a v^a)^2 \GradEnt^i
			\notag \\
& \ \
		+ 
		\exp(-3 \Densrenormalized) \Speed^{-2} \frac{p_{;\Ent}}{\bar{\varrho}}(\partial_a v^a) \GradEnt^b \partial_b v^i 
		- 
		\exp(-3 \Densrenormalized) \Speed^{-2} \frac{p_{;\Ent}}{\bar{\varrho}}(\GradEnt^a \partial_a v^b) \partial_b v^i 
		\notag \\
	& \ \
		+ 
		2 \exp(-3 \Densrenormalized) \Speed^{-2} \frac{p_{;\Ent}}{\bar{\varrho}} (\GradEnt^a \partial_a \Densrenormalized)  \Transport v^i
		- 
		2 \exp(-3 \Densrenormalized) \Speed^{-2} \frac{p_{;\Ent}}{\bar{\varrho}} (\Transport \Densrenormalized) \GradEnt^a \partial_a v^i
		\notag \\
	&  \ \
			+ 
			2 \exp(-3 \Densrenormalized) \Speed^{-3} \Speed_{;\Densrenormalized}  
			\frac{p_{;\Ent}}{\bar{\varrho}} (\GradEnt^a \partial_a \Densrenormalized)  \Transport v^i
			- 
			2 \exp(-3 \Densrenormalized) \Speed^{-3} \Speed_{;\Densrenormalized} 
			\frac{p_{;\Ent}}{\bar{\varrho}} (\Transport \Densrenormalized) \GradEnt^a \partial_a v^i
		\notag \\
	& \ \
		+ 
		\exp(-3 \Densrenormalized) \Speed^{-2} \frac{p_{;\Ent;\Densrenormalized}}{\bar{\varrho}} (\Transport \Densrenormalized) \GradEnt^a \partial_a v^i
		- 
		\exp(-3 \Densrenormalized) \Speed^{-2} \frac{p_{;\Ent;\Densrenormalized}}{\bar{\varrho}} (\GradEnt^a \partial_a \Densrenormalized)  \Transport v^i
		\notag \\
	& \ \
		+
		\exp(-3 \Densrenormalized) \Speed^{-2} \frac{p_{;\Ent;\Densrenormalized}}{\bar{\varrho}}  (\Transport v^a) (\partial_a \Densrenormalized) \GradEnt^i
		-
		\exp(-3 \Densrenormalized) \Speed^{-2} \frac{p_{;\Ent;\Densrenormalized}}{\bar{\varrho}} (\Transport \Densrenormalized) (\partial_a v^a) \GradEnt^i
		\notag \\
	& \ \
		+
		2 \exp(-3 \Densrenormalized) \Speed^{-2} \frac{p_{;\Ent}}{\bar{\varrho}} (\Transport \Densrenormalized)(\partial_a v^a) \GradEnt^i
		- 
		2 \exp(-3 \Densrenormalized) \Speed^{-2} \frac{p_{;\Ent}}{\bar{\varrho}}  (\Transport v^a) (\partial_a \Densrenormalized) \GradEnt^i
		\notag \\
	& \ \
		 	+
			2 \exp(-3 \Densrenormalized) \Speed^{-3} \Speed_{;\Densrenormalized} \frac{p_{;\Ent}}{\bar{\varrho}} 
			(\Transport \Densrenormalized)(\partial_a v^a) \GradEnt^i
		 	- 
		 	2 \exp(-3 \Densrenormalized) \Speed^{-3} \Speed_{;\Densrenormalized} \frac{p_{;\Ent}}{\bar{\varrho}}  
		 	(\Transport v^a) (\partial_a \Densrenormalized) \GradEnt^i
		 		\notag \\
		& \ \
				+ 
				2 \exp(-3 \Densrenormalized) \Speed^{-3} \Speed_{;\Ent} \frac{p_{;\Ent}}{\bar{\varrho}} 
				(\Transport v^i) \delta_{ab} \GradEnt^a \GradEnt^b
				-
				2 \exp(-3 \Densrenormalized) \Speed^{-3} \Speed_{;\Ent} \frac{p_{;\Ent}}{\bar{\varrho}} 
				\delta_{ab} \GradEnt^a (\Transport v^b) \GradEnt^i
			\notag \\
		& \ \
			+ 
			\exp(-3 \Densrenormalized) \Speed^{-2} \frac{p_{;\Ent;\Ent}}{\bar{\varrho}} (\Transport v^i) \delta_{ab} \GradEnt^a \GradEnt^b
			- 
			\exp(-3 \Densrenormalized) \Speed^{-2} \frac{p_{;\Ent;\Ent}}{\bar{\varrho}} \delta_{ab} (\Transport v^a) \GradEnt^b \GradEnt^i.
			\notag 
\end{align}	
From equation \eqref{E:PROOFFINALEVOLUTIONEQUATIONFLATCURLRENORMALIZEDVORTICITY}
and definition \eqref{E:RENORMALIZEDCURLOFSPECIFICVORTICITY},
we conclude the desired equation \eqref{E:PROOFEVOLUTIONEQUATIONFLATCURLRENORMALIZEDVORTICITY}.


\end{proof}

\subsection{Deriving the wave equations}
\label{SS:MAINTHEOREMWAVEEQUATIONS}
Recall that the covariant wave operator $\square_g$
is defined in Def.~\ref{E:WAVEOPERATORARBITRARYCOORDINATES}.
In the next lemma, we provide an explicit expression for
$\square_g \phi$ that holds relative to the Cartesian coordinates.

\begin{lemma}[$\square_g$ \textbf{relative to the Cartesian coordinates}]
	\label{L:COVARIANTWAVEOPRELATIVETOCARTESIAN}
	Let $g$ be the acoustical metric from Def.~\ref{D:ACOUSTICALMETRIC}.
	The covariant wave operator $\square_g$ acts on scalar functions $\phi$
	via the following identity, where
	RHS~\eqref{E:COVARIANTWVAEOPERATORINRECTANGULARCOORDINATES}
	is expressed in Cartesian coordinates:
	\begin{align} \label{E:COVARIANTWVAEOPERATORINRECTANGULARCOORDINATES}
	\square_g \phi
	& = - 
			\Transport \Transport \phi
		+ 
			\Speed^2 \delta^{ab} \partial_a \partial_b \phi
		+ 
			2 \Speed^{-1} \Speed_{;\Densrenormalized} (\Transport \Densrenormalized) \Transport \phi
		- 
			(\partial_a v^a) \Transport \phi
		- 
			\Speed^{-1} \Speed_{;\Densrenormalized} (g^{-1})^{\alpha \beta} (\partial_{\alpha} \Densrenormalized) \partial_{\beta} \phi
			\\
	& \ \
			- 
			\Speed \Speed_{;\Ent} \GradEnt^a \partial_a \phi
			+ 
			3 \Speed^{-1} \Speed_{;\Ent} (\Transport \Ent) \Transport \phi.
		\notag
\end{align}
\end{lemma}
\begin{proof}
It is straightforward to compute using equations 
\eqref{E:ACOUSTICALMETRIC}-\eqref{E:INVERSEACOUSTICALMETRIC}
that relative to Cartesian coordinates, we have
\begin{align} \label{E:DETG}
	\mbox{\upshape det} g
	& = - \Speed^{-6}
\end{align}
and hence
\begin{align} \label{E:ROOTDETGTIMESGINVERSE}
	\sqrt{|\mbox{\upshape det} g|} g^{-1} 
	& = 
		- \Speed^{-3} \Transport \otimes \Transport
		+
		\Speed^{-1} \sum_{a=1}^3 \partial_a \otimes \partial_a.
\end{align}
Using 
definition \eqref{E:ENTROPYGRADIENT}
and equations
\eqref{E:WAVEOPERATORARBITRARYCOORDINATES},
\eqref{E:DETG}, 
and
\eqref{E:ROOTDETGTIMESGINVERSE}, 
we compute that
\begin{align} \label{E:FIRSTFORMULACOVARIANTWVAEOPERATORINRECTANGULARCOORDINATES}
	\square_g \phi
	 & = 
	- 
	\Speed^3
	\left(
		\Transport^{\alpha} \partial_{\alpha} (\Speed^{-3})
	\right)
	\Transport^{\beta} \partial_{\beta} \phi
	- 
	(\partial_{\alpha} \Transport^{\alpha}) \Transport^{\beta} \partial_{\beta} \phi
	- 
	(\Transport^{\alpha} \partial_{\alpha} \Transport^{\beta}) \partial_{\beta} \phi
		\\
&  \ \
	-
	\Transport^{\alpha} \Transport^{\beta} \partial_{\alpha} \partial_{\beta}
	\phi
	+ 
	\Speed^2 \delta^{ab} \partial_a \partial_b \phi
	- 
	\Speed \Speed_{;\Densrenormalized} \delta^{ab} (\partial_a \Densrenormalized) \partial_b \phi
	- 
	\Speed \Speed_{;\Ent} \GradEnt^a \partial_a \phi.
	\notag
\end{align}
Finally, from \eqref{E:FIRSTFORMULACOVARIANTWVAEOPERATORINRECTANGULARCOORDINATES},
the expression \eqref{E:MATERIALVECTORVIELDRELATIVETORECTANGULAR} for $\Transport$, 
the expression \eqref{E:INVERSEACOUSTICALMETRIC} for $g^{-1}$,
and simple calculations, we arrive at \eqref{E:COVARIANTWVAEOPERATORINRECTANGULARCOORDINATES}.
\end{proof}

In the next lemma, we derive equation \eqref{E:RENORMALIZEDDENSITYWAVEEQUATION}.

\begin{lemma}[\textbf{Wave equation for} $\Densrenormalized$]
\label{L:WAVEEQUATIONFORLOGDENSITY}
The compressible Euler equations 
\eqref{E:TRANSPORTDENSRENORMALIZEDRELATIVETORECTANGULAR}-\eqref{E:ENTROPYTRANSPORT}
imply the following covariant wave equation for the
logarithmic density variable $\Densrenormalized$
from Def.~\ref{D:RESCALEDVARIABLES},
where $\DivofEntrenormalized$ is the modified fluid variable
from Def.~\ref{D:RENORMALIZEDCURLOFSPECIFICVORTICITY}:
\begin{align} \label{E:PROOFRENORMALIZEDDENSITYWAVEEQUATION}
\square_g \Densrenormalized 
& = -
		\exp(\Densrenormalized) \frac{p_{;\Ent}}{\bar{\varrho}} \DivofEntrenormalized
		- 
		3 \Speed^{-1} \Speed_{;\Densrenormalized} 
		(g^{-1})^{\alpha \beta} \partial_{\alpha} \Densrenormalized \partial_{\beta} \Densrenormalized
		+ 
		(\partial_a v^a) (\partial_b v^b)
		-
		(\partial_a v^b) \partial_b v^a
			\\
& \ \
		-
		2 \exp(-\Densrenormalized) \frac{p_{;\Ent;\Densrenormalized}}{\bar{\varrho}} \GradEnt^a \partial_a \Densrenormalized 
		-
		\exp(-\Densrenormalized) \frac{p_{;\Ent;\Ent}}{\bar{\varrho}} \delta_{ab} \GradEnt^a \GradEnt^b.
		\notag
\end{align}
\end{lemma}

\begin{proof}
First, using \eqref{E:COVARIANTWVAEOPERATORINRECTANGULARCOORDINATES}
with $\phi = \Densrenormalized$,
equation \eqref{E:TRANSPORTDENSRENORMALIZEDRELATIVETORECTANGULAR},
and equation \eqref{E:ENTROPYTRANSPORT} 
(which implies that the last product on RHS~\eqref{E:COVARIANTWVAEOPERATORINRECTANGULARCOORDINATES} vanishes),
we compute that
\begin{align} \label{E:FIRSTCOMPUTATIONRENORMALIZEDDENSITYWAVEEQUATION}
	\square_g \Densrenormalized
	& = - \Transport \Transport \Densrenormalized
		+ 
		\Speed^2 \delta^{ab} \partial_a \partial_b \Densrenormalized
		+ 
		2 \Speed^{-1} \Speed_{;\Densrenormalized} (\Transport \Densrenormalized)^2
		+ 
		(\partial_a v^a)^2
		- 
		\Speed^{-1} \Speed_{;\Densrenormalized} 
			(g^{-1})^{\alpha \beta} 
			\partial_{\alpha} \Densrenormalized \partial_{\beta} \Densrenormalized
		- 
		\Speed \Speed_{;\Ent} \GradEnt^a \partial_a \Densrenormalized.
\end{align}
Next, we use
definitions \eqref{E:ENTROPYGRADIENT} and \eqref{E:MATERIALVECTORVIELDRELATIVETORECTANGULAR},
equations
\eqref{E:TRANSPORTDENSRENORMALIZEDRELATIVETORECTANGULAR}-\eqref{E:ENTROPYTRANSPORT},
the chain rule identity
$
\displaystyle
2 \Speed \Speed_{;\Ent} 
= (\Speed^2)_{;\Ent} 
= \left(\frac{1}{\bar{\varrho}} \exp(-\Densrenormalized) p_{;\Densrenormalized} \right)_{;\Ent}
= \exp(-\Densrenormalized) \frac{p_{;\Densrenormalized;\Ent}}{\bar{\varrho}}
$,
and the identity 
$
p_{;\Densrenormalized;\Ent} 
= 
p_{;\Ent;\Densrenormalized}
$
to compute that
\begin{align} \label{E:TWOTRANSPORTAPPLIEDTORENORMALIZEDDENSITYEXPRESSION}
	\Transport \Transport \Densrenormalized
	& = - \partial_a (\Transport v^a)
		+ (\partial_a v^b) \partial_b v^a 
			\\
	& = 
		\Speed^2 \delta^{ab} \partial_a \partial_b \Densrenormalized
		+
		\delta^{ab} (\partial_a \Speed^2)  \partial_b \Densrenormalized
		+
		\partial_a
		\left\lbrace
			\exp(-\Densrenormalized) \frac{p_{;\Ent}}{\bar{\varrho}} \delta^{ab} \partial_b \Ent
		\right\rbrace
		+ 
		(\partial_a v^b) \partial_b v^a
			\notag \\
	& = 
		\Speed^2 \delta^{ab} \partial_a \partial_b \Densrenormalized
		+
		\exp(-\Densrenormalized) \frac{p_{;\Ent}}{\bar{\varrho}} \delta^{ab} \partial_a \partial_b \Ent
		+
		2 \Speed \Speed_{;\Densrenormalized} \delta^{ab} \partial_a \Densrenormalized \partial_b \Densrenormalized
			\notag 
				\\
	& \ \
		+
		2 \Speed \Speed_{;\Ent} \delta^{ab} \partial_a \Densrenormalized \partial_b \Ent
		- 
		\exp(-\Densrenormalized) \frac{p_{;\Ent}}{\bar{\varrho}} \delta^{ab} \partial_a \Densrenormalized \partial_b \Ent
		+
		\exp(-\Densrenormalized) \frac{p_{;\Ent;\Densrenormalized}}{\bar{\varrho}} \delta^{ab} \partial_a \Densrenormalized \partial_b \Ent
			\notag \\
	& \ \
		+
		\exp(-\Densrenormalized) \frac{p_{;\Ent;\Ent}}{\bar{\varrho}} \delta^{ab} \partial_a \Ent \partial_b \Ent
		+
		(\partial_a v^b) \partial_b v^a
		\notag
			\\
	& = 
		\Speed^2 \delta^{ab} \partial_a \partial_b \Densrenormalized
		+
		\exp(-\Densrenormalized) \frac{p_{;\Ent}}{\bar{\varrho}} \Flatdiv \GradEnt
		+
		2 \Speed \Speed_{;\Densrenormalized} \delta^{ab} \partial_a \Densrenormalized \partial_b \Densrenormalized
		+
		(\partial_a v^b) \partial_b v^a
			\notag 
				\\
	& \ \
		+
		2 \exp(-\Densrenormalized) \frac{p_{;\Ent;\Densrenormalized}}{\bar{\varrho}} \GradEnt^a \partial_a \Densrenormalized 
		- 
		\exp(-\Densrenormalized) \frac{p_{;\Ent}}{\bar{\varrho}} \GradEnt^a \partial_a \Densrenormalized
		+
		\exp(-\Densrenormalized) \frac{p_{;\Ent;\Ent}}{\bar{\varrho}} \delta_{ab} \GradEnt^a \GradEnt^b.
		\notag
\end{align}
Finally, using \eqref{E:TWOTRANSPORTAPPLIEDTORENORMALIZEDDENSITYEXPRESSION} 
to substitute for the term
$- \Transport \Transport \Densrenormalized$ on RHS~\eqref{E:FIRSTCOMPUTATIONRENORMALIZEDDENSITYWAVEEQUATION},
using the identity
$
\Transport \Densrenormalized \Transport \Densrenormalized
-
\Speed^2 \delta^{ab} \partial_a \Densrenormalized \partial_b \Densrenormalized
= - (g^{-1})^{\alpha \beta} \partial_{\alpha} \Densrenormalized \partial_{\beta} \Densrenormalized
$
(see \eqref{E:INVERSEACOUSTICALMETRIC}),
and using definition \eqref{E:RENORMALIZEDDIVOFENTROPY}
to algebraically substitute for the factor $\Flatdiv \GradEnt$ on RHS~\eqref{E:TWOTRANSPORTAPPLIEDTORENORMALIZEDDENSITYEXPRESSION},
we arrive at the desired expression 
\eqref{E:PROOFRENORMALIZEDDENSITYWAVEEQUATION}.
\end{proof}

We now establish equation \eqref{E:VELOCITYWAVEEQUATION}.

\begin{lemma}[\textbf{Wave equation for} $v^i$]
\label{L:WAVEEQUATIONFORV}
The compressible Euler equations 
\eqref{E:TRANSPORTDENSRENORMALIZEDRELATIVETORECTANGULAR}-\eqref{E:ENTROPYTRANSPORT}
imply the following covariant wave equation for the
scalar-valued functions $v^i$, $(i=1,2,3)$,
where $\lbrace \CurlofVortrenormalized^i \rbrace_{i=1,2,3}$ are 
the Cartesian components of the modified fluid variable
from Def.~\ref{D:RENORMALIZEDCURLOFSPECIFICVORTICITY}:
\begin{align} \label{E:PROOFVELOCITYWAVEEQUATION}
\square_g v^i
& = 	- 
			\Speed^2 \exp(2\Densrenormalized) \CurlofVortrenormalized^i 
			- 
			\left\lbrace
				1+\Speed^{-1} \Speed'
			\right\rbrace
			(g^{-1})^{\alpha \beta} \partial_{\alpha} \Densrenormalized \partial_{\beta} v^i
				\\
& \ \
		+ 
		2 \exp(\Densrenormalized) \epsilon_{iab} (\Transport v^a) \Vortrenormalized^b
		-
		\frac{p_{;\Ent}}{\bar{\varrho}} \epsilon_{iab} \Vortrenormalized^a \GradEnt^b
		 \notag \\
& \ \
		- 
		\exp(-\Densrenormalized) \frac{p_{;\Ent}}{\bar{\varrho}} \GradEnt^a \partial_a v^i
		- 
		\frac{1}{2} \exp(-\Densrenormalized) \frac{p_{;\Densrenormalized;\Ent}}{\bar{\varrho}} \GradEnt^a \partial_a v^i
		\notag \\
& \ \
		- 
		2 \exp(-\Densrenormalized) \Speed^{-1} \Speed_{;\Densrenormalized} \frac{p_{;\Ent}}{\bar{\varrho}} 
		(\Transport \Densrenormalized) \GradEnt^i
		+
		\exp(-\Densrenormalized) \frac{p_{;\Ent;\Densrenormalized}}{\bar{\varrho}} (\Transport \Densrenormalized) \GradEnt^i.
		\notag
\end{align}
\end{lemma}

\begin{proof}
First, we use equation \eqref{E:COVARIANTWVAEOPERATORINRECTANGULARCOORDINATES} with $\phi = v^i$, 
definition \eqref{E:ENTROPYGRADIENT},
equation \eqref{E:TRANSPORTVELOCITYRELATIVETORECTANGULAR},
and equation \eqref{E:ENTROPYTRANSPORT} 
(which implies that the last product on RHS~\eqref{E:COVARIANTWVAEOPERATORINRECTANGULARCOORDINATES} vanishes),
to compute
\begin{align} \label{E:FIRSTCOMPUTATIONVELOCITYWAVEEQUATION}
	\square_g v^i
	& = - 
			\Transport \Transport v^i
		+ 
		\Speed^2 \delta^{ab} \partial_a \partial_b v^i
		- 
		2 \Speed \Speed_{;\Densrenormalized} (\Transport \Densrenormalized) \delta^{ia} \partial_a \Densrenormalized
		- 
		2 \exp(-\Densrenormalized) \Speed^{-1} \Speed_{;\Densrenormalized} \frac{p_{;\Ent}}{\bar{\varrho}} 
			(\Transport \Densrenormalized) \GradEnt^i
			\\
	& \ \
		- 
		(\partial_a v^a) \Transport v^i
		- 
		\Speed^{-1} \Speed_{;\Densrenormalized} (g^{-1})^{\alpha \beta} (\partial_{\alpha} \Densrenormalized) \partial_{\beta} v^i
		- 
		\Speed \Speed_{;\Ent} \GradEnt^a \partial_a v^i.
		\notag
\end{align}
Next, we use definitions \eqref{E:ENTROPYGRADIENT} and \eqref{E:MATERIALVECTORVIELDRELATIVETORECTANGULAR},
equations \eqref{E:TRANSPORTDENSRENORMALIZEDRELATIVETORECTANGULAR}-\eqref{E:ENTROPYTRANSPORT},
and the already established equation \eqref{E:PROOFGRADENTROPYTRANSPORT} to compute that
\begin{align} \label{E:TWOTRANSPORTAPPLIEDTOVELOCITYEXPRESSION}
\Transport \Transport v^i
	& = - 
			\Speed^2 \delta^{ia} \Transport \partial_a \Densrenormalized
			- 
			2 \Speed \Speed_{;\Densrenormalized} \Transport \Densrenormalized \delta^{ia} \partial_a \Densrenormalized
				\\
	& \ \
			+ 
			\exp(-\Densrenormalized) \frac{p_{;\Ent}}{\bar{\varrho}} (\Transport \Densrenormalized) \GradEnt^i
			- 
			\exp(-\Densrenormalized) \frac{p_{;\Ent;\Densrenormalized}}{\bar{\varrho}} (\Transport \Densrenormalized) \GradEnt^i
			- 
			\exp(-\Densrenormalized) \frac{p_{;\Ent}}{\bar{\varrho}} \Transport \GradEnt^i
			\notag \\
	& = 
		- 
		\Speed^2 \delta^{ia} \partial_a (\Transport \Densrenormalized)
		+ 
		\Speed^2 \delta^{ia} (\partial_a v^b) \partial_b \Densrenormalized
		- 
		2 \Speed \Speed_{;\Densrenormalized} (\Transport \Densrenormalized) \delta^{ia} \partial_a \Densrenormalized
			\notag \\
	& \ \
		+ 
		\exp(-\Densrenormalized) \frac{p_{;\Ent}}{\bar{\varrho}} (\Transport \Densrenormalized) \GradEnt^i
		- 
		\exp(-\Densrenormalized) \frac{p_{;\Ent;\Densrenormalized}}{\bar{\varrho}} (\Transport \Densrenormalized) \GradEnt^i
		+ 
		\exp(-\Densrenormalized) \frac{p_{;\Ent}}{\bar{\varrho}} \GradEnt^a \partial_a v^i
		-
		\epsilon_{iab} \frac{p_{;\Ent}}{\bar{\varrho}} \Vortrenormalized^a \GradEnt^b
			\notag \\
	& = 
			\Speed^2 \delta^{ia} \delta_d^b \partial_a (\partial_b v^d)
		- 
		\delta^{ia} \delta_{bd} (\partial_a v^b) \Transport v^d
		-
		\exp(-\Densrenormalized) \frac{p_{;\Ent}}{\bar{\varrho}}
		\delta^{ij} \delta_{ab} (\partial_j v^a) \GradEnt^b
		- 
		2 \Speed \Speed_{;\Densrenormalized} (\Transport \Densrenormalized) \delta^{ia} \partial_a \Densrenormalized
		\notag \\
	& \ \
		+ 
		\exp(-\Densrenormalized) \frac{p_{;\Ent}}{\bar{\varrho}} (\Transport \Densrenormalized) \GradEnt^i
		- 
		\exp(-\Densrenormalized) \frac{p_{;\Ent;\Densrenormalized}}{\bar{\varrho}} (\Transport \Densrenormalized) \GradEnt^i
		+ 
		\exp(-\Densrenormalized) \frac{p_{;\Ent}}{\bar{\varrho}} \GradEnt^a \partial_a v^i
		-
		\epsilon_{iab} \frac{p_{;\Ent}}{\bar{\varrho}} \Vortrenormalized^a \GradEnt^b
		\notag \\
	& = 
		\Speed^2 \delta^{ab} \partial_a \partial_b v^i
		+ 
		\Speed^2 \delta^{ia} \partial_d (\partial_a v^d - \partial_d v^a)
			\notag \\
	& \ \
		- 
		(\Transport v^a) \partial_a v^i 
		+
		(\Transport v^a)(\partial_a v^i - \partial_i v^a)
		\notag	\\
	& \ \
		+
		\exp(-\Densrenormalized) \frac{p_{;\Ent}}{\bar{\varrho}} \GradEnt^a (\partial_a v^i - \partial_i v^a)
		- 
		2 \Speed \Speed_{;\Densrenormalized} (\Transport \Densrenormalized) \delta^{ia} \partial_a \Densrenormalized
			\notag \\
	& \ \
		+ 
		\exp(-\Densrenormalized) \frac{p_{;\Ent}}{\bar{\varrho}} (\Transport \Densrenormalized) \GradEnt^i
		- 
		\exp(-\Densrenormalized) \frac{p_{;\Ent;\Densrenormalized}}{\bar{\varrho}} (\Transport \Densrenormalized) \GradEnt^i
		+ 
		\exp(-\Densrenormalized) \frac{p_{;\Ent}}{\bar{\varrho}} \GradEnt^a \partial_a v^i
		-
		\frac{p_{;\Ent}}{\bar{\varrho}} \epsilon_{iab}  \Vortrenormalized^a \GradEnt^b.
			\notag
\end{align}
Next, we use 
definitions \eqref{E:ENTROPYGRADIENT} and \eqref{E:VORTICITYDEFINITION},
the identity \eqref{E:VORTICIITYRENORMALIZEDVORTICITYANDDENSITYRELATION},
equation \eqref{E:TRANSPORTVELOCITYRELATIVETORECTANGULAR}, 
and the antisymmetry of $\epsilon_{\cdots}$
to derive the identities
\begin{align}
\Speed^2 \delta^{ia} \partial_d (\partial_a v^d - \partial_d v^a)
	& = \Speed^2 \epsilon_{iab} \partial_a \omega^b
	= \Speed^2 \Flatcurl \omega^i
		\label{E:FIRSTIDUSEDINDERIVINGVELOCITYWAVEEQN}
		\\
	& = \Speed^2 \exp(\Densrenormalized)  \Flatcurl \Vortrenormalized^i
		+
		\Speed^2 \exp(\Densrenormalized) 
		\epsilon_{iab} \Vortrenormalized^b \partial_a \Densrenormalized
	\notag \\
	& =
		\Speed^2 \exp(\Densrenormalized)  \Flatcurl \Vortrenormalized^i 
			-
			\exp(\Densrenormalized) \epsilon_{iab} (\Transport v^a) \Vortrenormalized^b
			+
			\frac{p_{;\Ent}}{\bar{\varrho}} \epsilon_{iab} \Vortrenormalized^a \GradEnt^b,
			\notag \\
	\exp(-\Densrenormalized) \frac{p_{;\Ent}}{\bar{\varrho}}
	\GradEnt^a (\partial_a v^i - \partial_i v^a)
	& = 
	\frac{p_{;\Ent}}{\bar{\varrho}} \epsilon_{bai} \GradEnt^a \Vortrenormalized^b
		= 
		\frac{p_{;\Ent}}{\bar{\varrho}} \epsilon_{iab} \Vortrenormalized^a \GradEnt^b,
		\label{E:SIMPLECURLID} 
		\\
(\Transport v^a)(\partial_a v^i - \partial_i v^a)
 & 
=  \exp(\Densrenormalized) \epsilon_{bai} (\Transport v^a) \Vortrenormalized^b 
= - \exp(\Densrenormalized) \epsilon_{iab} (\Transport v^a) \Vortrenormalized^b.
\label{E:SECONDIDUSEDINDERIVINGVELOCITYWAVEEQN}
\end{align}
Substituting the RHSs of 
\eqref{E:FIRSTIDUSEDINDERIVINGVELOCITYWAVEEQN}-\eqref{E:SECONDIDUSEDINDERIVINGVELOCITYWAVEEQN}
for the relevant terms on RHS~\eqref{E:TWOTRANSPORTAPPLIEDTOVELOCITYEXPRESSION},
we obtain
\begin{align} \label{E:TWOTRANSPORTAPPLIEDTOVELOCITYSECONDEXPRESSION}
	\Transport \Transport v^i
	& = \Speed^2 \delta^{ab} \partial_a \partial_b v^i
			+ 
			\Speed^2 \exp(\Densrenormalized) \Flatcurl \Vortrenormalized^i 
			-
			2 \exp(\Densrenormalized) \epsilon_{iab} (\Transport v^a) \Vortrenormalized^b
			+
			\frac{p_{;\Ent}}{\bar{\varrho}} \epsilon_{iab} \Vortrenormalized^a \GradEnt^b
			\\
	& \ \
		- (\Transport v^a) \partial_a v^i
		- 2 \Speed \Speed_{;\Densrenormalized} (\Transport \Densrenormalized) \delta^{ia} \partial_a \Densrenormalized
		\notag
			\\
	& \ \
		+ 
		\exp(-\Densrenormalized) \frac{p_{;\Ent}}{\bar{\varrho}} \GradEnt^a \partial_a v^i
		+ 
		\exp(-\Densrenormalized) \frac{p_{;\Ent}}{\bar{\varrho}} (\Transport \Densrenormalized) \GradEnt^i
		- 
		\exp(-\Densrenormalized) \frac{p_{;\Ent;\Densrenormalized}}{\bar{\varrho}} (\Transport \Densrenormalized) \GradEnt^i.
		\notag
\end{align}
Next, substituting $-$RHS~\eqref{E:TWOTRANSPORTAPPLIEDTOVELOCITYSECONDEXPRESSION}
for the term
$- \Transport \Transport v^i$ 
on RHS~\eqref{E:FIRSTCOMPUTATIONVELOCITYWAVEEQUATION},
and using the chain rule identity
$
\displaystyle
\Speed \Speed_{;\Ent} 
= \frac{1}{2} (\Speed^2)_{;\Ent} 
= \frac{1}{2} \left(\frac{1}{\bar{\varrho}} \exp(-\Densrenormalized) p_{;\Densrenormalized} \right)_{;\Ent}
= \frac{1}{2} \exp(-\Densrenormalized) \frac{p_{;\Densrenormalized;\Ent}}{\bar{\varrho}}
$,
we compute that
\begin{align}\label{first.thm.prelim.1}
\square_g v^i
& = - \Speed^2 \exp(\Densrenormalized)  (\Flatcurl \Vortrenormalized)^i 
			+
			2 \exp(\Densrenormalized) \epsilon_{iab} (\Transport v^a) \Vortrenormalized^b
			-
			\frac{p_{;\Ent}}{\bar{\varrho}} \epsilon_{iab} \Vortrenormalized^a \GradEnt^b
				\\
& \ \
			+
			\left\lbrace
				(\Transport v^a) \partial_a v^i
					-
				(\partial_a v^a) \Transport v^i
			\right\rbrace
			- \Speed^{-1} \Speed_{;\Densrenormalized} (g^{-1})^{\alpha \beta} \partial_{\alpha} \Densrenormalized \partial_{\beta} v^i
				\notag
				\\
& \ \
		- 
		\exp(-\Densrenormalized) \frac{p_{;\Ent}}{\bar{\varrho}} \GradEnt^a \partial_a v^i
		- 
		\frac{1}{2} \exp(-\Densrenormalized) \frac{p_{;\Densrenormalized;\Ent}}{\bar{\varrho}} \GradEnt^a \partial_a v^i
		\notag \\
& \ \
		- 
		2 \exp(-\Densrenormalized) \Speed^{-1} \Speed_{;\Densrenormalized} \frac{p_{;\Ent}}{\bar{\varrho}} 
		(\Transport \Densrenormalized) \GradEnt^i
		- 
		\exp(-\Densrenormalized) \frac{p_{;\Ent}}{\bar{\varrho}} (\Transport \Densrenormalized) \GradEnt^i
		+ 
		\exp(-\Densrenormalized) \frac{p_{;\Ent;\Densrenormalized}}{\bar{\varrho}} (\Transport \Densrenormalized) \GradEnt^i.
		\notag
\end{align}
To handle the terms $\lbrace \cdot \rbrace$ in \eqref{first.thm.prelim.1}, we use 
definitions \eqref{E:ENTROPYGRADIENT} and \eqref{E:INVERSEACOUSTICALMETRIC} 
and equations \eqref{E:TRANSPORTDENSRENORMALIZEDRELATIVETORECTANGULAR}-\eqref{E:TRANSPORTVELOCITYRELATIVETORECTANGULAR}
to obtain
\begin{align}\label{first.thm.prelim.2}
(\Transport v^a) \partial_a v^i
-
(\partial_a v^a) \Transport v^i
& =
	- 
	\Speed^2 \delta^{ab} (\partial_b \Densrenormalized) \partial_a v^i
	- 
	\exp(-\Densrenormalized) \frac{p_{;\Ent}}{\bar{\varrho}} \GradEnt^a \partial_a v^i
	+
	(\Transport \Densrenormalized) \Transport v^i
	\\
& = - \exp(-\Densrenormalized) \frac{p_{;\Ent}}{\bar{\varrho}} \GradEnt^a \partial_a v^i
		-(g^{-1})^{\alpha \beta} (\partial_{\alpha} \Densrenormalized) \partial_{\beta} v^i.
	\notag 
\end{align}
Next, substituting \eqref{first.thm.prelim.2} into \eqref{first.thm.prelim.1}, we 
deduce that
\begin{align} \label{E:ALMOSTPROOFVELOCITYWAVEEQUATION}
\square_g v^i
& = - \Speed^2 \exp(\Densrenormalized) (\Flatcurl \Vortrenormalized)^i 
			- 
			\left\lbrace
				1+\Speed^{-1} \Speed_{;\Densrenormalized}
			\right\rbrace
			(g^{-1})^{\alpha \beta} \partial_{\alpha} \Densrenormalized \partial_{\beta} v^i
				\\
& \ \
		+ 
		2 \exp(\Densrenormalized) \epsilon_{iab} (\Transport v^a) \Vortrenormalized^b
		-
		\frac{p_{;\Ent}}{\bar{\varrho}} \epsilon_{iab} \Vortrenormalized^a \GradEnt^b
		 \notag \\
& \ \
		- 
		2 \exp(-\Densrenormalized) \frac{p_{;\Ent}}{\bar{\varrho}} \GradEnt^a \partial_a v^i
		- 
		\frac{1}{2} \exp(-\Densrenormalized) \frac{p_{;\Densrenormalized;\Ent}}{\bar{\varrho}} \GradEnt^a \partial_a v^i
		\notag \\
& \ \
		- 
		2 \exp(-\Densrenormalized) \Speed^{-1} \Speed_{;\Densrenormalized} \frac{p_{;\Ent}}{\bar{\varrho}} 
		(\Transport \Densrenormalized) \GradEnt^i
		- 
		\exp(-\Densrenormalized) \frac{p_{;\Ent}}{\bar{\varrho}} (\Transport \Densrenormalized) \GradEnt^i
		+ 
		\exp(-\Densrenormalized) \frac{p_{;\Ent;\Densrenormalized}}{\bar{\varrho}} (\Transport \Densrenormalized) \GradEnt^i.
		\notag
\end{align}
Finally, we use equation \eqref{E:RENORMALIZEDCURLOFSPECIFICVORTICITY}
to algebraically substitute for the term
$(\Flatcurl \Vortrenormalized)^i$
on RHS~\eqref{E:ALMOSTPROOFVELOCITYWAVEEQUATION}
and equation \eqref{E:TRANSPORTDENSRENORMALIZEDRELATIVETORECTANGULAR}
to replace the term 
$
\displaystyle
- \exp(-3\Densrenormalized) \Speed^{-2} \frac{p_{;\Ent}}{\bar{\varrho}} (\partial_a v^a) \GradEnt^i
$ 
on RHS~\eqref{E:RENORMALIZEDCURLOFSPECIFICVORTICITY} with 
$
\displaystyle
\exp(-3\Densrenormalized) \Speed^{-2} \frac{p_{;\Ent}}{\bar{\varrho}} (\Transport \Densrenormalized) \GradEnt^i
$, 
which in total yields the desired equation \eqref{E:PROOFVELOCITYWAVEEQUATION}.

\end{proof}

\bibliographystyle{amsalpha}
\bibliography{JBib}

\end{document}